\documentclass[a4paper, 11pt]{article}
\usepackage[utf8]{inputenc}
\usepackage[OT1]{fontenc}
\usepackage{lmodern}     
\usepackage[english]{babel}
\usepackage{cite, hyphenat, hyperref}
\usepackage[normalem]{ulem}
\usepackage[usenames]{color}
\usepackage{comment}

\usepackage{amscd, amsfonts, amsmath}
\usepackage{amssymb, amsthm}
\usepackage{bbm, mathrsfs, mathtools}
\usepackage{enumerate}
\usepackage{makeidx}
\usepackage{calc}
\usepackage[shortlabels]{enumitem}
\usepackage{nccfoots}
\usepackage{appendix}
\usepackage[a4paper, top=2.5cm, bottom=2.5cm, left=3cm, right=3cm]{geometry}
\allowdisplaybreaks
\numberwithin{equation}{section}
\theoremstyle{plain}

\newtheorem{Lemma}{Lemma}[section]
\newtheorem{Proposition}[Lemma]{Proposition}
\newtheorem{Theorem}[Lemma]{Theorem}
\newtheorem{Corollary}[Lemma]{Corollary}

\theoremstyle{definition}
\newtheorem{Definition}[Lemma]{Definition}
\newtheorem{Example}[Lemma]{Example}

\newtheorem{Remark}[Lemma]{Remark}

\def\oe{\"{o}}
\def\a{\`{a}}
\def\quoteleft{`}
\def\B{{\rm B}}

\newcommand{\R}{\mathbb{R}}
\newcommand{\N}{\mathbb{N}}
\DeclareMathOperator{\esssup}{ess\,sup}
\DeclareMathOperator{\essinf}{ess\,inf}
\begin{document}
%-------------------------------------------------------------------
% Title & Abstract
%-------------------------------------------------------------------
\title{Mild to classical solutions for XVA equations under stochastic volatility}

\author{Damiano Brigo\footnote{Dep.~of Mathematics, Imperial College London, United Kingdom. {\tt damiano.brigo@imperial.ac.uk}} \and 
Federico Graceffa\setcounter{footnote}{6}\footnote{Dep.~of Mathematics,
Imperial College London, United Kingdom. {\tt federico.graceffa@gmail.com}} \and 
Alexander Kalinin\setcounter{footnote}{3}
\footnote{Dep.~of Mathematics, LMU Munich, Germany. {\tt alex.kalinin@mail.de}. The third author gratefully acknowledges support from Imperial College London through a former Chapman fellowship.}}

\date{\today}
\maketitle

\begin{abstract}
We extend the valuation of contingent claims in presence of default, collateral and funding to a random functional setting and characterise pre-default value processes by martingales. Pre-default value semimartingales can also be described by BSDEs with random path-dependent coefficients and martingales as drivers. En route, we generalise previous settings by relaxing conditions on the available market information, allowing for an arbitrary default-free filtration and constructing a broad class of default times. Moreover, under stochastic volatility, we characterise pre-default value processes via mild solutions to parabolic semilinear PDEs and give sufficient conditions for mild solutions to exist uniquely and to be classical.
\end{abstract}

\\
{\bf Keywords:} XVA, valuation, collateral, funding costs, default time, stochastic volatility, stochastic differential equation, mild solution, semilinear parabolic PDE.

\section{Introduction}\label{se:intro}

The aim of this paper is to address the valuation of contingent claims in a financial market under default risk, collateralisation and funding costs and benefits. Based on a general probabilistic setting, we develop a market model from previous works that consists of an investor and a counterparty entering a derivative contract. To evaluate such an agreement with default-free information only, we derive a nonlinear pre-default valuation equation and characterise its solutions, the pre-default value processes. 

By focusing on a stochastic volatility model for the underlying risky asset and its generalised variance, or simply quasi variance, we will reach a parabolic semilinear partial differential equation (PDE) that establishes a direct relation between pre-default value processes and mild solutions. While pursuing this goal, we will achieve further extensions of preceding papers in this area, and the two articles~\cite{BriFraPal16} and~\cite{BriFraPal19} in particular, that focused on viscosity and classical solutions. Our main contributions to the existing literature can be described as follows: 
\begin{enumerate}[(1)]
\item \emph{The available market information may fail to provide any knowledge about the first time of default}. In the earlier work~\cite{BriFraPal19} even full insight into the separate default times of the investor and the counterparty was required, as the two latter filtrations in~\eqref{eq:market condition 1} were supposed to be equal.

\item To handle the general relation~\eqref{eq:market condition 1} between the two filtrations that model the default-free and the available market information, \emph{a variety of representations for conditional expectations} is derived in Section~\ref{se:2.2}. Thereby, Corollary~\ref{co:identification} explains how to identify random quantities before the first time of default occurs and \emph{the default-free filtration is arbitrary}. In particular, it does not need to coincide with the (augmented) natural filtration of a diffusion. 
 
\item \emph{The default times of the two parties, except being conditionally independent and admitting a distribution satisfying weak regularity conditions, are arbitrary}. This is based on an explicit construction in Section~\ref{se:2.3}, which allows for \emph{a detailed analysis of default times}, including a formula for their survival functions in Proposition~\ref{pr:density formula}. Hitting times that involve a gamma distribution, or more specifically, an exponential distribution, as considered for example in~\cite{BriFraPal19}, are feasible, as shown in Example~\ref{ex:gamma distributed hitting times}, and we refer to~\cite{CreSon16} for a discussion on related issues on immersion.

\item \emph{The pre-default valuation equation~\eqref{eq:pre-default valuation} that only requires default-free information is deduced from the generalised valuation equation~\eqref{eq:valuation}} in Proposition~\ref{pr:pre-default valuation}. To this end, for all cash flows, costs and benefits appearing in the valuation, we give concise financial interpretations and state the necessary measurability, path regularity and integrability conditions in Sections~\ref{se:3.2} and~\ref{se:3.3}.

\item We give \emph{two characterisations for pre-default value processes}, the solutions to~\eqref{eq:pre-default valuation}. While Proposition~\ref{pr:martingale characterisation} relates pre-default valuation with the martingale property of the process in~\eqref{eq:pre-default martingale}, Corollary~\ref{co:semimartingale characterisation} describes value processes that are semimartingales by the BSDE~\eqref{eq:backward stochastic representation} with random path-dependent coefficients, driven by a martingale and analysed in Proposition~\ref{pr:backward stochastic representation}. In the previous work~\cite{BriPal14}, for instance, necessary and sufficient conditions for the existence of solutions to the pre-default valuation equation were not explicitly given.

\item \emph{A stochastic volatility model}, described by the two-dimensional SDE~\eqref{eq:stochastic volatility model}, is introduced in Section~\ref{se:4.1}. Regarding the quasi variance process, we give \emph{a criterion for solutions to one-dimensional SDEs to have a.s.~positive paths} in Proposition~\ref{pr:positivity of the variance process}, by extending the main result in~\cite{MisPos08}, and demonstrate in Example~\ref{ex:processes with positive paths} that sums of power functions such as~\eqref{eq:sums of power functions} may appear as drift and diffusion coefficients. Combined with the uniqueness and existence results from~\cite{KalMeyPro21}, Proposition~\ref{pr:transformed SDE} proves that the transformed SDE~\eqref{eq:transformed stochastic volatility model}, obtained by taking the log-price process, is uniquely solvable and yields a diffusion. Then, Example~\ref{ex:stochastic volatility model} applies this result to the specific SDE~\eqref{eq:specific model}, which extends the Heston model~\cite{Hes93} and the Garch diffusion model~\cite{Lew00}.

\item By imposing the dynamics~\eqref{eq:stochastic volatility model} on the price process and its quasi variance in the market model from Section~\ref{se:3}, we eventually reach the parabolic semilinear PDE~\eqref{eq:parabolic equation}. One of the main achievements of the article is that \emph{we characterise pre-default value processes by means of mild solutions} to~\eqref{eq:parabolic equation} in Theorem~\ref{th:pre-default value process}. As the derived diffusion serves as Makov process in the setting of~\cite{Kal20}, we obtain unique bounded mild solutions in Proposition~\ref{pr:mild solution} and, under the conditions of Corollary~\ref{co:classical solution}, mild solutions are in fact classical.
\end{enumerate}

To the best of our knowledge, our paper is the first work on nonlinear valuation and XVA equations to propose mild solutions as middle ground between viscosity and classical solutions for parabolic semilinear valuation PDEs, including also stochastic volatility.

While viscosity solutions can be described by means of test functions to bypass a priori considerations regarding differentiability, mild solutions stem from related implicit integral equations and allow for Picard iterations, as Proposition~2.12 in~\cite{Kal20} shows, for example. Further, the mild solution concept leads to general derivative formulas, deduced in~\cite{ParPen92}, using Lemma~1.1, Corollary~2.8 and Theorems~2.9 and~3.2 therein. In this sense mild solutions are more tractable than those of viscosity type. If a parabolic semilinear PDE, which may also be path-dependent, admits continuous coefficients, then, under certain linear and polynomial growth conditions and a Lipschitz condition, Corollary 4.17 in~\cite{CosFedGozRosTou18} asserts that the two notions coincide. The valuation PDE that we derive, however, does not meet these regularity conditions and, according to the characterisation that we find in Theorem~\ref{th:pre-default value process}, the mild solution concept is indeed suitable.

Note that the mathematical and modelling achievements of this paper are not obtained in the most general setting in terms of financial adjustments, as we do not include capital valuation adjustments and initial margins in our analysis. However, we believe that the default, collateral as variation margin and funding effects we are considering are more than sufficient to highlight the mathematical difficulties of these nonlinear valuation problems. For this purpose, we would like to contextualize this work in the broad area of nonlinear valuation and valuation adjustments, or \quoteleft XVA'. 

Prior to the financial crisis of 2007-2008, financial institutions at times ignored the credit risk
of highly-rated counterparties in valuing and hedging contingent claims. Then, in a short period of about one month, around October 2008, eight mainstream financial institutions defaulted
(Fannie Mae, Freddie Mac, Lehman Brothers, Washington Mutual, Landsbanki, Glitnir and Kaupthing, to which we could also add Merrill Lynch that was saved through a merge with Bank of America).
This highlighted dramatically the fact that no institution could be considered default-free, no matter how systemic or prestigious. This forced dealers and financial institutions to
reassess the valuation of contingent claims, leading to a much more widespread adoption of collateralisation, through various adjustments to their book value.

We will now list some of these adjustments as separate effects, but one should keep in mind that the nonlinearity of the valuation equations makes this separation quite artificial. 
In any case it is difficult to do justice to the entire literature on such valuation adjustments. For a full introduction to credit and funding valuation adjustments and all related references we refer to the first chapter of either~\cite{BriMorPal13} or~\cite{CreBie14}. Here we will only provide a quick summary for context and a few references, before moving to the full nonlinear valuation equation and its analysis.

Firstly, the credit valuation adjustment (CVA) has been introduced to correct the value of a trade with the expected costs borne by one dealer in scenarios where its counterparty defaults. CVA had been around for some time, see for example~\cite{BriMas05}, and its most sophisticated version can include credit migration and ratings transition, as shown in~\cite{BieCiaIyi13}. Further, it already leads to BSDEs under replacement closeout, which was taken into account in~\cite{CreSon15} and~\cite{BriPal14}. It is worth pointing out that collateralisation has not completely eliminated CVA. In~\cite{BriCapPal14}, for instance, it is shown that for some particular deals gap risk may leave a quite large CVA even in presence of daily collateralisation. This is one of the reasons for the introduction of the initial margin as a further collateralisation tool supplementing the variation margin.

Secondly, the debit valuations adjustment (DVA), that on one hand 
is simply CVA seen from the other side, corrects the price with the expected benefits to the dealer due to scenarios where the dealer has an early default on the trade. DVA may lead to a controversial profit that can be booked when the credit quality of the dealer deteriorates, which has led a discussion on considering it more of a funding benefit than a debit adjustment. While the Basel Committee has made recommendations against the use of DVA, accounting standards by the FASB accept DVA for fair value. A detailed discussion can be found in~\cite{BriMorPal13}. On top of this, DVA is very difficult to hedge, as this would involve selling protection on oneself, and this is a further reason why regulators opposed it.

After CVA and DVA, the funding valuation adjustment (FVA) was introduced. FVA is the price adjustment due to the cost of funding the trading activity surrounding a trade. To maintain a trade, the trading desk needs to borrow funds from the bank treasury, giving back funds occasionally. All borrowing and lending has a cost or remuneration in terms of interest fees, and this has to be accounted for. Following FVA, a capital valuation adjustment (KVA) has started being discussed for the cost of capital one has to set aside in order to be able to trade. We will not address KVA here, since its very definition is currently subject to intense debate in the industry. Instead, we refer to~\cite{CreSabSon20} for a recent work addressing the cost of capital. A further adjustment that has been considered is a charge for the cost of setting up the initial margin for a trade. This is often called margin valuation adjustment, or MVA, and was assessed for example in~\cite{BriPal14} and more recently in~\cite{BiaGnoOli21}, where multiple curve effects are also discussed.  

All such adjustments may concern both over the counter (OTC) derivatives trades and derivatives trades done through central clearing houses (CCP). These two cases are compared in~\cite{BriPal14}, where the full mathematical structure of the problem of valuation under possibly asymmetric initial and variation margins, funding costs, liquidation delay and credit gap risk is explored. This nonlinear valuation analysis has been made more rigorous in the subsequent paper~\cite{BriFraPal19} and by many other authors.

For an early example of how asymmetric interest rates, even in absence of credit risk, lead to BSDEs see~\cite{ElKPenQuen17}. The paper~\cite{BicCapStu18} deals with the mathematical analysis of valuation equations in presence of all the above-mentioned effects and risks, except KVA. CVA and FVA are analysed in~\cite{BifBlaPitSun16} in the area of life insurance contracts, and longevity swaps in particular.
Finally, an in-depth discussion of replication in presence of default and funding effects is presented in~\cite{BriBueFraPalRut18}, discussing also valuation in general settings when replication is not assumed. This article contributes to the literature on nonlinear valuation equations, from BSDEs to PDEs, of the type seen in the above-mentioned works, especially~\cite{CreSon15},~\cite{BriPal14},~\cite{BicCapStu18} and~\cite{BriFraPal19}, by focusing on mild solutions among the other contributions listed earlier.
\smallskip

The paper is structured as follows.  Section~\ref{se:2} sets up the notation and discusses the required probabilistic methods to handle the financial market model. Namely, after a concise introduction of the notation in Section~\ref{se:2.1}, we deal with conditional expectations in Section~\ref{se:2.2} and provide a class of default times in Section~\ref{se:2.3}.

In Section~\ref{se:3} we specify and analyse the market model. While Section~\ref{se:3.1} explains the setting, all the cash flows, costs and benefits that are relevant to determine the price of the derivative contract are quantified in Section~\ref{se:3.2}. Then, in Section~\ref{se:3.3} the pre-default valuation equation~\eqref{eq:pre-default valuation} is derived and its solutions, the pre-default value processes, are characterised in Proposition~\ref{pr:martingale characterisation} and Corollary~\ref{co:semimartingale characterisation}.

In Section~\ref{se:4} we impose a general stochastic volatility model on the underlying risky asset and its quasi variance to deduce the pre-default valuation PDE~\eqref{eq:parabolic equation}. To this end, Section~\ref{se:4.1} considers the SDE~\eqref{eq:stochastic volatility model} that governs the volatility model with regard to pathwise uniqueness, strong existence, moment estimates and positivity of paths. As a result, Proposition~\ref{pr:transformed SDE} shows that the transformed SDE~\eqref{eq:transformed stochastic volatility model} yields a diffusion.

Section~\ref{se:4.2} discusses a deterministic setting of the market model for the valuation PDE to prevail, and pre-default value processes are characterised by means of mild solutions in Theorem~\ref{th:pre-default value process} there. An existence and uniqueness result for bounded mild solutions is derived in Proposition~\ref{pr:mild solution} and sufficient conditions for mild solutions to be classical are given in Corollary~\ref{co:classical solution}.

All proofs for the probabilistic methods in Section~\ref{se:2.2}, the constructed hitting times in Section~\ref{se:2.3} and the market model of Section~\ref{se:3} are deferred to Section~\ref{se:5}. The results for the volatility model and the valuation PDE in Section~\ref{se:4} are proven in Section~\ref{se:6}.

\section{Preliminaries}\label{se:2}

Throughout the paper, let $(\Omega,\mathscr{F},P)$ denote a probability space, $T > 0$ and $(\mathscr{F}_{t})_{t\in [0,T]}$, $(\tilde{\mathscr{F}}_{t})_{t\in [0,T]}$ be two filtrations of $\mathscr{F}$.

\subsection{Notation and basic concepts}\label{se:2.1}

We recall that the extended non-negative real line $[0,\infty]$ is completely metrizable in such a way that the resulting trace topology of $\R_{+}$ agrees with the topology on $\R_{+}$ induced by the absolute value function. For instance, take the metric given by
\begin{equation*}
d_{\infty}(x,y) = |f_{\infty}(x) - f_{\infty}(y)|
\end{equation*}
for any $x,y\in[0,\infty]$ with the strictly increasing homeomorphism $f_{\infty}:\R_{+}\rightarrow [0,1[$ given by $f_{\infty}(x) := x/(1 + x)$ that satisfies $f_{\infty}(\infty) = 1$, where we set $f(\infty) := \lim_{x\uparrow\infty} f(x)$ for any real-valued monotone function $f$ defined on some interval. We shall use the induced topology of $d_{\infty}$ in Sections~\ref{se:2.2} and~\ref{se:2.3}.

For $p\in [1,\infty[$ let $\mathscr{L}^{p}(\R)$ denote the linear space of all real-valued Borel measurable $p$-fold Lebesgue integrable functions on $[0,T]$ and $\mathscr{L}^{p}(\R_{+})$ stand for the convex cone of all $\R_{+}$-valued functions in $\mathscr{L}^{p}(\R)$. For the Banach space of all real-valued c\a dl\a g functions on $[0,T]$, endowed with the supremum norm, we use the standard notation $D([0,T])$.

Let $\mathscr{S}$ and $\tilde{\mathscr{S}}$ be the linear spaces of all (real-valued) processes that are adapted to $(\mathscr{F}_{t})_{t\in [0,T]}$ and $(\tilde{\mathscr{F}}_{t})_{t\in [0,T]}$, respectively, which will be used extensively in Section~\ref{se:3}. Further, a real-valued function $u$ on $[0,T]\times\R\times ]0,\infty[$ will be called right-continuous if for each $(s,x,v)\in [0,T]\times\R\times ]0,\infty[$ and any $\varepsilon > 0$ there is $\delta > 0$ such that
\begin{equation*}
|u(s,x,v) - u(t,y,w)| < \varepsilon
\end{equation*}
for all $(t,y,w)\in [s,T]\times\R\times ]0,\infty[$ with $|s-t| + |x-y| + |v-w| < \delta$. This notion of right-continuity in time and continuity in space from Definition~2.1 in~\cite{Kal20} will be used for the right-hand Feller property of a diffusion in Section~\ref{se:4}.

\subsection{Representations of conditional expectations}\label{se:2.2}

In this section let $\mathscr{T}$ be a non-empty finite set of $[0,T]\cup\{\infty\}$-valued random variables. Each $\tau\in\mathscr{T}$ defines the smallest filtration $(\mathscr{H}_{t}^{\tau})_{t\in [0,T]}$ under which it becomes a stopping time. Namely,
\begin{equation}\label{eq:stopping time sigma-field}
\mathscr{H}_{t}^{\tau}=\sigma\big(\mathbbm{1}_{\{\tau \leq s\}}:s\in [0,t]\big)\quad\text{for all $t\in [0,T]$.}
\end{equation}
By setting $\mathscr{H}_{t}:= \bigvee_{\tau\in\mathscr{T}}\mathscr{H}_{t}^{\tau}$ for any $t\in [0,T]$, we obtain the smallest filtration under which any $\tau\in\mathscr{T}$ is a stopping time. Then the $(\mathscr{H}_{t})_{t\in [0,T]}$-stopping time $\rho:=\min_{\tau\in\mathscr{T}}\tau$ gives rise to the filtration $(\mathscr{F}_{t}^{\mathscr{T}})_{t\in [0,T]}$ defined via
\begin{equation*}
\mathscr{F}_{t}^{\mathscr{T}}:=\big\{\tilde{A}\in\mathscr{F}\,|\,\exists A\in\mathscr{F}_{t}: \{\rho > t\}\cap A = \{\rho > t\}\cap\tilde{A}\big\},
\end{equation*}
which satisfies $\mathscr{F}_{t}\vee\mathscr{H}_{t}\subset\mathscr{F}_{t}^{\mathscr{T}}$ for any $t\in [0,T]$. These concepts generalise the framework in~\cite{BieRut02}[Section 5.1.1] and yield an essential relation between conditional expectations, given another filtration $(\tilde{\mathscr{F}}_{t})_{t\in [0,T]}$ such that $\mathscr{F}_{t}\subset\tilde{\mathscr{F}}_{t}\subset\mathscr{F}_{t}\vee\mathscr{H}_{t}$ for all $t\in [0,T]$.

\begin{Lemma}\label{le:essential identity}
Any $[0,\infty]$-valued random variable $X$ satisfies 
\begin{equation*}
E[X\mathbbm{1}_{\{\rho > t\}}|\tilde{\mathscr{F}}_{s}] P(\rho > s|\mathscr{F}_{s}) =  E[X\mathbbm{1}_{\{\rho > t\}}|\mathscr{F}_{s}]P(\rho > s|\tilde{\mathscr{F}}_{s})\quad\text{a.s.}
\end{equation*}
for all $s,t\in [0,T]$ with $s\leq t$.
\end{Lemma}

We notice that any decreasing sequence $(A_{t})_{t\in [0,T]}$ in $\mathscr{F}$ satisfies $P(A_{s}|\mathscr{F}_{s}) \geq P(A_{t}|\mathscr{F}_{s})$ $= E[P(A_{t}|\mathscr{F}_{t})|\mathscr{F}_{s}]$ a.s.~for all $s,t\in [0,T]$ with $s\leq t$. In particular, for every random variable $\tau$ with values in $[0,T]\cup\{\infty\}$ we have
\begin{equation}\label{eq:survival process}
P(\tau > t|\mathscr{F}_{t}) = G_{t}(\tau)\quad\text{a.s.}\quad\text{for any $t\in [0,T]$}
\end{equation}
and some $[0,1]$-valued $(\mathscr{F}_{t})_{t\in [0,T]}$-supermartingale $G(\tau)$, which is called an \emph{survival process} of $\tau$ relative to this filtration, unique up to a modification. This fact allows us to identify random variables before $\rho$ occurs. 

\begin{Corollary}\label{co:identification}
For $t\in [0,T]$ let $X$ and $\tilde{X}$ be two $\R_{+}$-valued random variables that are measurable relative to $\mathscr{F}_{t}$ and $\tilde{\mathscr{F}}_{t}$, respectively. Then $X = \tilde{X}$ a.s.~on $\{\rho > t\}$ if and only if
\begin{equation}\label{eq:conditional relation}
XG_{t}(\rho) = E[\tilde{X}\mathbbm{1}_{\{\rho > t\}}|\mathscr{F}_{t}]\quad\text{a.s.}
\end{equation}
In this case, $X$ is a.s.~uniquely determined as soon as $G_{t}(\rho) > 0$ a.s.
\end{Corollary}

Now we rewrite a conditional expectation of a stopped integral by means of the survival process $G(\rho)$.

\begin{Lemma}\label{le:stopped integral}
Let $s\in [0,T]$ and $G(\rho)$ be measurable. If $X$ and $\tilde{X}$ are two $[0,\infty]$-valued measurable processes such that $X_{t}$ is $\mathscr{F}_{t}$-measurable and $X_{t} = \tilde{X}_{t}$ a.s.~on $\{\rho > t\}$ for all $t\in [s,T]$, then
\begin{equation*}
E\bigg[\int_{s}^{T\wedge\rho} \tilde{X}_{t}\,dt\,\bigg|\,\mathscr{F}_{s}\bigg] = E\bigg[\int_{s}^{T}X_{t}G_{t}(\rho)\,dt\,\bigg|\,\mathscr{F}_{s}\bigg]\quad\text{a.s.}
\end{equation*}
\end{Lemma}

To consider conditional expectations of processes combined with stopping times, we require a general concept of conditional independence.

\begin{Definition}
Let $m\in\mathbb{N}$ and $\tau_{1},\dots,\tau_{m}$ be $[0,T]\cup\{\infty\}$-valued random variables. Then $\tau_{1},\dots,\tau_{m}$ are called \emph{$(\mathscr{F}_{t})_{t\in [0,T]}$-conditionally independent} if
\begin{equation*}
P(\tau_{1} > s_{1},\dots,\tau_{m} > s_{m}|\mathscr{F}_{t}) = P(\tau_{1} > s_{1}|\mathscr{F}_{t})\cdots P(\tau_{m} > s_{m}|\mathscr{F}_{t})\quad\text{a.s.}
\end{equation*}
for each $t\in [0,T]$ and any $s_{1},\dots,s_{m}\in [0,t]$.
\end{Definition}

As $[0,T]\cup\{\infty\}$ is closed in the Polish space $[0,\infty]$, any random variable $\tau$ taking all its values there admits a regular conditional probability $K$ given $\mathscr{F}_{t}$, where $t\in [0,T]$. That is, $K$ is a Markovian kernel from $(\Omega,\mathscr{F}_{t})$ to $[0,T]\cup\{\infty\}$ such that
\begin{equation*}
P(\tau\in B|\mathscr{F}_{t}) = K(\cdot,B)\quad\text{a.s.}\quad \text{for any $B\in\mathscr{B}([0,T]\cup\{\infty\})$}.
\end{equation*}
In consequence, if two $[0,T]\cup\{\infty\}$-valued random variables are $(\mathscr{F}_{t})_{t\in [0,T]}$-conditionally independent, then their joint conditional distribution with respect to $\mathscr{F}_{t}$ is completely determined up to time $t$ in the following sense.

\begin{Lemma}\label{le:conditional distribution}
Let $\sigma,\tau$ be two $[0,T]\cup\{\infty\}$-valued $(\mathscr{F}_{t})_{t\in [0,T]}$-conditionally independent random variables and $t\in [0,T]$. Then
\begin{equation}\label{eq:conditional distribution}
P((\sigma,\tau)\in C|\mathscr{F}_{t})(\omega) = K(\omega,\cdot)\otimes L(\omega,\cdot)(C)\quad\text{for $P$-a.e.~$\omega\in\Omega$,}
\end{equation}
all $C\in\mathscr{B}(([0,t]\cup\{\infty\})^{2})$ and any two respective regular conditional probabilities $K$ and $L$ of $\sigma$ and $\tau$ given $\mathscr{F}_{t}$.
\end{Lemma}

We conclude with the following integral representation within conditional expectations, which extends Proposition 5.11 in~\cite{BieRut02}.

\begin{Proposition}\label{pr:conditional integral representation}
Let $s\in [0,T[$ and $\sigma,\tau\in\mathscr{T}$. Assume that $\tilde{X}\in\tilde{\mathscr{S}}$ admits bounded left-continuous paths such that the following three conditions hold:
\begin{enumerate}[(i)]
\item $G(\sigma)$ is right-continuous and of finite variation and $\sigma$, $\tau$ are $(\mathscr{F}_{t})_{t\in [0,T]}$-conditionally independent.
\item There exists an $(\mathscr{F}_{t})_{t\in [0,T]}$-progressively measurable process $X$ with bounded paths satisfying $X_{t} = \tilde{X}_{t}$ a.s.~on $\{t < \sigma \leq T\}$ for each $t\in ]s,T]$.
\item The paths $G(\tau)(\omega)$ and $X(\omega)$ are left-continuous except at countably many points, excluding any discontinuity point of $G(\sigma)(\omega)$, for each $\omega\in\Omega$.
\end{enumerate}
If $\sup_{t\in ]s,T]} |\tilde{X}_{t}|\mathbbm{1}_{\{s < \sigma\leq T\wedge\tau\}}$ and $\sup_{t\in ]s,T]}|X_{t}|G_{t}(\tau)(V_{T}(\sigma)-V_{s}(\sigma))$ are integrable, where $V(\sigma)$ is the variation process of $G(\sigma)$, then
\begin{equation*}
E[\tilde{X}^{\sigma}_{T}\mathbbm{1}_{\{s < \sigma\leq T\wedge\tau\}}|\mathscr{F}_{s}] = - E\bigg[\int_{]s,T]} X_{t}G_{t}(\tau)\,dG_{t}(\sigma)\,\bigg|\,\mathscr{F}_{s}\bigg]\quad\text{a.s.}
\end{equation*}
\end{Proposition}

\subsection{Construction of conditionally independent hitting times}\label{se:2.3}

For given $m\in\N$ let $X$ be an $[0,\infty]^{m}$-valued $(\mathscr{F}_{t})_{t\in [0,T]}$-adapted right-continuous process and $\xi$ be an $\mathbb{R}_{+}^{m}$-valued $\tilde{\mathscr{F}}_{0}$-measurable random vector that is independent of $\mathscr{F}_{T}$ such that $\xi_{1},\dots,\xi_{m}$ are independent.

We assume that the $i$-th coordinate process of $X$, denoted by $X^{(i)}$, is increasing, let $G_{i}$ be the survival function of $\xi_{i}$ and define a function $\tau_{i}$ on $\Omega$ with values in $[0,T]\cup\{\infty\}$ via 
\begin{equation*}
\tau_{i} := \inf\{ t\in [0,T]\,|\, X_{t}^{(i)} \geq \xi_{i}\}
\end{equation*}
for any $i\in\{1,\dots,m\}$. Then the hitting time $\tau_{i}$ does not need to be an $(\mathscr{F}_{t})_{t\in [0,T]}$-stopping time, as $\xi_{i}$ may fail to be $\mathscr{F}_{0}$-measurable. However, the following facts hold.

\begin{Lemma}\label{le:hitting time}
The functions $\tau_{1},\dots,\tau_{m}$ are $(\tilde{\mathscr{F}}_{t})_{t\in [0,T]}$-stopping times that are conditionally independent relative to $(\mathscr{F}_{t})_{t\in [0,T]}$ such that $\{\tau_{j} > t\} = \{X_{t}^{(j)} < \xi_{j}\}$ and
\begin{equation*}
P(\tau_{1} > s_{1},\dots,\tau_{j} > s_{j}|\mathscr{F}_{t}) = G_{1}(X_{s_{1}}^{(1)})\cdots G_{j}(X_{s_{j}}^{(j)})\quad\text{a.s.}
\end{equation*}
for any $j\in\{1,\dots,m\}$, each $t\in [0,T]$ and every $s_{1},\dots,s_{j}\in [0,t]$.
\end{Lemma}

As a direct consequence, $\rho:=\min_{i\in\{1,\dots,m\}} \tau_{i}$ is an $(\tilde{\mathscr{F}}_{t})_{t\in [0,T]}$-stopping time and each $(\mathscr{F}_{t})_{t\in [0,T]}$-survival process $G(\rho)$ of $\rho$ satisfies
\begin{equation}\label{eq:hitting time conditional distribution}
G_{s}(\rho) = P(\rho > s|\mathscr{F}_{t}) = G_{1}(X_{s}^{(1)})\cdots G_{m}(X_{s}^{(m)})\quad\text{a.s.}
\end{equation}
for any $s,t\in [0,T]$ with $s\leq t$. Further relevant properties may be inferred by using $a\in\R_{+}^{m}$ and $b\in [0,\infty]^{m}$ defined coordinatewise via $a_{i}:=\essinf\xi_{i}$ and $b_{i}:=\esssup\xi_{i}$.

\begin{Lemma}\label{le:minimal hitting time}
For each $s\in [0,T]$ the following three assertions hold:
\begin{enumerate}[(i)]
\item $G_{s}(\rho) > 0$ a.s.~$\Leftrightarrow$ $X_{s}^{(i)} < b_{i}$ a.s.~for all $i\in\{1,\dots,m\}$.
\item $\rho\leq s$ a.s.~$\Leftrightarrow$ $X_{s}^{(i)}\geq b_{i}$ for some $i\in\{1,\dots,m\}$ a.s. Similarly,
\begin{equation*}
\rho > s\quad\text{a.s.}\quad \Leftrightarrow\quad X_{s}^{(i)} \leq a_{i}\quad\text{a.s.},\quad\text{if}\quad \xi_{i} > a_{i}\quad\text{a.s.},
\end{equation*}
and $X_{s}^{(i)} < a_{i}$ a.s., if $P(\xi_{i} = a_{i}) > 0$, for any $i\in\{1,\dots,m\}$.
\item $\rho \neq s$ a.s.~whenever $s > 0$,  $X$ is a.s.~continuous and $G_{1},\dots,G_{m}$ are continuous. 
\end{enumerate}
\end{Lemma}

\begin{Example}\label{ex:integral form}
Let $\hat{x}\in\R_{+}^{m}$ and $\lambda$ be an $[0,\infty]^{m}$-valued process that is progressively measurable relative to $(\mathscr{F}_{t})_{t\in [0,T]}$ such that
\begin{equation}\label{eq:integral form}
X_{t} = \hat{x} + \int_{0}^{t}\lambda_{s}\,ds\quad\text{for all $t\in [0,T]$.}
\end{equation}
Then $X$ is left-continuous, by monotone convergence, and the assumed right-continuity of $X$ holds if and only if for every $\omega\in\Omega$ there is $t_{\omega}\in ]0,T]$ such that
\begin{equation*}
\sum_{i=1}^{m}\int_{0}^{t_{\omega}}\lambda_{s}^{(i)}(\omega)\,ds < \infty.
\end{equation*}
In addition, Lemma~\ref{le:minimal hitting time} entails the following three statements:
\begin{enumerate}[(1)]
\item Assume that $b_{i} = \infty$ for all $i\in\{1,\dots,m\}$. Then $G_{s}(\rho) > 0$ a.s.~for any $s\in [0,T]$ if and only if $\lambda$ admits a.s.~(Lebesgue) integrable paths, and
\begin{equation*}
\rho < \infty\quad\text{a.s.}\quad\Leftrightarrow\quad \sum_{i=1}^{m}\int_{0}^{T}\lambda_{t}^{(i)}\,dt = \infty\quad\text{a.s.}
\end{equation*}
\item If $\hat{x}_{i}\geq a_{i}$ for any $i\in\{1,\dots,m\}$ and the event of all $\omega\in\Omega$ with $\int_{0}^{t_{\omega}}\lambda_{s}^{(i)}(\omega)\,ds > 0$ for all $i\in\{1,\dots,m\}$ has positive probability, then $P(\rho > s) < 1$ for any $s\in ]0,T]$.
\item $\rho \neq 0$ a.s.~$\Leftrightarrow$ For each $i\in\{1,\dots,m\}$ we have $\hat{x}_{i}\leq a_{i}$ with equality if and only if $\xi_{i} > a_{i}$ a.s. Further, $\rho\neq s$ a.s.~for all $s\in ]0,T]$ if $G_{1},\dots,G_{m}$ are continuous.
\end{enumerate}
\end{Example}

Note that if $X_{0}^{(i)}\geq a_{i}$ a.s.~for some $i\in\{1,\dots,m\}$ and $\xi_{i}$ is a.s.~constant, which is equivalent to the condition that $a_{i} = b_{i}$, then $\rho = 0$ a.s., since $P(\rho > 0) \leq G_{i}(b_{i}) = 0$. For this reason, let us now assume that $a_{i} < b_{i}$ for all $i\in\{1,\dots,m\}$.

We define an event in $\mathscr{F}_{t}$ by $\Lambda_{t}:=\bigcap_{i=1}^{m}\{X_{t}^{(i)} < b_{i}\}$ for each $t\in [0,T]$. While $\{\rho  > t\}$ is included in $\Lambda_{t}$, we have $P(\rho > t) > 0$ if and only if $P(\Lambda_{t}) > 0$, by Lemmas~\ref{le:hitting time} and~\ref{le:minimal hitting time}. Based on these considerations, we provide a  formula for the survival function of $\rho$.

\begin{Proposition}\label{pr:density formula}
Let $\hat{x}\in\R_{+}^{m}$ and $\lambda$ be some $[0,\infty]^{m}$-valued $(\mathscr{F}_{t})_{t\in [0,T]}$-progressively measurable process satisfying~\eqref{eq:integral form} such that for each $\omega\in\Omega$ we have
\begin{equation}\label{eq:positivtity and finiteness condition}
0 < \int_{0}^{t_{\omega}}\lambda_{s}^{(i)}(\omega)\,ds < \infty \quad\text{for all $i\in\{1,\dots,m\}$ and some $t_{\omega}\in ]0,T]$}.
\end{equation}
If $\hat{x}_{i}\in [a_{i},b_{i}[$ and $G_{i}$ is continuously differentiable on $]a_{i},b_{i}[$ for any $i\in\{1,\dots,m\}$, then
\begin{align*}
P(\rho > t) &= G_{1}(\hat{x}_{1})\cdots G_{m}(\hat{x}_{m})P(\Lambda_{t})\\
&\quad + \int_{0}^{t}E\bigg[G_{1}(X_{s}^{(1)})\cdots G_{m}(X_{s}^{(m)})\sum_{i=1}^{m}\lambda_{s}^{(i)}\bigg(\frac{G_{i}'}{G_{i}}\bigg)(X_{s}^{(i)});\Lambda_{t}\bigg]\,ds
\end{align*}
for every $t\in[0,T]$.
\end{Proposition}

\begin{Remark}\label{re:pathwise stated condition}
The pathwise stated condition~\eqref{eq:positivtity and finiteness condition} is satisfied if and only if for any $\omega\in\Omega$ there are $t_{\omega}\in ]0,T]$ and $s_{\omega}\in ]0,t_{\omega}]$ such that
\begin{equation*}
\lambda(\omega) > 0\quad\text{a.e.~on $]0,s_{\omega}[$}\quad\text{and}\quad \text{$\lambda(\omega)$ is integrable on $[0,t_{\omega}]$}.
\end{equation*}
For instance, this is the case is if $\lambda$ is $]0,\infty[^{m}$-valued and admits integrable paths.
\end{Remark}

To conclude our analysis, let us impose the gamma distribution on $\xi_{1},\dots,\xi_{m}$. This includes the hitting times considered in~\cite{BriFraPal19} as special case, by choosing an exponential distribution with mean one.

\begin{Example}\label{ex:gamma distributed hitting times}
For each $i\in\{1,\dots,m\}$ let $\xi_{i}$ be gamma distributed with shape $\alpha_{i} > 0$ and rate $\beta_{i} > 0$. That is, its survival function and the gamma function $\Gamma$ satisfy
\begin{equation*}
G_{i}(x) = \frac{\beta_{i}^{\alpha_{i}}}{\Gamma(\alpha_{i})}\int_{x}^{\infty} y^{\alpha_{i} - 1}e^{-\beta_{i} y}\,dy\quad\text{for all $x\in\R_{+}$.}
\end{equation*}
We suppose that $\lambda$ is an $[0,\infty]^{m}$-valued $(\mathscr{F}_{t})_{t\in [0,T]}$-progressively measurable process such that $X_{t} = \int_{0}^{t}\lambda_{s}\,ds$ for all $t\in [0,T]$ and~\eqref{eq:positivtity and finiteness condition} holds. Then the formula~\eqref{eq:hitting time conditional distribution} yields that
\begin{equation*}
P(\rho > s|\mathscr{F}_{t}) = \frac{\gamma(\alpha_{1},\beta_{1}X_{s}^{(1)})}{\Gamma(\alpha_{1})}\cdots \frac{\gamma(\alpha_{m},\beta_{m}X_{s}^{(m)})}{\Gamma(\alpha_{m})}\quad\text{a.s.}
\end{equation*}
for any $s,t\in [0,T]$ with $s\leq t$, where $\gamma:]0,\infty[^{2}\rightarrow ]0,\infty[$, $\gamma(\alpha,x):=\int_{x}^{\infty}y^{\alpha-1}e^{-y}\,dy$ is the upper incomplete gamma function. From Example~\ref{ex:integral form} we in particular infer that
\begin{equation*}
P(\rho > s) < 1\quad\text{for all $s\in ]0,T]$}\quad\text{and}\quad \rho\neq s\quad\text{a.s.}\quad\text{for any $s\in [0,T]$.}
\end{equation*}
Moreover, if $\lambda$ admits integrable paths, then $P(\rho = \infty) > 0$ and Proposition~\ref{pr:density formula} entails that the distribution of $\rho$ decomposes into its continuous and discrete part. Namely,
\begin{equation*}
P(\rho\in B) = \int_{B\cap [0,T]}\varphi_{\rho}(s)\,ds + \bigg(1 - \int_{0}^{T}\varphi_{\rho}(s)\,ds\bigg)\delta_{\infty}(B)
\end{equation*}
for any $B\in\mathscr{B}([0,T]\cup\{\infty\})$ with the measurable integrable function $\varphi_{\rho}:[0,T]\rightarrow [0,\infty]$ given by
\begin{equation*}
\varphi_{\rho}(s) := E\bigg[\frac{\gamma(\alpha_{1},\beta_{1}X_{s}^{(1)})}{\Gamma(\alpha_{1})}\cdots \frac{\gamma(\alpha_{m},\beta_{m}X_{s}^{(m)})}{\Gamma(\alpha_{m})}\sum_{i=1}^{m}\lambda_{s}^{(i)}\frac{\beta_{i}^{\alpha_{i}}(X_{s}^{(i)})^{\alpha_{i}-1}}{\gamma(\alpha_{i},\beta_{i}X_{s}^{(i)})}e^{-\beta_{i}X_{s}^{(i)}}\bigg].
\end{equation*}
\end{Example}

\section{A financial market model with default}\label{se:3}

We aim to evaluate a derivative contract between an \emph{investor} $\mathcal{I}$ and a \emph{counterparty} $\mathcal{C}$, both considered as financial entities, with a special focus on the case that $\mathcal{I}$ stands for an \emph{investment bank} $\mathcal{B}$.

\subsection{Model specifications}\label{se:3.1}

In the sequel, we interpret the two filtrations $(\mathscr{F}_{t})_{t\in [0,T]}$ and $(\tilde{\mathscr{F}}_{t})_{t\in [0,T]}$ as the \emph{temporal developments of the default-free information} and the \emph{whole available information} on an underlying financial market, respectively.

We use two $[0,T]\cup\{\infty\}$-valued random variables $\tau_{\mathcal{I}}$ and $\tau_{\mathcal{C}}$ to model the respective \emph{default times} of the investor $\mathcal{I}$ and the counterparty $\mathcal{C}$. Then $\tau:=\tau_{\mathcal{I}}\wedge\tau_{\mathcal{C}}$ stands for the \emph{time of a party to default first}. By using the notation in~\eqref{eq:stopping time sigma-field}, we require that
\begin{equation}\label{eq:market condition 1}
\mathscr{F}_{t}\subset \tilde{\mathscr{F}}_{t}\subset\mathscr{F}_{t}\vee\mathscr{H}_{t}^{\tau_{\mathcal{I}}}\vee\mathscr{H}_{t}^{\tau_{\mathcal{C}}}\quad\text{for all $t\in [0,T]$.}
\end{equation}
Thus, the available market information could yield no knowledge about the first time of default and it may fail to give any insight into the respective default times of $\mathcal{I}$ and $\mathcal{C}$.

In our continuous-time setting we assume that the distributions of $\tau_{\mathcal{I}}$ and $\tau_{\mathcal{C}}$ admit at most one atom, which is at infinity, and both parties cannot default simultaneously. That is, for any $t\in [0,T]$ we have
\begin{equation}\label{eq:market condition 2}
P(\tau_{\mathcal{I}}=t) = P(\tau_{\mathcal{C}} = t) = 0\quad\text{and}\quad P(\tau_{\mathcal{I}} = \tau_{\mathcal{C}}, \tau <\infty) = 0.
\end{equation}
The first condition implies that $\tau \neq t$ a.s.~for all $t\in [0,T]$. However, as $\{\tau_{\mathcal{I}} = \tau_{\mathcal{C}} = \infty\}$ $= \{\tau = \infty\}$ and we have made no restrictions on $\tilde{P}(\tau_{\mathcal{I}} = \infty)$ and $\tilde{P}(\tau_{\mathcal{C}}=\infty)$, both entities may not default at all. So, we allow for $\tilde{P}(\tau = \infty)\in [0,1]$.

\begin{Remark}
The event $\{\tau_{\mathcal{I}}=\tau_{\mathcal{C}},\tau < \infty\}$ of simultaneous default is a null set if $\tau_{\mathcal{I}}$ and $\tau_{\mathcal{C}}$ are $(\mathscr{F}_{t})_{t\in [0,T]}$-conditionally independent under $\tilde{P}$. Indeed, in this case Lemma~\ref{le:conditional distribution} gives
\begin{equation}\label{eq:simultaneous default}
\tilde{P}(\tau_{\mathcal{I}} = \tau_{\mathcal{C}},\tau < \infty) = \tilde{P}((\tau_{\mathcal{I}},\tau_{\mathcal{C}})\in\Delta) = \tilde{E}\bigg[\int_{[0,T]} K(\cdot,\{t\})\,L(\cdot,dt)\bigg] = 0
\end{equation}
for any two respective regular conditional probabilities $K$ and $L$ of $\tau_{\mathcal{I}}$ and $\tau_{\mathcal{C}}$ under $\tilde{P}$ given $\mathscr{F}_{T}$, where $\Delta := \{(s,t)\in [0,T]^{2}\,|\, s = t\}$. Thereby, we note that
\begin{equation*}
K(\omega,\{t\}) = 0\quad\text{for all $(\omega,t)\in N^{c}\times [0,T]$}
\end{equation*}
and some null set $N\in\mathscr{F}_{T}$, since $\tilde{P}(\tau_{\mathcal{I}} = t|\mathscr{F}_{t}) = 0$ a.s.~for all $t\in [0,T]$ and $\mathscr{B}([0,T]\cup\{\infty\})$ is countably generated. This justifies that the expectation in~\eqref{eq:simultaneous default} vanishes.
\end{Remark}

Next, for any $(\mathscr{F}_{t})_{t\in [0,T]}$-progressively measurable process $\gamma$ with integrable paths we introduce an $]0,\infty[$-valued function $D(\gamma)$ on $[0,T]^{2}\times\Omega$ by 
\begin{equation*}
D_{s,t}(\gamma) := \exp\bigg(-\int_{s}^{t}\gamma_{\tilde{s}}\,d\tilde{s}\bigg),\quad\text{if $s\leq t$,}
\end{equation*}
and $D_{s,t}(\gamma):=1$, otherwise. Then the function $[0,T]^{2}\rightarrow ]0,\infty[$, $(s,t)\mapsto D_{s,t}(\gamma)(\omega)$ is continuous for any $\omega\in\Omega$ and $D_{s,t}(\gamma)$ is $\mathscr{F}_{t}$-measurable for all $s,t\in [0,T]$. Moreover, $D(\gamma)$ is bounded as soon as $\gamma$ is bounded from below. 

Let $r$ be an $(\mathscr{F}_{t})_{t\in [0,T]}$-progressively measurable process with integrable paths that represents the \emph{instantaneous risk-free interest rate}. Then $D_{s,t}(r)$ is the \emph{discount factor} from time $s\in [0,T]$ to $t\in [s,T]$. Put differently, $D_{s,t}(r)$ specifies the required amount to invest risk-free at time $s$, in order to receive $1$ unit of cash at time $t$.

We let $\tilde{P}$ be a \emph{local martingale measure} after a time $t_{0}\in [0,T]$ in the sense that any discounted price process of a traded non-dividend-paying risky asset is an $(\tilde{\mathscr{F}}_{t})_{t\in [t_{0},T]}$-local martingale. That is, there is a non-empty set of processes $\tilde{U}\in\tilde{\mathscr{S}}$, representing the price processes of all such assets, for which $[t_{0},T]\times\Omega\rightarrow\mathbb{R}$, $(t,\omega)\mapsto D_{0,t}(r)(\omega)\tilde{U}_{t}(\omega)$ is an $(\tilde{\mathscr{F}}_{t})_{t\in [t_{0},T]}$-local martingale under $\tilde{P}$. At all times, $\tilde{P}$ is ought to be equivalent to $P$.

Given the available market information, we will derive an equation for the \emph{value process}, denoted by $\tilde{\mathscr{V}}\in\tilde{\mathscr{S}}$, of a trading strategy that hedges the contract between $\mathcal{I}$ and $\mathcal{C}$ under $\tilde{P}$ that leads to no arbitrage. 

In the end, however, we seek a valuation that does not involve any knowledge of the default of any of the two parties, and the \emph{valuation equation} for $\tilde{\mathscr{V}}$ includes quantities that merely depend on its pre-default part in the following sense.

As introduced in~\eqref{eq:survival process}, let $G(\sigma)$ be an $(\mathscr{F}_{t})_{t\in [0,T]}$-survival process under $\tilde{P}$ of a random variable $\sigma$ with values in $[0,T]\cup\{\infty\}$, which is an $[0,1]$-valued $(\mathscr{F}_{t})_{t\in [0,T]}$-supermartingale under $\tilde{P}$ such that
\begin{equation*}
\tilde{P}(\sigma > t|\mathscr{F}_{t}) = G_{t}(\sigma)\quad\text{a.s.~for all $t\in [0,T]$}.
\end{equation*}
Let us call a process $\tilde{X}$ \emph{integrable up to time $\tau$} if $[0,T]\times\Omega\rightarrow\R$, $(t,\omega)\mapsto \tilde{X}_{t}(\omega)\mathbbm{1}_{\{\tau > t\}}(\omega)$ is integrable. By Corollary~\ref{co:identification}, this property is satisfied by $\tilde{X}\in\tilde{\mathscr{S}}$ if and only if there is $X\in\mathscr{S}$ such that $XG(\tau)$ is integrable and $X_{s} = \tilde{X}_{s} $ a.s.~on $\{\tau > s\}$ for each $s\in [0,T]$. In this case,
\begin{equation}\label{eq:pre-default version}
X_{s}G_{s}(\tau) = \tilde{E}[\tilde{X}_{s}\mathbbm{1}_{\{\tau > s\}}|\mathscr{F}_{s}]\quad\text{a.s.~for all $s\in [0,T]$}
\end{equation}
and we shall call $X$ a \emph{pre-default version} of $\tilde{X}$. If in addition $G_{s}(\tau) > 0$ a.s.~for every $s\in [0,T]$, which implies that the probability that neither $\mathcal{I}$ nor $\mathcal{C}$ defaults at any time is positive, then $X$ is unique up to a modification.

In this spirit, we will introduce valuation based on default-free information only and analyse any \emph{pre-default value process} $\mathscr{V}$ defined as pre-default version of $\tilde{\mathscr{V}}$, which in turn should be integrable up to time $\tau$.

\begin{Remark}\label{re:martingale property}
Any $(\mathscr{F}_{t})_{t\in [t_{0},T]}$-martingale $X$ under $\tilde{P}$, in the local or standard sense, also satisfies the respective $(\tilde{\mathscr{F}}_{t})_{t\in [t_{0},T]}$-martingale property if for both $i\in\{\mathcal{I},\mathcal{C}\}$ we have
\begin{equation}\label{eq:condition on the survival process}
G_{s}(\tau_{i}) = \tilde{P}(\tau_{i} > s|\mathscr{F}_{t})\quad\text{a.s.~for any $s,t\in [0,T]$ with $s\leq t$}
\end{equation}
and $\tau_{\mathcal{I}}$, $\tau_{\mathcal{C}}$ are $(\mathscr{F}_{t})_{t\in [0,T]}$-conditionally independent under $\tilde{P}$. Due to Lemma~\ref{le:hitting time}, these conditions are met in Example~\ref{ex:default time specification}, which we will consider after analysing the model. 
\end{Remark}

\subsection{Incorporation of all relevant cash flows, costs and benefits}\label{se:3.2}

Let us summarise all cash flows, costs and benefits that may impact the value of the contract between $\mathcal{I}$ and $\mathcal{C}$. These quantities are the \emph{contractual derivative cash flows}~\eqref{eq:cash flows and costs 1}, the \emph{costs and benefits of a collateral account}~\eqref{eq:cash flows and costs 2}, the \emph{funding costs and benefits}~\eqref{eq:cash flows and costs 3}, the \emph{repo costs and benefits associated to the hedging account}~\eqref{eq:cash flows and costs 4} and the \emph{cash flows arising on the default of one of the two parties}~\eqref{eq:cash flows and costs 5}.

Despite the contractual cash flows, all remaining quantities are allowed to depend on the value process $\tilde{\mathscr{V}}$ or its pre-default version $\mathscr{V}$. For a mathematical description we define a \emph{time-dependent random functional} on a set $\mathscr{D}$ in $\tilde{\mathscr{S}}$ to be a function
\begin{equation*}
F:[0,T]\times\Omega\times\mathscr{D}\rightarrow\R,\quad (t,\omega,X)\mapsto F_{t}(X)(\omega)
\end{equation*}
for which $F(X):[0,T]\times\Omega\rightarrow\R$, $(t,\omega)\mapsto F_{t}(X)(\omega)$ is a process for every $X\in\mathscr{D}$. If there is a filtration $(\mathscr{G}_{t})_{t\in [0,T]}$ of $\mathscr{F}$ to which $F(X)$ is adapted for all $X\in\mathscr{D}$, then we will refer to an $(\mathscr{G}_{t})_{t\in [0,T]}$-time-dependent random functional.

\begin{enumerate}[(1)]
\item {\bf The contractual derivative cash flows} with $\mathcal{C}$ are supposed to depend on a payoff functional and a dividend-paying risky asset that is influenced by its variance, or squared volatility, in an extended sense.
\item[-] The \emph{price of the risky asset}, its \emph{quasi variance} and its \emph{instantaneous dividend rate} are modelled by two $\mathbb{R}_{+}$-valued c\a dl\a g processes $S$ and $V$ in $\mathscr{S}$ and some process $\pi$ that is $(\mathscr{F}_{t})_{t\in [0,T]}$-progressively measurable and admits integrable paths, respectively.
\item[-] The $\mathbb{R}_{+}$-valued Borel measurable functional $\Phi$ defined on the closed set of all paths in $D([0,T])\times D([0,T])$ with non-negative entries represents the \emph{payoff functional}.
\item[-] The \emph{contractual cash flows} consist of the amount $\Phi(S,V)$ paid at maturity and dividends according to the rate $\pi$. The continuous process $\null_{\mathrm{con}}\mathrm{CF}$ representing the discounted future cash flows at any time point is given by
\begin{equation}\label{eq:cash flows and costs 1}
\null_{\mathrm{con}}\mathrm{CF}_{s} := D_{s,T}(r)\Phi(S,V)\mathbbm{1}_{\{\tau > T\}} + \int_{s}^{T\wedge\tau}D_{s,t}(r)\pi_{t}\,dt.
\end{equation}
\item {\bf The costs and benefits of a collateral account} arising from the collateralisation procedure to mitigate the default risk, subject to the collateral remuneration rate.
\item[-] Namely, the collateral serves as guarantee in case of default and the party receiving it will have to remunerate it at a certain interest rate, called the \emph{collateral rate}, determined by the contract. We assume that the assets received as collateral can be re-hypotecated and do not have to be kept segregated.
\item[-] For an $(\mathscr{F}_{t})_{t\in [0,T]}$-time-dependent random functional $C$ on $\mathscr{S}$ the process $C(\mathscr{V})$, required to be c\a gl\a d, models the \emph{cash flows of the collateral procedure}.
\item[-] The \emph{collateral rates of each party} are represented by two $(\mathscr{F}_{t})_{t\in [0,T]}$-progressively measurable processes $\null_{+}c$ and $\null_{-}c$ with integrable paths, respectively.
\item[-] So, $\mathcal{I}$ is a collateral receiver remunerating the assets at the rate $\null_{+}c_{t}$ on $\{C_{t}(\mathscr{V}) > 0\}$ and a collateral provider investing at the rate $\null_{-}c_{t}$ on $\{C_{t}(\mathscr{V}) < 0\}$ for all $t\in [0,T]$. 
\item[-] By means of the $(\mathscr{F}_{t})_{t\in [0,T]}$-time-dependent random functional $c$ on $\mathscr{S}$ given by
\begin{equation*}
c_{t}(X) := \null_{+}c_{t}\mathbbm{1}_{]0,\infty[}\big(C_{t}(X)\big) + \null_{-}c_{t}\mathbbm{1}_{]-\infty,0[}\big(C_{t}(X)\big),
\end{equation*}
the process $c(\mathscr{V})$ represents the \emph{respective collateral rate}.
\item[-] We define a time-dependent random functional $\null_{\mathrm{col}}\mathrm{C}$ on the set of all $X\in\mathscr{S}$ for which $C(X)$ is c\a gl\a d by
\begin{equation}\label{eq:cash flows and costs 2}
\null_{\mathrm{col}}\mathrm{C}_{s}(X) := \int_{s}^{T\wedge\tau}D_{s,t}(r)(c_{t}(X) - r_{t})C_{t}(X)\,dt.
\end{equation}
Then the continuous process $\null_{\mathrm{col}}\mathrm{C}(\mathscr{V})$ stands for the time evolution of the discounted future \emph{collateral costs and benefits}.

\item {\bf The costs and benefits of a funding account} that may accrue, as $\mathcal{I}$ is supposed to have access to an account for borrowing or investing money at two respective risk-free interest rates.
\item[-] Given an $(\tilde{\mathscr{F}}_{t})_{t\in [0,T]}$-time-dependent random functional $\tilde{F}$ on $\tilde{\mathscr{S}}$, the process $\tilde{F}(\tilde{\mathscr{V}})$, supposed to be c\a gl\a d, stands for the \emph{funding amount}.
\item[-] The \emph{interest rates for borrowing and lending} are given by the $(\tilde{\mathscr{F}}_{t})_{t\in [0,T]}$-progressively measurable processes $\null_{+}\tilde{f}$ and $\null_{-}\tilde{f}$ with integrable paths, respectively.
\item[-] Hence, $\mathcal{I}$ is borrowing the amount $\tilde{F}_{t}(\tilde{\mathscr{V}})$ at the interest rate $\null_{+}\tilde{f}_{t}$ on $\{\tilde{F}_{t}(\tilde{\mathscr{V}}) > 0\}$ and she is lending $-\tilde{F}_{t}(\tilde{\mathscr{V}})$ at the rate $\null_{-}\tilde{f}_{t}$ on $\{\tilde{F}_{t}(\tilde{\mathscr{V}}) < 0\}$ for all $t\in [0,T]$. 
\item[-] By using the $(\tilde{\mathscr{F}}_{t})_{t\in [0,T]}$-time-dependent random functional $\tilde{f}$ on $\tilde{\mathscr{S}}$ defined via
\begin{equation*}
\tilde{f}_{t}(\tilde{X}) := \null_{+}\tilde{f}_{t}\mathbbm{1}_{]0,\infty[}\big(\tilde{F}_{t}(\tilde{X})\big) + \null_{-}\tilde{f}_{t}\mathbbm{1}_{]-\infty,0[}\big(\tilde{F}_{t}(\tilde{X})\big),
\end{equation*}
the process $\tilde{f}(\tilde{\mathscr{V}})$ yields the \emph{respective funding rate}.
\item[-] We introduce a time-dependent random functional $\null_{\mathrm{fun}}\mathrm{C}$ on the set of all $\tilde{X}\in\tilde{\mathscr{S}}$ for which $\tilde{F}(\tilde{X})$ is c\a gl\a d by
\begin{equation}\label{eq:cash flows and costs 3}
\null_{\mathrm{fun}}\mathrm{C}_{s}(\tilde{X}) := \int_{s}^{T\wedge\tau}D_{s,t}(r)(\tilde{f}_{t}(\tilde{X}) - r_{t})\tilde{F}_{t}(\tilde{X})\,dt.
\end{equation}
Then the continuous process $\null_{\mathrm{fun}}\mathrm{C}(\tilde{\mathscr{V}})$ represents the temporal development of the present value of the \emph{funding costs and benefits}.

\item As $\mathcal{I}$ may be a bank $\mathcal{B}$, we assume that she may enter repurchase agreements to hedge its exposure. For this reason, {\bf the repo costs and the benefits that result from hedging the derivative} should be taken into account.
\item[-] For an $(\tilde{\mathscr{F}}_{t})_{t\in [0,T]}$-time-dependent random functional $\tilde{H}$ on $\tilde{\mathscr{S}}$ the process $\tilde{H}(\tilde{\mathscr{V}})$, assumed to be c\a gl\a d, measures the \emph{value of the risky asset position that $\mathcal{I}$ has via the repo}.
\item[-] The two \emph{repo rates} are given by $(\tilde{\mathscr{F}}_{t})_{t\in [0,T]}$-progressively measurable processes $\null_{+}\tilde{h}$ and $\null_{-}\tilde{h}$ with integrable paths.
\item[-] Thus, $\mathcal{I}$ borrows a risky asset with the repo rate $\null_{+}\tilde{h}_{t}$ on $\{\tilde{H}_{t}(\tilde{\mathscr{V}}) > 0\}$ and lends a risky asset with the rate $\null_{-}\tilde{h}_{t}$ on $\{\tilde{H}_{t}(\tilde{\mathscr{V}}) < 0\}$ for each $t\in [0,T]$.
\item[-] We let the $(\tilde{\mathscr{F}}_{t})_{t\in [0,T]}$-time-dependent random functional $\tilde{h}$ on $\tilde{\mathscr{S}}$ be given by
\begin{equation*}
\tilde{h}_{t}(\tilde{X}) := \null_{+}\tilde{h}_{t}\mathbbm{1}_{]0,\infty[}\big(\tilde{H}_{t}(\tilde{X})\big) + \null_{-}\tilde{h}_{t}\mathbbm{1}_{]-\infty,0[}\big(\tilde{H}_{t}(\tilde{X})\big),
\end{equation*}
which yields $\tilde{h}(\tilde{\mathscr{V}})$ as \emph{respective repo rate}.
\item[-] Thereby, we suppose that $\mathcal{I}$ continuously rolls over repo contracts and that at each point $t\in [0,T]$ \emph{she receives in the repo the exact value of the assets she is lending}. Thus, the gain of the repo position is given by the growth of the assets that are being repoed minus $\tilde{h}_{t}(\tilde{\mathscr{V}})(-\tilde{H}_{t}(\tilde{\mathscr{V}}))$, the repo rate times the amount of cash received.
\item[-] In this context, let the time-dependent random functional $\null_{\mathrm{hed}}\mathrm{C}$ on the set of all $\tilde{X}\in\tilde{\mathscr{S}}$ for which $\tilde{H}(\tilde{X})$ is c\a gl\a d be defined via
\begin{equation}\label{eq:cash flows and costs 4}
\null_{\mathrm{hed}}\mathrm{C}_{s}(\tilde{X}) := \int_{s}^{T\wedge\tau}D_{s,t}(r)\big(r_{t} - \tilde{h}_{t}(\tilde{X})\big)\tilde{H}_{t}(\tilde{X})\,dt.
\end{equation}
Then the continuous process $\null_{\mathrm{hed}}\mathrm{C} (\tilde{\mathscr{V}})$ stands for the present value development of the \emph{hedging costs and benefits}.
\item {\bf The cash flows arising on the default of one of the two parties} that can be computed with the \emph{residual value of the claim}, the \emph{net exposure}, the \emph{losses given default} and the \emph{funding amount}.

\item[-] For an $(\mathscr{F}_{t})_{t\in [0,T]}$-time-dependent random functional $\varepsilon$ on $\mathscr{S}$ the process $\varepsilon(\mathscr{V})$, which is ought to be c\a gl\a d, models the time evolution of the \emph{close-out value}. 
\item[-] We interpret $\varepsilon_{\tau}(\mathscr{V})$ as the \emph{residual value of the claim} at the time $\tau$ of a party to default first on $\{\tau < \infty\}$, since $\tau_{\mathcal{I}}\neq \tau_{\mathcal{C}}$ a.s.~on this event.
\item[-] On $\{\tau_{\mathcal{I}} > \tau_{\mathcal{C}}\}$ we specify that if the \emph{net exposure} $(\varepsilon_{\tau} - C_{\tau})(\mathscr{V})$ at the time of default is non-positive, then $\mathcal{I}$ is a net debtor and repays $\varepsilon_{\tau}(\mathscr{V})$ to $\mathcal{C}$.
\item[-] If instead $(\varepsilon_{\tau} - C_{\tau})(\mathscr{V}) > 0$, then $\mathcal{I}$ is a net creditor and recovers a fraction $1-\mathrm{LGD}_{\mathcal{I}}$ of its credits, in which case it receives $C_{\tau}(\mathscr{V}) + (1-\mathrm{LGD}_{\mathcal{C}})(\varepsilon_{\tau} - C_{\tau})(\mathscr{V})$.
\item[-] We implicitly assume that the loss fractions $\mathrm{LGD}_{\mathcal{I}},\mathrm{LGD}_{\mathcal{C}}\in [0,1]$, which denote the \emph{losses given defaults} of $\mathcal{I}$ and $\mathcal{C}$, respectively, are deterministic exogenous quantities.
\item[-] The case in which $\mathcal{I}$ is a bank and defaults before $\mathcal{C}$ is symmetrical. If, however, $\mathcal{I}\neq\mathcal{B}$, then merely $\varepsilon_{\tau}(\mathscr{V})$ is being considered on $\{\tau_{\mathcal{I}} < \tau_{\mathcal{C}}\}$. 
\item[-] We define a time-dependent random functional $\null_{\mathrm{def,c}}\mathrm{CF}$ on the set of all $X\in\mathscr{S}$ for which $\varepsilon(X)$ and $C(X)$ are c\a gl\a d via
\begin{equation*}
\begin{split}
\null_{\mathrm{def,c}}\mathrm{CF}_{s}(X) &:= D_{s,\tau}(r)\big(\varepsilon_{\tau}(X) - \mathrm{LGD}_{\mathcal{C}}(\varepsilon_{\tau} - C_{\tau})^{+}(X)\mathbbm{1}_{\{\tau_{\mathcal{I}} > \tau_{\mathcal{C}}\}}\big)\\
&\quad + D_{s,\tau}(r)\mathrm{LGD}_{\mathcal{I}}(\varepsilon_{\tau} - C_{\tau})^{-}(X)\mathbbm{1}_{\{\mathcal{I}=\mathcal{B},\,\tau_{\mathcal{I}} < \tau_{\mathcal{C}}\}}
\end{split}
\end{equation*}
on $\{s < \tau < T\}$ and $\null_{\mathrm{def,c}}\mathrm{CF}_{s}(X):=0$ on the complement of this set. Then, according to our reasoning, the \emph{discounted future cash flows on default due to the contract} can be modelled by the c\a dl\a g process $\null_{\mathrm{def,c}}\mathrm{CF}(\mathscr{V})$.
\item[-] As we suppose that if $\mathcal{I}$ is a bank and has a cash surplus, then it may invest into risk-free assets, we also consider the \emph{cash flows on the bank's default due to funding}. 
\item[-] To this end, let the time-dependent random functional $\null_{\mathrm{def,f}}\mathrm{CF}$ on the set of all $\tilde{X}\in\tilde{\mathscr{S}}$ for which $\tilde{F}(\tilde{X})$ is c\a gl\a d be given by
\begin{equation*}
\null_{\mathrm{def,f}}\mathrm{CF}_{s}(\tilde{X}) := D_{s,\tau}(r)\mathrm{LGD}_{\mathcal{I}}\tilde{F}_{\tau}^{+}(\tilde{X})\mathbbm{1}_{\{\mathcal{I}=\mathcal{B},\,\tau_{\mathcal{I}} < \tau_{\mathcal{C}}\}}
\end{equation*}
on $\{s < \tau < T\}$ and $\null_{\mathrm{def,f}}\mathrm{CF}_{s}(\tilde{X}):=0$ on its complement. Then the c\a dl\a g process $\null_{\mathrm{def,f}}\mathrm{CF}(\tilde{\mathscr{V}})$ yields the time evolution of the corresponding net present value.
\item[-] Finally, for every $(X,\tilde{X})\in\mathscr{S}\times\tilde{\mathscr{S}}$ for which $C(X)$, $\varepsilon(X)$ and $\tilde{F}(\tilde{X})$ are c\a gl\a d we set
\begin{equation}\label{eq:cash flows and costs 5}
\null_{\mathrm{def}}\mathrm{CF}_{s}(X,\tilde{X}) := \null_{\mathrm{def,c}}\mathrm{CF}_{s}(X) + \null_{\mathrm{def,f}}\mathrm{CF}_{s}(\tilde{X}).
\end{equation}
Then the process $\null_{\mathrm{def}}\mathrm{CF}(\mathscr{V},\tilde{\mathscr{V}})$ sums up both sources of default risk.
\end{enumerate}

\subsection{The pre-default valuation equation}\label{se:3.3}

For the valuation of the derivative contract let us first ensure the integrability of the net present values of all the cash flows, costs and benefits generally given by~\eqref{eq:cash flows and costs 1}-\eqref{eq:cash flows and costs 5}. Throughout this section, $(\Omega,\mathscr{F},\tilde{P})$ serves as underlying probability space.

Let $\tilde{\mathscr{L}}(r,\tau)$ be the linear space of all random variables $X$ for which $D_{s,T}(r)|X|\mathbbm{1}_{\{\tau > T\}}$ is $\tilde{P}$-integrable for any $s\in [0,T]$ and $\tilde{\mathscr{S}}(r,\tau)$ be the linear space of all  measurable processes $X$ satisfying
\begin{equation*}
\tilde{E}\bigg[\int_{s}^{T\wedge\tau}D_{s,t}(r)|X_{t}|\,dt\bigg] < \infty \quad\text{for all $s\in [0,T[$.}
\end{equation*}
Furthermore, by $\tilde{\mathscr{D}}(r,\tau)$ we denote the linear space of all c\a gl\a d processes $X$ such that $\sup_{t\in ]s,T[}D_{s,t}(r)|X_{t}|\mathbbm{1}_{\{s < \tau < T\}}$ is $\tilde{P}$-integrable for each $s\in [0,T[$ and we set
\begin{equation*}
\tilde{\mathscr{L}}(r):=\tilde{\mathscr{L}}(r,\infty)\quad\text{and}\quad\tilde{\mathscr{S}}(r):=\tilde{\mathscr{S}}(r,\infty).
\end{equation*}

\begin{Remark}\label{re:integrability of pre-default components}
For any two $\R_{+}$-valued random variables $X$ and $\tilde{X}$ that are measurable with respect to $\mathscr{F}_{T}$ and $\tilde{\mathscr{F}_{T}}$ such that $X = \tilde{X}$ a.s.~on $\{\tau > T\}$ Corollary~\ref{co:identification} implies
\begin{equation*}
\tilde{E}\big[D_{s,T}(r)\tilde{X}\mathbbm{1}_{\{\tau > T\}}\big] = \tilde{E}\big[D_{s,T}(r)XG_{T}(\tau)\big]\quad\text{for any $s\in [0,T]$}.
\end{equation*}
Thus, $\tilde{X}\in\tilde{\mathscr{L}}(r,\tau)$ $\Leftrightarrow$ $XG_{T}(\tau)\in\tilde{\mathscr{L}}(r)$. Further, if $G(\tau)$ is measurable and now $X$ and $\tilde{X}$ denote two $\R_{+}$-valued processes that are progressively measurable relative to $(\mathscr{F}_{t})_{t\in [0,T]}$ and $(\tilde{\mathscr{F}}_{t})_{t\in [0,T]}$, respectively, then
\begin{equation*}
\tilde{E}\bigg[\int_{s}^{T\wedge\tau}D_{s,t}(r)\tilde{X}_{t}\,dt\bigg] = \tilde{E}\bigg[\int_{s}^{T}D_{s,t}(r)X_{t}G_{t}(\tau)\,dt\bigg] 
\end{equation*}
as soon as $X_{t} = \tilde{X}_{t}$ a.s.~on $\{\tau > t\}$ for all $t\in [0,T]$, according to Lemma~\ref{le:stopped integral}. This in turn shows that $\tilde{X}\in\tilde{\mathscr{S}}(r,\tau)$ $\Leftrightarrow$ $XG(\tau)\in\tilde{\mathscr{S}}(r)$.
\end{Remark}

Based on our definitions, the process $\null_{\mathrm{con}}\mathrm{CF}$ given by~\eqref{eq:cash flows and costs 1} that  models the discounted contractual derivative cash flows is integrable if our first model assumption holds:
\begin{enumerate}[label=(M.\arabic*), ref=M.\arabic*, leftmargin=\widthof{(M.1)} + \labelsep]
\item\label{con:m.1} The amount $\Phi(S,V)$ paid at maturity and the dividend rate $\pi$ lie in $\tilde{\mathscr{L}}(r,\tau)$ and $\tilde{\mathscr{S}}(r,\tau)$, respectively.
\end{enumerate}

By Remark~\ref{re:integrability of pre-default components}, if $G(\tau)$ were measurable, then, equivalently, we could have asked for $\Phi(S,V)G_{T}(\tau)\in\tilde{\mathscr{L}}(r)$ and $\pi G(\tau)\in\tilde{\mathscr{S}}(r)$. To consider the remaining quantities, let $\tilde{\mathscr{V}}\in\tilde{\mathscr{S}}$ be integrable up to time $\tau$ and $\mathscr{V}$ be a pre-default version of it.

We suppose that the collateral process, the funding amount, the hedging process and the close-out value possess c\a gl\a d paths or in short,
\begin{equation}\label{eq:con.1}
\text{$C(\mathscr{V})$, $\tilde{F}(\tilde{\mathscr{V}})$, $\tilde{H}(\tilde{\mathscr{V}})$ and $\varepsilon(\mathscr{V})$ are c\a gl\a d.}
\end{equation}
Then the processes $\null_{\mathrm{col}}\mathrm{C}(\mathscr{V})$, $\null_{\mathrm{fun}}\mathrm{C}(\tilde{\mathscr{V}})$ and $\null_{\mathrm{hed}}\mathrm{C}(\tilde{\mathscr{V}})$, introduced in~\eqref{eq:cash flows and costs 2}-\eqref{eq:cash flows and costs 4}, that model the collateral, funding and hedging costs and benefits, respectively, are integrable if
\begin{equation}\label{eq:con.2}
\text{$(c(\mathscr{V}) - r)C(\mathscr{V})$, $(\tilde{f}(\tilde{\mathscr{V}})-r)\tilde{F}(\tilde{\mathscr{V}})$, $(r-\tilde{h}(\tilde{\mathscr{V}}))\tilde{H}(\tilde{\mathscr{V}})$ lie in $\tilde{\mathscr{S}}(r,\tau)$.}
\end{equation}

For the integrability of $\null_{\mathrm{def}}\mathrm{CF}(\mathscr{V},\tilde{\mathscr{V}})$, given by~\eqref{eq:cash flows and costs 5} and modelling the cash flows on the default of one of the two parties, it suffices that the respective cash flows appearing on the default of $\mathcal{I}$ and $\mathcal{C}$ are elements of $\tilde{\mathscr{D}}(r,\tau)$. Namely,
\begin{equation}\label{eq:con.3}
\begin{split}
&\text{$(\varepsilon(\mathscr{V}) + \mathrm{LGD}_{\mathcal{I}}((\varepsilon - C)^{-}(\mathscr{V}) + \tilde{F}^{+}(\tilde{\mathscr{V}}))\mathbbm{1}_{\{\mathcal{I}=\mathcal{B}\}})\mathbbm{1}_{\{\tau_{\mathcal{I}} < \tau_{\mathcal{C}}\}}$}\\
&\text{and $(\varepsilon -\mathrm{LGD}_{\mathcal{C}}(\varepsilon - C)^{+})(\mathscr{V})\mathbbm{1}_{\{\tau_{\mathcal{I}} > \tau_{\mathcal{C}}\}}$ belong to $\tilde{\mathscr{D}}(r,\tau)$}.
\end{split}
\end{equation}
 In fact, as condition~\eqref{eq:market condition 2} states that $\tau_{\mathcal{I}}\neq \tau_{\mathcal{C}}$ a.s.~on $\{\tau < \infty\}$, we readily check that the random variable $|\null_{\mathrm{def}}\mathrm{CF}_{s}(\mathscr{V},\tilde{\mathscr{V}})|$ is bounded by
\begin{align*}
&\sup_{t\in ]s,T[}D_{s,t}(r)|\varepsilon_{t}(\mathscr{V}) + \mathrm{LGD}_{\mathcal{I}}((\varepsilon_{t} - C_{t})^{-}(\mathscr{V}) + \tilde{F}_{t}^{+}(\tilde{\mathscr{V}}))\mathbbm{1}_{\{\mathcal{I}=\mathcal{B}\}}|\mathbbm{1}_{\{\tau_{\mathcal{I}} < \tau_{\mathcal{C}}\}}\\ 
&\quad + \sup_{t\in ]s,T[} D_{s,t}(r)|\varepsilon_{t} -\mathrm{LGD}_{\mathcal{C}}(\varepsilon_{t} - C_{t})^{+}|(\mathscr{V})\mathbbm{1}_{\{\tau_{\mathcal{I}} > \tau_{\mathcal{C}}\}}
\end{align*}
a.s.~on $\{s < \tau < T\}$ for all $s\in [0,T[$, which entails the asserted integrability. Let us now suppose that~\eqref{con:m.1} holds. For $\tilde{\mathscr{V}}$ to be the value process of a hedging strategy of the contract under $\tilde{P}$, we require that~\eqref{eq:con.1}-\eqref{eq:con.3} be satisfied.

Further, we stipulate that $\tilde{\mathscr{V}}_{s}$ agrees with the conditional expectation of the sum of the net present values of all cash flows, costs and benefits relative to the current available market information under $\tilde{P}$ for each $s\in [0,T]$. Namely,
\begin{equation}\label{eq:valuation}
\tilde{\mathscr{V}}_{s} = \tilde{E}\big[\null_{\mathrm{con}}\mathrm{CF}_{s} - \null_{\mathrm{col}}\mathrm{C}_{s}(\mathscr{V}) - \null_{\mathrm{fun}}\mathrm{C}_{s}(\tilde{\mathscr{V}}) - \null_{\mathrm{hed}}\mathrm{C}_{s}(\tilde{\mathscr{V}}) + \null_{\mathrm{def}}\mathrm{CF}_{s}(\mathscr{V},\tilde{\mathscr{V}})\big|\tilde{\mathscr{F}}_{s}\big]\quad\text{a.s.}
\end{equation}
for any $s\in [0,T]$ and from our considerations we infer that $\tilde{\mathscr{V}}$ is necessarily $\tilde{P}$-integrable. In particular, the terminal value condition $\tilde{\mathscr{V}}_{T} = \null_{\mathrm{con}}\mathrm{CF}_{T} = \Phi(S,V)$ a.s.~must hold.

This implicit conditional representation refines the \emph{valuation equation}~(1) in~\cite{BriFraPal19}, which is build on the valuation problems in~\cite{PalPerBri11} and~\cite{PalPerBri12}. As immediate consequence of~\eqref{eq:pre-default version}, the pre-default version $\mathscr{V}$ of the value process $\tilde{\mathscr{V}}$ satisfies
\begin{equation}\label{eq:preliminary pre-default valuation}
\mathscr{V}_{s}G_{s}(\tau) = \tilde{E}\big[\null_{\mathrm{con}}\mathrm{CF}_{s} - \null_{\mathrm{col}}\mathrm{C}_{s}(\mathscr{V}) - \null_{\mathrm{fun}}\mathrm{C}_{s}(\tilde{\mathscr{V}}) - \null_{\mathrm{hed}}\mathrm{C}_{s}(\tilde{\mathscr{V}}) + \null_{\mathrm{def}}\mathrm{CF}_{s}(\mathscr{V},\tilde{\mathscr{V}})\big|\mathscr{F}_{s}\big]\quad\text{a.s.}
\end{equation}
for all $s\in [0,T]$, as the quantities $\null_{\mathrm{con}}\mathrm{CF}_{s} $, $\null_{\mathrm{col}}\mathrm{C}_{s}(\mathscr{V})$, $\null_{\mathrm{fun}}\mathrm{C}_{s}(\tilde{\mathscr{V}})$, $\null_{\mathrm{hed}}\mathrm{C}_{s}(\tilde{\mathscr{V}})$ and $\null_{\mathrm{def}}\mathrm{CF}_{s}(\mathscr{V},\tilde{\mathscr{V}})$ vanish on $\{\tau \leq s\}$. Thus, to derive a valuation equation for $\mathscr{V}$ restricted to default-free information, we will replace the $\mathscr{F}_{T}\vee\mathscr{H}_{T}^{\tau_{\mathcal{I}}}\vee\mathscr{H}_{T}^{\tau_{\mathcal{C}}}$-measurable random variables
\begin{equation*}
\null_{\mathrm{con}}\mathrm{CF}_{s},\quad \null_{\mathrm{col}}\mathrm{C}_{s}(\mathscr{V}),\quad\null_{\mathrm{fun}}\mathrm{C}_{s}(\tilde{\mathscr{V}}),\quad \null_{\mathrm{hed}}\mathrm{C}_{s}(\tilde{\mathscr{V}})\quad\text{and}\quad \null_{\mathrm{def}}\mathrm{CF}_{s}(\mathscr{V},\tilde{\mathscr{V}})
\end{equation*}
within the conditional expectation in~\eqref{eq:preliminary pre-default valuation} by $\mathscr{F}_{T}$-measurable ones that may merely depend on $\mathscr{V}$. For this purpose, we will use the probabilistic results from Section~\ref{se:2.2} and require a set of model assumptions:
\begin{enumerate}[label=(M.\arabic*), ref=M.\arabic*, leftmargin=\widthof{(M.4)} + \labelsep]
\setcounter{enumi}{1}
\item\label{con:m.2} $G(\tau_{\mathcal{I}})$ and $ G(\tau_{\mathcal{C}})$ are continuous and of finite variation, $\tau_{\mathcal{I}}$ and $\tau_{\mathcal{C}}$ are conditionally independent relative to $(\mathscr{F}_{t})_{t\in [0,T]}$ under $\tilde{P}$ and $G(\tau) = G(\tau_{\mathcal{I}})G(\tau_{\mathcal{C}})$.
\item\label{con:m.3} There are two $(\mathscr{F}_{t})_{t\in [0,T]}$-time-dependent random functionals $F$ and $H$ on $\mathscr{S}$ such that $F(X)$ and $H(X)$ serve as pre-default versions of $\tilde{F}(\tilde{X})$ and $\tilde{H}(\tilde{X})$, respectively, for any $\tilde{X}\in\tilde{\mathscr{S}}$ that is integrable up to time $\tau$ with pre-default version $X$.
\item\label{con:m.4} The funding rates $\null_{+}\tilde{f}$, $\null_{-}\tilde{f}$ and the hedging rates $\null_{+}\tilde{h}$, $\null_{-}\tilde{h}$ are integrable up to time $\tau$ and admit $(\mathscr{F}_{t})_{t\in [0,T]}$-progressively measurable pre-default versions $\null_{+}f$, $\null_{-}f$ and $\null_{+}h$, $\null_{-}h$, respectively, with integrable paths.
\end{enumerate}

\begin{Remark}\label{re:identity}
The identity in~\eqref{con:m.2} simply states that $G(\tau)$ and $G(\tau_{\mathcal{I}})G(\tau_{\mathcal{C}})$ are not only modifications of each other, but in fact equal. This ensures that all the paths of $G(\tau)$ are continuous and of finite variation.
\end{Remark}

Under~\eqref{con:m.3} and~\eqref{con:m.4}, we may define two time-dependent random functionals $f$ and $h$ on $\mathscr{S}$ relative to $(\mathscr{F}_{t})_{t\in [0,T]}$ by $f_{t}(X) := \null_{+}f_{t}\mathbbm{1}_{]0,\infty[}(F_{t}(X)) + \null_{-}f_{t}\mathbbm{1}_{]-\infty,0[}(F_{t}(X))$ and
\begin{equation*}
h_{t}(X) := \null_{+}h_{t}\mathbbm{1}_{]0,\infty[}\big(H_{t}(X)\big) + \null_{-}h_{t}\mathbbm{1}_{]-\infty,0[}\big(H_{t}(X)\big).
\end{equation*}
Then $\tilde{f}(\tilde{X})$ and $\tilde{h}(\tilde{X})$ admit $f(X)$ and $h(X)$ as pre-default versions, respectively, for any $\tilde{X}\in\tilde{\mathscr{S}}$ that is integrable up to time $\tau$ with pre-default version $X$. Further, $f(X)$ and $h(X)$ are $(\mathscr{F}_{t})_{t\in [0,T]}$-progressively measurable if $F(X)$ and $H(X)$ are.

Now let the pre-default funding amount $F(\mathscr{V})$ and the pre-default hedging process $H(\mathscr{V})$ be c\a gl\a d. In this case,~\eqref{eq:con.1} entails that the following path regularity condition for the pre-default version $\mathscr{V}$ of $\tilde{\mathscr{V}}$ holds:
\begin{enumerate}[label=(C.\arabic*), ref=C.\arabic*, leftmargin=\widthof{(C.1)} + \labelsep]
\item\label{con:c.1} $C(\mathscr{V})$, $F(\mathscr{V})$, $H(\mathscr{V})$ and $ \varepsilon(\mathscr{V})$ are c\a gl\a d.
\end{enumerate}

Let in addition~\eqref{con:m.2} be satisfied. Then Remark~\ref{re:integrability of pre-default components} shows that~\eqref{eq:con.2} is valid if and only if the following integrability condition is valid:
\begin{enumerate}[label=(C.\arabic*), ref=C.\arabic*, leftmargin=\widthof{(C.1)} + \labelsep]
\setcounter{enumi}{1}
\item\label{con:c.2} The product of $G(\tau)$ with any of the processes $(c(\mathscr{V})-r)C(\mathscr{V})$, $(f(\mathscr{V})-r)F(\mathscr{V})$ and $(r-h(\mathscr{V}))H(\mathscr{V})$ belongs to $\tilde{\mathscr{S}}(r)$.
\end{enumerate}

To handle the conditional expectation of $\null_{\mathrm{def}}\mathrm{CF}_{s}(\mathscr{V},\tilde{\mathscr{V}})$ relative to $\mathscr{F}_{s}$ in~\eqref{eq:preliminary pre-default valuation} for any $s\in [0,T]$, we need another integrability condition that involves the variation process $V(\tau_{i})$ of $G(\tau_{i})$ for both $i\in\{\mathcal{I},\mathcal{C}\}$. In this regard, we recall that $V(\tau_{i}) = 1- G(\tau_{i})\in [0,1]$ if $G(\tau_{i})$ is decreasing, as in Example~\ref{ex:default time specification} below.
\begin{enumerate}[label=(C.\arabic*), ref=C.\arabic*,leftmargin=\widthof{(C.2)} + \labelsep]
\setcounter{enumi}{2}
\item\label{con:c.3} $\sup_{t\in ]s,T[} D_{s,t}(r)|\varepsilon_{t} + \mathrm{LGD}_{\mathcal{I}}((\varepsilon_{t} - C_{t})^{-} + F_{t}^{+})\mathbbm{1}_{\{\mathcal{I}=\mathcal{B}\}}|(\mathscr{V})G_{t}(\tau_{\mathcal{C}})(V_{T}(\tau_{\mathcal{I}})-V_{s}(\tau_{\mathcal{I}}))$ and
\begin{equation*}
\sup_{t\in ]s,T[} D_{s,t}(r)|\varepsilon_{t} -\mathrm{LGD}_{\mathcal{C}}(\varepsilon_{t} - C_{t})^{+}|(\mathscr{V})G_{t}(\tau_{\mathcal{I}})(V_{T}(\tau_{\mathcal{C}})-V_{s}(\tau_{\mathcal{C}}))
\end{equation*}
are $\tilde{P}$-integrable for any $s\in [0,T[$.
\end{enumerate}

For a clear and concise overview, let us summarise all appearing quantities by defining three $(\mathscr{F}_{t})_{t\in [0,T]}$-time-dependent random functionals $\null_{0}\B$, $\null_{\mathcal{I}}\B$ and $\null_{\mathcal{C}}\B$ on $\mathscr{S}$ via
\begin{equation}\label{eq:inhomogeneity process components}
\begin{split}
\null_{0}\B_{t}(X) &:= \pi_{t} - (c_{t}(X) - r_{t})C_{t}(X) - (f_{t}(X) - r_{t})F_{t}(X) - (r_{t} - h_{t}(X))H_{t}(X),\\
\null_{\mathcal{I}}\B_{t}(X) &:= \varepsilon_{t}(X) + \mathrm{LGD}_{\mathcal{I}}((\varepsilon_{t} - C_{t})^{-} + F_{t}^{+})(X)\mathbbm{1}_{\{\mathcal{I}=\mathcal{B}\}}\quad\text{and}\\
\null_{\mathcal{C}}\B_{t}(X) &:= \varepsilon_{t}(X) -\mathrm{LGD}_{\mathcal{C}}(\varepsilon_{t} - C_{t})^{+}(X).
\end{split}
\end{equation}
Then $\null_{0}\B(X)$ is $(\mathscr{F}_{t})_{t\in [0,T]}$-progressively measurable for any $X\in\mathscr{S}$ as soon as $C(X)$, $F(X)$ and $H(X)$ are. Thus,~\eqref{con:c.2} implies $\null_{0}\B(\mathscr{V})G(\tau)\in\tilde{\mathscr{S}}(r)$. Further, $\null_{\mathcal{I}}\B(X)$ and $\null_{\mathcal{C}}\B(X)$ are c\a gl\a d if $C(X)$, $F(X)$ and $\varepsilon(X)$ are. In this case, the two Riemann-Stieltjes integrals
\begin{equation*}
\int_{s}^{T}D_{s,t}(r)|\null_{\mathcal{I}}\B_{t}(X)|G_{t}(\tau_{\mathcal{C}})\,dV_{t}(\tau_{\mathcal{I}})\quad\text{and}\quad \int_{s}^{T}D_{s,t}(r)|\null_{\mathcal{C}}\B_{t}(X)|G_{t}(\tau_{\mathcal{I}})\,dV_{t}(\tau_{\mathcal{C}})
\end{equation*}
are finite for each $s\in [0,T[$, and for $X=\mathscr{V}$ each of these integrals is bounded by the respective random variable in~\eqref{con:c.3}, which ensures their $\tilde{P}$-integrability.

Consequently, we may introduce an $(\mathscr{F}_{t})_{t\in [0,T]}$-time-dependent random functional $A$ on the set of all $X\in\mathscr{S}$ for which $C(X)$, $F(X)$, $H(X)$ and $\varepsilon(X)$ are c\a gl\a d by
\begin{equation}\label{eq:finite variation process}
\begin{split}
A_{t}(X) &:= \int_{0}^{t}\null_{0}\B_{s}(X)G_{s}(\tau)\,ds - \int_{0}^{t}\null_{\mathcal{I}}\B_{s}(X)G_{s}(\tau_{\mathcal{C}})\,dG_{s}(\tau_{\mathcal{I}})\\
&\quad - \int_{0}^{t}\null_{\mathcal{C}}\B_{s}(X)G_{s}(\tau_{\mathcal{I}})\,dG_{s}(\tau_{\mathcal{C}}).
\end{split}
\end{equation}
We readily see that $A(X)$ is a continuous process of finite variation for each process $X$ in its domain and the Riemann-Stieltjes integral $\int_{s}^{T}D_{s,t}(r)\,dA_{t}(\mathscr{V})$ is $\tilde{P}$-integrable for every $s\in [0,T]$ if~\eqref{con:c.2} and~\eqref{con:c.3} are valid.

Based on all these measurability, path regularity and integrability considerations, we may now rewrite the conditional expectations appearing in~\eqref{eq:preliminary pre-default valuation} as follows.

\begin{Proposition}\label{pr:pre-default valuation}
Let~\eqref{con:m.1}-\eqref{con:m.4} hold and $\tilde{\mathscr{V}}\in\tilde{\mathscr{S}}$ be integrable with pre-default version $\mathscr{V}$ such that $\tilde{F}(\tilde{\mathscr{V}})$, $\tilde{H}(\tilde{\mathscr{V}})$ are c\a gl\a d and~\eqref{eq:con.3} and~\eqref{con:c.1}-\eqref{con:c.3} hold. Then
\begin{equation}\label{eq:conditional expectation representation}
\begin{split}
\tilde{E}\big[\null_{\mathrm{con}}\mathrm{CF}_{s} - \null_{\mathrm{col}}\mathrm{C}_{s}(\mathscr{V}) &- \null_{\mathrm{fun}}\mathrm{C}_{s}(\tilde{\mathscr{V}})  - \null_{\mathrm{hed}}\mathrm{C}_{s}(\tilde{\mathscr{V}}) + \null_{\mathrm{def}}\mathrm{CF}_{s}(\mathscr{V},\tilde{\mathscr{V}})\big|\mathscr{F}_{s}\big]\\ &= \tilde{E}\bigg[D_{s,T}(r)\Phi(S,V)G_{T}(\tau) + \int_{s}^{T}D_{s,t}(r)\,dA_{t}(\mathscr{V})\,\bigg|\,\mathscr{F}_{s}\bigg]
\end{split}
\end{equation}
a.s.~for each $s\in [0,T]$.
\end{Proposition}

This result leads us to a valuation equation involving default-free information only. Namely, by a solution to the \emph{pre-default valuation equation}
\begin{equation}\label{eq:pre-default valuation}\tag{VE}
\mathscr{V}_{s}G_{s}(\tau) = \tilde{E}\bigg[D_{s,T}(r)\Phi(S,V)G_{T}(\tau) + \int_{s}^{T}D_{s,t}(r)\,dA_{t}(\mathscr{V})\,\bigg|\,\mathscr{F}_{s}\bigg]\quad\text{a.s.}
\end{equation}
for $s\in [t_{0},T]$ we shall mean a process $\mathscr{V}\in\mathscr{S}$ such that the path regularity condition~\eqref{con:c.1} and the two integrability conditions~\eqref{con:c.2} and~\eqref{con:c.3} hold and the almost sure identity in~\eqref{eq:pre-default valuation} is satisfied for each $s\in [t_{0},T]$. 

In this case, $\mathscr{V}_{s}G_{s}(\tau)$ is necessarily integrable and $\mathscr{V}_{T} = \Phi(S,V)$ a.s.~on $\{G_{T}(\tau) >0 \}$. Thus, we obtain a \emph{martingale characterisation} for any such \emph{pre-default value process}.

\begin{Proposition}\label{pr:martingale characterisation}
Assume that~\eqref{con:m.1}-\eqref{con:m.4} are valid and $\mathscr{V}\in\mathscr{S}$ satisfies~\eqref{con:c.1}-\eqref{con:c.3}. Then the $(\mathscr{F}_{t})_{t\in [0,T]}$-adapted continuous process
\begin{equation}\label{eq:finite variation process 2}
[0,T]\times\Omega\rightarrow\mathbb{R},\quad (t,\omega)\mapsto\int_{0}^{t}D_{0,s}(r)(\omega)\,dA_{s}(\mathscr{V})(\omega)
\end{equation}
of finite variation is integrable. Moreover, $\mathscr{V}$ solves~\eqref{eq:pre-default valuation} if and only if $\null_{\mathscr{V}}M\in\mathscr{S}$ defined via
\begin{equation}\label{eq:pre-default martingale}
\null_{\mathscr{V}}M_{t} := D_{0,t}(r)\mathscr{V}_{t}G_{t}(\tau) + \int_{0}^{t}D_{0,s}(r)\,dA_{s}(\mathscr{V})
\end{equation}
is an $(\mathscr{F}_{t})_{t\in [t_{0},T]}$-martingale under $\tilde{P}$ and $\mathscr{V}_{T} = \Phi(S,V)$ a.s.~on $\{G_{T}(\tau) > 0\}$.
\end{Proposition}

For an \emph{implicit backward stochastic integral representation of any pre-default value process} we assume until the end of this section that $(\Omega,\mathscr{F},(\mathscr{F}_{t})_{t\in [0,T]},\tilde{P})$ satisfies the usual conditions.

\begin{Proposition}\label{pr:backward stochastic representation}
Let~\eqref{con:m.1}-\eqref{con:m.4} be valid and~\eqref{con:c.1}-\eqref{con:c.3} hold for $\mathscr{V}\in\mathscr{S}$. Then $\mathscr{V}G(\tau)$ is a continuous $(\mathscr{F}_{t})_{t\in [t_{0},T]}$-semimartingale if and only if $\null_{\mathscr{V}}M$ is. In this case,
\begin{equation}\label{eq:auxiliary backward integral representation}
\mathscr{V}_{s}G_{s}(\tau) =  \mathscr{V}_{T}G_{T}(\tau) + \int_{s}^{T}\big(dA_{t}(\mathscr{V}) - r_{t}\mathscr{V}_{t}G_{t}(\tau)\,dt\big) - \int_{s}^{T}D_{0,t}(-r)\,d\null_{\mathscr{V}}M_{t}
\end{equation}
for any $s\in [t_{0},T]$ a.s. If in addition $G(\tau) > 0$, then $\null_{\mathscr{V}}M$ is, up to indistinguishability, the unique continuous $(\mathscr{F}_{t})_{t\in [t_{0},T]}$-semimartingale satisfying
\begin{equation}\label{eq:backward stochastic representation}
\begin{split}
\mathscr{V}_{s} &= \mathscr{V}_{T} + \int_{s}^{T}\big(\null_{0}\B_{t}(\mathscr{V}) - r_{t}\mathscr{V}_{t}\big)\,dt - \int_{s}^{T}\frac{\null_{\mathcal{I}}\B_{t}(\mathscr{V}) - \mathscr{V}_{t}}{G_{t}(\tau_{\mathcal{I}})}\,dG_{t}(\tau_{\mathcal{I}})\\
&\quad - \int_{s}^{T}\frac{\null_{\mathcal{C}}\B_{t}(\mathscr{V}) - \mathscr{V}_{t}}{G_{t}(\tau_{\mathcal{C}})}\,dG_{t}(\tau_{\mathcal{C}}) - \int_{s}^{T}\frac{D_{0,t}(-r)}{G_{t}(\tau)}\,d\null_{\mathscr{V}}M_{t}
\end{split}
\end{equation}
for all $s\in [t_{0},T]$ a.s.~and $\null_{\mathscr{V}}M_{t_{1}} = D_{0,t_{1}}(r)\mathscr{V}_{t_{1}}G_{t_{1}}(\tau) + \int_{0}^{t_{1}}D_{0,t}(r)\,dA_{t}(\mathscr{V})$ a.s.~for some $t_{1}\in [t_{0},T]$.
\end{Proposition}

As a consequence of the preceding two results, we are able to \emph{characterise pre-default value processes that are semimartingales}.

\begin{Corollary}\label{co:semimartingale characterisation}
Let~\eqref{con:m.1}-\eqref{con:m.4} hold, $G(\tau) > 0$, $\mathscr{V}\in\mathscr{S}$ be continuous and~\eqref{con:c.1}-\eqref{con:c.3} be valid. Then $\mathscr{V}$ is an $(\mathscr{F}_{t})_{t\in [t_{0},T]}$-semimartingale solving~\eqref{eq:pre-default valuation} if and only if
\begin{equation*}
\text{$D_{0,t_{0}}(r)\mathscr{V}_{t_{0}}G_{t_{0}}(\tau)$ is $\tilde{P}$-integrable},\quad \mathscr{V}_{T} = \Phi(S,V)\quad\text{a.s.}
\end{equation*}
and there is a continuous $(\mathscr{F}_{t})_{t\in [t_{0},T]}$-martingale $M$ that takes the role of $\null_{\mathscr{V}}M$ in~\eqref{eq:backward stochastic representation}. If this is the case, then $M_{t} - M_{t_{0}} = \null_{\mathscr{V}}M_{t} - \null_{\mathscr{V}}M_{t_{0}}$ for all $t\in [t_{0},T]$ a.s.
\end{Corollary}

In addition to our four model assumptions let $G(\tau_{\mathcal{I}})$ and $G(\tau_{\mathcal{C}})$ be not only continuous and of finite variation but absolutely continuous and positive. Then the same holds for $G(\tau)$ and any continuous $\mathscr{V}\in\mathscr{S}$ satisfying~\eqref{con:c.1}-\eqref{con:c.3} solves~\eqref{eq:pre-default valuation} if and only if
\begin{equation}\label{eq:specific representation}
\begin{split}
\mathscr{V}_{s} &= \tilde{E}\bigg[D_{s,T}(r)\Phi(S,V)\frac{G_{T}(\tau)}{G_{s}(\tau)}\,\bigg|\,\mathscr{F}_{s}\bigg]\\
&\quad + \tilde{E}\bigg[\int_{s}^{T}D_{s,t}(r)\frac{G_{t}(\tau)}{G_{s}(\tau)}\bigg(\B_{t}(\mathscr{V}) + \bigg(r_{t} - \frac{\dot{G}_{t}(\tau)}{G_{t}(\tau)}\bigg)\mathscr{V}_{t}\bigg)\,dt\,\bigg|\,\mathscr{F}_{s}\bigg]
\end{split}
\end{equation}
for any $s\in [t_{0},T]$ with the $(\mathscr{F}_{t})_{t\in [0,T]}$-time-dependent random functional $\B$ defined on the set of all continuous $X\in\mathscr{S}$ for which $C(X)$, $F(X)$, $H(X)$ and $\varepsilon(X)$ are c\a gl\a d by
\begin{equation}\label{eq:inhomogeneity process}
\B_{t}(X) := \null_{0}\B_{t}(X) - r_{t}X_{t} - \frac{\dot{G}_{t}(\tau_{\mathcal{I}})}{G_{t}(\tau_{\mathcal{I}})}(\null_{\mathcal{I}}\B_{t}(X) - X_{t}) - \frac{\dot{G}_{t}(\tau_{\mathcal{C}})}{G_{t}(\tau_{\mathcal{C}})}(\null_{\mathcal{C}}\B_{t}(X) - X_{t}).
\end{equation}
Moreover, if $\mathscr{V}$, or equivalently, $\null_{\mathscr{V}}M$ is an $(\mathscr{F}_{t})_{t\in [t_{0},T]}$-semimartingale, then we may rewrite the implicit backward stochastic integral representation~\eqref{eq:backward stochastic representation} in the form
\begin{equation}\label{eq:specific representation 2}
\mathscr{V}_{s} = \mathscr{V}_{T} + \int_{s}^{T}\B_{t}(\mathscr{V})\,dt - \int_{s}^{T}\frac{D_{0,t}(-r)}{G_{t}(\tau)}\,d\null_{\mathscr{V}}M_{t}\quad\text{for all $s\in [t_{0},T]$ a.s.}
\end{equation}

\begin{Example}\label{ex:default time specification}
For both $i\in\{\mathcal{I},\mathcal{C}\}$ let $\lambda^{(i)}$ be an $\R_{+}$-valued $(\mathscr{F}_{t})_{t\in [0,T]}$-progressively measurable process with integrable paths such that every $\omega\in\Omega$ satisfies
\begin{equation*}
\int_{0}^{t_{\omega}}\lambda_{s}^{(i)}(\omega)\,ds > 0\quad\text{for some $t_{\omega}\in ]0,T]$},
\end{equation*}
which holds if $\lambda^{(i)} > 0$, for instance. Further, let $\xi_{i}$ be an $]0,\infty[$-valued $\tilde{\mathscr{F}}_{0}$-measurable random variable that is gamma distributed with shape $\alpha_{i} > 0$ and rate $\beta_{i} > 0$ such that
\begin{equation*}
\tau_{i} = \inf\bigg\{t\in [0,T]\,\bigg|\, \int_{0}^{t}\lambda_{s}^{(i)}\,ds \geq \xi_{i}\bigg\}.
\end{equation*}
We suppose that $(\xi_{\mathcal{I}},\xi_{\mathcal{C}})$ is independent of $\mathscr{F}_{T}$ and $\xi_{\mathcal{I}}$ and $\xi_{\mathcal{C}}$ are independent. Then Lemma~\ref{le:hitting time} and Example~\ref{ex:gamma distributed hitting times} show that~\eqref{eq:market condition 2} and~\eqref{con:m.2} hold for  $G(\tau) = G(\tau_{\mathcal{I}})G(\tau_{\mathcal{C}})$ and both $G(\tau_{\mathcal{I}})$ and $G(\tau_{\mathcal{C}})$ are absolutely continuous and positive. Moreover,
\begin{equation*}
-\frac{\dot{G}_{t}(\tau_{i})}{G_{t}(\tau_{i})} = \frac{\beta_{i}^{\alpha_{i}}\lambda_{t}^{(i)}}{\gamma\big(\alpha_{i},\beta_{i}\int_{0}^{t}\lambda_{s}^{(i)}\,ds)}\bigg(\int_{0}^{t}\lambda_{s}^{(i)}\,ds\bigg)^{\alpha_{i} - 1}\exp\bigg(-\beta_{i}\int_{0}^{t}\lambda_{s}^{(i)}\,ds\bigg)
\end{equation*}
for a.e.~$t\in [0,T]$ for both $i\in\{\mathcal{I},\mathcal{C}\}$, where $\gamma$ is the upper incomplete gamma function, and this formula this reduces to $-\dot{G}(\tau_{i})/G(\tau_{i}) = \beta_{i}\lambda^{(i)}$ a.e.~when $\alpha_{i}=1$. Thus, in the case $\alpha_{\mathcal{I}}=\alpha_{\mathcal{C}}=1$ we have
\begin{align*}
\mathrm{B}_{t}(X) &= \null_{0}\B_{t}(X) - r_{t}X_{t} +\big(\beta_{\mathcal{I}}\lambda_{t}^{(\mathcal{I})} + \beta_{\mathcal{C}}\lambda_{t}^{(\mathcal{C})}\big)\big(\varepsilon_{t}(X) - X_{t}\big)\\
&\quad + \beta_{\mathcal{I}}\lambda_{t}^{(\mathcal{I})}\mathrm{LGD}_{\mathcal{I}}((\varepsilon_{t} - C_{t})^{-} + F_{t}^{+})(X)\mathbbm{1}_{\{\mathcal{I} =\mathcal{B}\}} - \beta_{\mathcal{C}}\lambda_{t}^{(\mathcal{C})}\mathrm{LGD}_{\mathcal{C}}(\varepsilon_{t} - C_{t})^{+}(X)
\end{align*}
for a.e.~$t\in [0,T]$ and any $X\in\mathscr{S}$ for which $C(X)$, $F(X)$, $H(X)$ and $\varepsilon(X)$ are c\a gl\a d. In particular, if $\xi_{\mathcal{I}}$ and $\xi_{\mathcal{C}}$ are exponentially distributed with mean one and the \emph{financing hypothesis}
\begin{equation}\label{eq:financing hypothesis}
X_{t} = C_{t}(X) + F_{t}(X)\quad\text{for all $(t,X)\in [0,T]\times\mathscr{S}$}
\end{equation}
holds, then~\eqref{eq:specific representation} and~\eqref{eq:specific representation 2} yield the respective identities~(5) and~(6) for the pre-default value process in~\cite{BriFraPal19}, where $(\mathscr{F}_{t})_{t\in [0,T]}$ is the augmented filtration of a standard Brownian motion and the martingale representation theorem may be applied.
\end{Example}

\section{A parabolic equation for the pre-default valuation}\label{se:4}

\subsection{A general stochastic volatility model}\label{se:4.1}

In the sequel, let the filtered probability space $(\Omega,\mathscr{F},(\mathscr{F}_{t})_{t\in [0,T]},P)$ satisfy the usual conditions and suppose that there are two standard $(\mathscr{F}_{t})_{t\in [0,T]}$-Brownian motions $\hat{W}$ and $\tilde{W}$ with covariation
\begin{equation*}
\langle \hat{W},\tilde{W}\rangle_{t} = \int_{0}^{t}\rho(s)\,ds \quad\text{for all $t\in [0,T]$ a.s},
\end{equation*}
where $\rho:[0,T]\rightarrow ]-1,1[$ is a measurable function satisfying $\int_{0}^{T}(1-\rho(s)^{2})^{-1}\,ds < \infty$. Let $b:[0,T]\rightarrow\mathbb{R}$ and $\zeta,\eta,\theta:[0,T]\times\mathbb{R}\rightarrow\mathbb{R}$ be measurable and consider the two-dimensional SDE starting at time $t_{0}\in [0,T]$:
\begin{equation}\label{eq:stochastic volatility model}
\begin{split}
dS_{t} &= b(t)S_{t}\,dt + \theta(t,V_{t})S_{t}\,d\hat{W}_{t},\\
dV_{t} &= \zeta(t,V_{t})\,dt + \eta(t,V_{t})\,d\tilde{W}_{t}
\end{split}
\end{equation}
for $t\in [t_{0},T]$. Given any weak solution $(S,V)$, we will interpret $S$ as \emph{price process of the risky asset} and $V$ as \emph{quasi variance} or \emph{quasi squared volatility process} influencing $S$ via the function $\theta$, which will satisfy the $1/2$-H\oe lder continuity condition in~\eqref{con:v.2}.

Under a weak integrability condition, the unique solution to the linear SDE in~\eqref{eq:stochastic volatility model} is readily recalled, by using stochastic exponentials for local martingales. For this purpose, let $V$ be an adapted continuous process with positive paths.
\begin{enumerate}[label=(V.\arabic*), ref=V.\arabic*, leftmargin=\widthof{(V.1)} + \labelsep]
\item\label{con:v.1} $b$ is integrable and for any compact set $K$ in $]0,\infty[$ there is $k_{\theta}\in\mathscr{L}^{2}(\mathbb{R}_{+})$ such that $|\theta(\cdot,v)| \leq k_{\theta}$ for each $v\in K$ a.e.
\end{enumerate}

\begin{Lemma}\label{le:price process representation}
Let~\eqref{con:v.1} hold and $\chi$ be an $\mathscr{F}_{t_{0}}$-measurable random variable. Then the first SDE in~\eqref{eq:stochastic volatility model} admits a unique solution $S$ such that $S_{t_{0}} = \chi$ a.s. In fact,
\begin{equation}\label{eq:stochastic exponential}
S_{t} = \chi e^{\int_{t_{0}}^{t}\theta(s,V_{s})\,d\hat{W}_{s} + \int_{t_{0}}^{t}b(s) - \frac{1}{2}\theta(s,V_{s})^{2}\,ds}
\end{equation}
for any $t\in [t_{0},T]$ a.s. In particular, if $\chi$ and $\exp(\frac{1}{2}\int_{t_{0}}^{T}\theta(s,V_{s})^{2}\,ds)$ are integrable, then so is $S$ and $E[S_{t}] = E[\chi]\exp(\int_{t_{0}}^{t}b(s)\,ds)$ for all $t\in [t_{0},T]$.
\end{Lemma}

In view of the preceding lemma, for any $\mathscr{F}_{t_{0}}$-measurable positive random variable $\chi$, there exists a unique solution $S$ to the first SDE in~\eqref{eq:stochastic volatility model} with positive paths such that $S_{t_{0}} = \chi$ a.s. Then the \emph{logarithmised process} $X:=\log(S)$ satisfies
\begin{equation}\label{eq:log-price dynamics}
dX_{t} = \big(b(t) - (1/2)\theta(t,V_{t})^{2}\big)\,dt + \theta(t,V_{t})\,d\hat{W}_{t}\quad\text{for $t\in [t_{0},T]$}
\end{equation}
and it is the unique strong solution to~\eqref{eq:log-price dynamics} with $X_{t_{0}} = \log(\chi)$ a.s. \emph{Growth and comparison estimates} for solutions to such SDEs with different controlling processes follow from a weak integrability and a H\oe lder condition on the function $\theta$:
\begin{enumerate}[label=(V.\arabic*), ref=V.\arabic*, leftmargin=\widthof{(V.2)} + \labelsep]
\setcounter{enumi}{1}
\item\label{con:v.2} There are $v_{0} > 0$ and  $\lambda_{\theta}\in\mathscr{L}^{2}(\mathbb{R}_{+})$ such that $\theta(\cdot,v_{0})$ is square-integrable and $|\theta(\cdot,v) - \theta(\cdot,\tilde{v})| \leq \lambda_{\theta}|v-\tilde{v}|^{1/2}$ for all $v,\tilde{v} > 0$ a.e.
\end{enumerate}

This requirement implies the sublinear growth condition: $|\theta(\cdot,v)| \leq k_{\theta} + \lambda_{\theta}|v|^{1/2}$ for any $v > 0$ a.e. with $k_{\theta} := |\theta(\cdot,v_{0})| + \lambda_{\theta}|v_{0}|^{1/2}$. Thus, if $X_{t_{0}}$ and $b$ were integrable, then the inequalities of Burkholder-Davis-Gundy, Minkowski and Young yield that
\begin{align}\nonumber
E\bigg[\sup_{s\in [t_{0},t]} |X_{s}|\bigg]  - E\big[|X_{t_{0}}|\big] &\leq \int_{t_{0}}^{t}|b(s)| + \frac{1}{2}E\big[\theta(s,V_{s})^{2}\big]\,ds + 2\bigg(\int_{t_{0}}^{t}E\big[\theta(s,V_{s})^{2}\big]\,ds\bigg)^{\frac{1}{2}}\\\label{eq:log-price growth estimate}
&\leq c_{0}(t_{0},t) + c_{1}(t_{0},t)\sup_{s\in [t_{0},t]} E\big[V_{s}\big]
\end{align}
for any $t\in [t_{0},T]$ with the two $\R_{+}$-valued continuous functions $c_{0}$ and $c_{1}$ defined on the set of all $(t_{1},t)\in [0,T]\times [0,T]$ with $t_{1}\leq t$ via
\begin{equation*}
\begin{split}
c_{0}(t_{1},t) &:= \int_{t_{1}}^{t}|b(s)| + k_{\theta}(s)^{2}\,ds + 2\bigg(\int_{t_{1}}^{t}k_{\theta}(s)^{2}\,ds\bigg)^{\frac{1}{2}} + \bigg(\int_{t_{1}}^{t}\lambda_{\theta}(s)^{2}\,ds\bigg)^{\frac{1}{2}}\\
\text{and}\quad c_{1}(t_{1},t) &:= \int_{t_{1}}^{t}\lambda_{\theta}(s)^{2}\,ds + \bigg(\int_{t_{1}}^{t}\lambda_{\theta}(s)^{2}\,ds\bigg)^{\frac{1}{2}}.
\end{split}
\end{equation*}
If in addition the function $[t_{0},T]\rightarrow [0,\infty]$, $t\mapsto E[V_{t}]$ is finite and bounded, then $\sup_{t\in [t_{0},T]} |X_{t}|$ is integrable, entailing that $X$ is uniformly integrable. A similar approach leads to the announced comparison bound.

\begin{Lemma}\label{le:log-price comparison estimate}
Let~\eqref{con:v.2} hold and $V$, $\tilde{V}$ be positive adapted continuous processes. Then any two solutions $X$ and $\tilde{X}$ to~\eqref{eq:log-price dynamics} with underlying processes $V$ and $\tilde{V}$, respectively, satisfy
\begin{equation}\label{eq:loc-price comparison estimate}
\begin{split}
E\bigg[\sup_{s\in [t_{0},t]}|X_{s}^{\sigma} - \tilde{X}_{s}^{\sigma}|\bigg] &\leq E\big[|X_{t_{0}} - \tilde{X}_{t_{0}}|\big]\\
&\quad + c_{2}(t_{0},t)\sup_{s\in [t_{0},t]}\big(1 + E\big[V_{s}^{\sigma}\big] + E\big[\tilde{V}_{s}^{\sigma}\big]\big)^{\frac{1}{2}} E\big[|V_{s}^{\sigma} - \tilde{V}_{s}^{\sigma}|\big]^{\frac{1}{2}}
\end{split}
\end{equation}
for all $t\in [t_{0},T]$ and each stopping time $\sigma$ with $\sigma\geq t_{0}$, where the $\R_{+}$-valued continuous function $c_{2}$ on the set of all $(t_{1},t)\in [0,T]\times [0,T]$ with $t_{1}\leq t$ is given by
\begin{equation*}
c_{2}(t_{1},t) := \int_{t_{1}}^{t}\big(k_{\theta}(s) + \lambda_{\theta}(s)\big)\lambda_{\theta}(s)\,ds + 2\bigg(\int_{t_{1}}^{t}\lambda_{\theta}(s)^{2}\,ds\bigg)^{\frac{1}{2}}.
\end{equation*}
\end{Lemma}

Let us now settle the question of pathwise uniqueness for the second SDE in~\eqref{eq:stochastic volatility model} by referring to a comparison estimate, under a one-sided Lipschitz continuity condition on the drift $\zeta$ and an Osgood continuity condition on compact sets on the diffusion $\eta$:
\begin{enumerate}[label=(V.\arabic*), ref=V.\arabic*, leftmargin=\widthof{(V.4)} + \labelsep]
\setcounter{enumi}{2}
\item\label{con:v.3} There is $\lambda_{\zeta}\in\mathscr{L}^{1}(\R)$ with $\mathrm{sgn}(v-\tilde{v})(\zeta(\cdot,v) - \zeta(\cdot,\tilde{v})) \leq \lambda_{\zeta}|v-\tilde{v}|$ for all $v,\tilde{v}\in\R$ a.e.
\item\label{con:v.4} For each $n\in\N$ there are $\lambda_{\eta,n}\in\mathscr{L}^{2}(\R_{+})$ and an increasing continuous function $\rho_{n}:\R_{+}\rightarrow\R_{+}$ that is positive on $]0,\infty[$ such that
\begin{equation*}
|\eta(\cdot,v) - \eta(\cdot,\tilde{v})| \leq \lambda_{\eta,n}\rho_{n}(|v-\tilde{v}|)
\end{equation*}
for any $v,\tilde{v}\in [-n,n]$ a.e.~and $\int_{0}^{1}\rho_{n}(v)^{-2}\,dv = \infty$.
\end{enumerate}

\begin{Remark}
The bound in~\eqref{con:v.3} is valid if and only if $(\zeta(\cdot,v) - \zeta(\cdot,\tilde{v}))/(v-\tilde{v}) \leq \lambda_{\zeta}$ for all $v,\tilde{v}\in\R$ with $v\neq\tilde{v}$ a.e. For instance, this holds if $\zeta(s,\cdot)$ is locally absolutely continuous and its weak derivative $\partial_{v}\zeta(s,\cdot)$ is bounded from above by $\lambda_{\zeta}(s)$ for a.e.~$s\in [0,T]$.
\end{Remark}

Under~\eqref{con:v.3} and~\eqref{con:v.4}, an application of Corollary~3.9, combined with Remark~3.10, and Proposition~3.13 in~\cite{KalMeyPro21} shows that there is pathwise uniqueness for the second SDE in~\eqref{eq:stochastic volatility model} and any two solutions $V$ and $\tilde{V}$ satisfy
\begin{equation}\label{eq:variance comparison estimate}
E\big[|V_{t} - \tilde{V}_{t}|\big] \leq e^{\int_{t_{0}}^{t}\lambda_{\zeta}(s)\,ds}E\big[|V_{t_{0}} - \tilde{V}_{t_{0}}|\big]
\end{equation}
for each $t\in [t_{0},T]$. Regarding strong existence, let us additionally require two conditions involving the growth and continuity of the drift and diffusion coefficients $\zeta$ and $\eta$:
\begin{enumerate}[label=(V.\arabic*), ref=V.\arabic*, leftmargin=\widthof{(V.6)} + \labelsep]
\setcounter{enumi}{4}
\item\label{con:v.5} There are $k_{\zeta},l_{\zeta}\in\mathscr{L}^{1}(\R)$ such that $k_{\zeta}\geq 0$ and $\mathrm{sgn}(v)\zeta(\cdot,v) \leq k_{\zeta} + l_{\zeta}|v|$ for any $v\in\R$ a.e.~and $\eta(\cdot,0) = 0$.
\item\label{con:v.6} $\zeta(t,\cdot)$ and $\eta(t,\cdot)$ are continuous for a.e.~$t\in [0,T]$. Further, for any $n\in\N$ there is $c_{\zeta,\eta,n}\in\R_{+}$ such that $|\zeta(\cdot,v)|\vee |\eta(\cdot,v)| \leq c_{\zeta,\eta,n}$ for each $v\in [-n,n]$ a.e.
\end{enumerate}

Then Theorem 3.27 in~\cite{KalMeyPro21} asserts that for any $\mathscr{F}_{t_{0}}$-measurable integrable random variable $\xi$ the second SDE in~\eqref{eq:stochastic volatility model} admits a unique strong solution $V$ such that $V_{t_{0}} = \xi$ a.s.~and
\begin{equation}\label{eq:variance growth estimate}
E\big[|V_{t}|\big] \leq e^{\int_{t_{0}}^{t}l_{\zeta}(s)\,ds}E\big[|\xi|\big] + \int_{t_{0}}^{t}e^{\int_{s}^{t}l_{\zeta}(\tilde{s})\,d\tilde{s}}k_{\zeta}(s)\,ds
\end{equation}
for all $t\in [t_{0},T]$. Now we give sufficient conditions for any solution to have a.s.~positive paths. To this end, we generalise Theorem 2.2 in Mishura and Posashkova~\cite{MisPos08}. There, it is in particular required that 
\begin{equation*}
\inf_{(t,x)\in [t_{0},T]\times [\delta,\infty[} \eta(t,x) > 0\quad\text{for any $\delta > 0$.}
\end{equation*}
A positivity condition on the diffusion coefficient $\eta$, which we omit. Instead, the weakened regularity condition that we impose reads as follows:
\begin{enumerate}[label=(V.\arabic*), ref=V.\arabic*, leftmargin=\widthof{(V.7)} + \labelsep]
\setcounter{enumi}{6}
\item\label{con:v.7} There are $\varepsilon > 0$ and $c_{0},c_{\zeta}\in\mathscr{L}^{1}(\mathbb{R}_{+})$ as well as increasing functions $\varphi_{0}:]0,\varepsilon]\rightarrow\mathbb{R}_{+}$ and $\varphi_{\zeta}:[\varepsilon,\infty[\rightarrow ]0,\infty[$ such that $\varphi_{\zeta}$ is continuous and
\begin{equation}\label{eq:positivity condition}
\frac{\eta(\cdot,v)^{2}}{2v^{2}} \leq \frac{\zeta(\cdot,v)}{v} + c_{0}\varphi_{0}(v)\quad\text{and}\quad \zeta(\cdot,\tilde{v})\geq - c_{\zeta}\varphi_{\zeta}(\tilde{v})
\end{equation}
for each $v\in ]0,\varepsilon[$ and any $\tilde{v}\geq \varepsilon$ a.e.
\end{enumerate}

Regardless of whether uniqueness in law holds for the underlying equation, under this condition any solution starting at a positive deterministic value remains positive.

\begin{Proposition}\label{pr:positivity of the variance process}
Let~\eqref{con:v.7} be valid. Then any solution $V$ to the second SDE in~\eqref{eq:stochastic volatility model} such that $V_{t_{0}} = v_{0}$ a.s.~for some $v_{0} > 0$~satisfies $V_{t} > 0$ for any $t\in [t_{0},T]$ a.s.
\end{Proposition}

\begin{Example}\label{ex:processes with positive paths}
For $n\in\N$ let $k,l_{1},\dots,l_{n}\in\mathscr{L}^{1}(\mathbb{R})$ and $\varphi_{1},\dots,\varphi_{n}$ be real-valued Borel measurable functions on $]0,\infty[$ such that $k\geq 0$,
\begin{equation*}
\zeta(t,v) = k(t) + l_{1}(t)\varphi_{1}(v) + \cdots + l_{n}(t)\varphi_{n}(v)\quad\text{and}\quad \limsup_{v\downarrow 0} \frac{|\varphi_{i}(v)|}{v}  < \infty
\end{equation*}
for all $t\in [0,T]$, any $v > 0$ and each $i\in\{1,\dots,n\}$. Suppose that there are $\varepsilon_{0} > 0$, $c_{\eta}\in\mathscr{L}^{2}(\R_{+})$, $\gamma\in [1/2,\infty[$ and an increasing function $\varphi:]0,\infty[\rightarrow\R_{+}$ satisfying
\begin{equation}\label{eq:special positivity condition}
|\eta(t,v)| \leq c_{\eta}(t) v^{\gamma}\big(1 + v\varphi(v)\big)^{\frac{1}{2}}\quad\text{for any $(t,v)\in[0,T]\times ]0,\varepsilon_{0}[$}.
\end{equation}
If $c_{\eta}^{2}/2 \leq k$ for $\gamma = 1/2$ and, less restrictively, $c_{\eta}^{2}\delta \leq k$ for some $\delta > 0$ whenever $\gamma\in ]1/2,1[$, then the first inequality in~\eqref{eq:positivity condition} holds. Indeed, take $\varepsilon\in ]0,\varepsilon_{0}]$ and $c > 0$ such that
\begin{equation*}
|\varphi_{i}(v)|\leq c v,\quad\text{and}\quad v^{2\gamma - 1} \leq 2\delta\quad\text{in case $\gamma\in ]1,2,1[$},
\end{equation*}
for any $i\in\{1,\dots,n\}$ and all $v\in ]0,\varepsilon_{0}]$. Then $c_{0}:= c\sum_{l=1}^{n}|l_{i}| + c_{\eta}^{2}/2$ and $\varphi_{0}:]0,\varepsilon]\rightarrow\R_{+}$ given by $\varphi_{0}(v):=1 + v^{2\gamma - 2}(\mathbbm{1}_{[1,\infty[}(\gamma) + v\varphi(v))$ satisfy
\begin{equation*}
\frac{c_{\eta}(t)^{2}}{2v} v^{2\gamma - 1}\big(1 + v\varphi(v)\big) \leq \frac{\zeta(t,v)}{v} + c_{0}(t)\varphi_{0}(v)
\end{equation*}
for all $(t,v)\in [0,T]\times ]0,\varepsilon[$. In particular, we may take $\alpha\in [1,\infty[^{n}$, $\beta\in [1/2,\infty[^{n}$ and $\lambda_{1},\dots,\lambda_{n}\in\mathscr{L}^{2}(\R)$ such that $\varphi_{1}(v) = v^{\alpha_{1}},\dots,\varphi_{n}(v) = v^{\alpha_{n}}$ and
\begin{equation}\label{eq:sums of power functions}
\eta(\cdot,v) = \lambda_{1}v^{\beta_{1}} + \cdots + \lambda_{n}v^{\beta_{n}}\quad\text{for all $v > 0$}.
\end{equation}
In this case, the estimate~\eqref{eq:special positivity condition} holds for the choice $c_{\eta} = \sum_{i=1}^{n}|\lambda_{i}|$, $\gamma=\min_{i\in\{1,\dots,n\}}\beta_{i}$ and $\varphi = 0$ as soon as $\varepsilon_{0} < 1$.
\end{Example}

Due to the integrability condition $\int_{0}^{T}(1-\rho(s)^{2})^{-1}\,ds < \infty$, there is another standard $(\mathscr{F}_{t})_{t\in [0,T]}$-Brownian motion $W$ that is independent of $\tilde{W}$ such that
\begin{equation*}
\hat{W}_{t} = \int_{0}^{t}\sqrt{1 - \rho(s)^{2}}\,dW_{s} + \int_{0}^{t}\rho(s)\,d\tilde{W}_{s}\quad\text{for all $t\in [0,T]$ a.s.}
\end{equation*}
By using this representation, a simple transformation shows that we can rearrange~\eqref{eq:stochastic volatility model} into the two-dimensional SDE
\begin{equation}\label{eq:transformed stochastic volatility model}
d\begin{pmatrix}
X_{t} \\
V_{t}
\end{pmatrix}
 = \begin{pmatrix}
b(t) - \frac{1}{2}\theta(t,V_{t})^{2} \\
\zeta(t,V_{t})
\end{pmatrix}\,dt + \begin{pmatrix}
\theta(t,V_{t})\sqrt{1-\rho(t)^{2}} & \theta(t,V_{t})\rho(t) \\
0 & \eta(t,V_{t})
\end{pmatrix}\,d
\begin{pmatrix}
W_{t} \\
\tilde{W}_{t}
\end{pmatrix}
\end{equation}
for $t\in [t_{0},T]$. Then the pair of two adapted continuous processes $X$ and $V$ is a solution to this SDE if and only if $(\exp(X),V)$ solves the initial one.

Further,~\eqref{eq:transformed stochastic volatility model} induces a linear differential operator $\mathscr{L}_{b,\zeta}$ on $C^{1,2}([0,T[\times\R\times ]0,\infty[)$ with values in the linear space of all real-valued measurable functions by
\begin{equation}\label{eq:differential operator}
\begin{split}
&\mathscr{L}_{b,\zeta}(\varphi)(t,x,v) := \bigg(b(t) - \frac{1}{2}\theta(t,v)^{2}\bigg)\frac{\partial\varphi}{\partial x}(t,x,v) + \zeta(t,v)\frac{\partial\varphi}{\partial v}(t,x,v)\\
&\quad + \frac{1}{2}\theta(t,v)^{2}\frac{\partial^{2}\varphi}{\partial x^{2}}(t,x,v) + \theta(t,v)\eta(t,v)\rho(t)\frac{\partial^{2}\varphi}{\partial x\partial v}(t,x,v) + \frac{1}{2}\eta(t,v)^{2}\frac{\partial^{2}\varphi}{\partial v^{2}}(t,x,v).
\end{split}
\end{equation}
This formula is obtained by multiplying the diffusion coefficient with its transpose, as we readily recall, and for every solution $(X,V)$ to~\eqref{eq:transformed stochastic volatility model} It{\^o}'s formula entails that the process $[t_{0},T]\times\Omega\rightarrow\R$,
\begin{equation*}
(t,\omega)\mapsto\varphi(t,X_{t},V_{t})(\omega) - \int_{t_{0}}^{t}\bigg(\frac{\partial}{\partial s} + \mathscr{L}_{b,\zeta}\bigg)(\varphi)(s,X_{s},V_{s})(\omega)\,ds
\end{equation*}
is a martingale for any $\varphi\in C_{0}^{1,2,2}([0,T]\times\R\times ]0,\infty[)$. We will now show that~\eqref{eq:transformed stochastic volatility model} yields a time-inhomogeneous Markov process with the right-hand Feller property in the sense of~\cite{Kal20}[Section 2.3], which will allow us to apply the results on mild solutions therein.

In fact, what we get is a continuous strong Markov process, or in short, a \emph{diffusion}, and we will realise it on the canonical space $\hat{\Omega}$ of all $\R\times ]0,\infty[$-valued continuous paths on $[0,T]$, endowed with its Borel $\sigma$-field $\hat{\mathscr{F}}$. Let $\hat{X}:[0,T]\times\hat{\Omega}\rightarrow\R$ and $\hat{V}:[0,T]\times\hat{\Omega}\rightarrow ]0,\infty[$ be given by
\begin{equation*}
\hat{X}_{t}(\hat{\omega}) := \hat{\omega}_{1}(t)\quad\text{and}\quad \hat{V}_{t}(\hat{\omega}) := \hat{\omega}_{2}(t).
\end{equation*}
Then $(\hat{X},\hat{V})$ serves as canonical process, its natural filtration $(\hat{\mathscr{F}}_{t})_{t\in [0,T]}$ satisfies $\hat{\mathscr{F}} = \hat{\mathscr{F}}_{T}$ and the law of any two processes $X:[0,T]\times\Omega\rightarrow\R$ and $V:[0,T]\times\Omega\rightarrow ]0,\infty[$ is of the form $P((X,V)\in \hat{B}) = P\circ (X,V)^{-1}((\hat{X},\hat{V})\in\hat{B})$ for all $\hat{B}\in\hat{\mathscr{F}}$.

\begin{Proposition}\label{pr:transformed SDE}
Under~\eqref{con:v.1}-\eqref{con:v.7}, the following four assertions hold:
\begin{enumerate}[(i)]
\item We have pathwise uniqueness for the SDE~\eqref{eq:transformed stochastic volatility model}.
\item For any $(x_{0},v_{0})\in\R\times ]0,\infty[$ there is a unique strong solution $(X^{t_{0},x_{0},v_{0}},V^{t_{0},v_{0}})$ to~\eqref{eq:transformed stochastic volatility model} such that $(X_{t_{0}}^{t_{0},x_{0},v_{0}},V_{t_{0}}^{t_{0},v_{0}}) = (x_{0},v_{0})$ a.s.~and $V^{t_{0},v_{0}} > 0$.
\item The map $[0,t]\times\R\times ]0,\infty[\rightarrow\R$, $(s,x,v)\mapsto E[\varphi(s,X_{t}^{s,x,v},V_{t}^{s,v})]$ is right-continuous for any $t\in ]0,T]$ and each $\varphi\in C_{b}([0,T]\times\R\times ]0,\infty[)$.
\item  Let $P_{s,x,v}$ be the law of the process $[0,T]\times\Omega\rightarrow\R\times ]0,\infty[$, $(t,\omega)\mapsto (X_{t\vee s}^{s,x,v},V_{t\vee s}^{s,v})(\omega)$ for any $(s,x,v)\in [0,T]\times\R\times ]0,\infty[$ and denote the set of all these measures by $\mathbb{P}$. Then $((\hat{X},\hat{V}),(\hat{\mathscr{F}}_{t})_{t\in [0,T]},\mathbb{P})$ is a diffusion that is right-hand Feller.
\end{enumerate}
\end{Proposition}

\begin{Example}\label{ex:stochastic volatility model}
Let $\hat{\theta}_{0},\hat{\theta}\in\mathscr{L}^{2}(\R)$, $n\in\N$, $\alpha\in [1,\infty[^{n}$ and $\beta\in [1/2,\infty[^{n}$. Assume that the functions $k,l_{0}:[0,T]\rightarrow\R$ and the maps $l,\lambda:[0,T]\rightarrow\R^{n}$ are measurable and bounded such that $\theta(\cdot,v) = \hat{\theta}_{0} + \hat{\theta}\sqrt{v}$,
\begin{equation*}
\zeta(\cdot,v) = k - l_{0} v + l_{1}v^{\alpha_{1}} + \cdots + l_{n}v^{\alpha_{n}}\quad\text{and}\quad \eta(\cdot,v) = \lambda_{1} v^{\beta_{1}} + \cdots + \lambda_{n}v^{\beta_{n}}
\end{equation*}
for any $v > 0$. Further, let $k\geq 0$ and $l_{i}\leq 0$ for all $i\in\{1,\dots,n\}$ and we require that for $\gamma:=\min\{\beta_{1},\dots,\beta_{n}\}$ it holds that $(\sum_{i=1}^{n}|\lambda_{i}|)^{2}/2\leq k$, if $\gamma=1/2$, and
\begin{equation*}
\bigg(\sum_{i=1}^{n}|\lambda_{i}|\bigg)^{2}\delta \leq k\quad\text{for some $\delta > 0$, if $\gamma\in ]1/2,1[$.}
\end{equation*}
By imposing a radial representation of $\theta$ and $\eta$ on $[0,T]\times\R$ and setting $\zeta(\cdot,v) = \zeta(\cdot,0)$ for all $v < 0$, the SDE~\eqref{eq:stochastic volatility model} reduces to
\begin{equation}\label{eq:specific model}
\begin{split}
dS_{t} &= b(t)S_{t}\,dt + \big(\hat{\theta}_{0}(t) + \hat{\theta}(t)|V_{t}|^{\frac{1}{2}}\big)S_{t}\,d\hat{W}_{t},\\
dV_{t} &= \bigg(k(t) - l_{0}(t)V_{t} + \sum_{i=1}^{n} l_{i}(t)(V_{t}^{+})^{\alpha_{i}}\bigg)\,dt + \sum_{i=1}^{n}\lambda_{i}(t)|V_{t}|^{\beta_{i}}\,d\tilde{W}_{t}
\end{split}
\end{equation}
for $t\in [t_{0},T]$ with initial conditions $S_{t_{0}} = s_{0}$ and $V_{t_{0}} = v_{0}$ a.s.~for some $s_{0},v_{0} > 0$. In particular, for $n=1$ and $l=0$ we recover the dynamics in a generalised time-dependent version of the following option pricing models:
\begin{enumerate}[(1)]
\item The \emph{stock price model} by Black and Scholes~\cite{BlaFisSch73} for $\hat{\theta}_{0} = k = l_{0} = \lambda = 0$ and $\hat{\theta} = 1$, which entails that any solution $V$ to the second SDE in~\eqref{eq:specific model} satisfies $V = v_{0}$ a.s., where $\sqrt{v_{0}}$ stands for the \emph{volatility}.
\item The \emph{stochastic volatility model by Heston}~\cite{Hes93} for $\hat{\theta}_{0} = 0$, $l_{0} > 0$ and $\beta=1/2$, in which case $V$ is a \emph{square-root diffusion}. There, $l_{0}$ is the \emph{mean reversion speed}, $k/l_{0}$ is the \emph{mean reversion level} and the same positivity condition $\lambda^{2}\leq 2k$ for $V$ applies.
\item The \emph{Garch diffusion model}~\cite{Lew00} for $\hat{\theta}_{0} = 0$, $l_{0} > 0$ and $\beta=1$. Similarly to the Heston model, $l_{0}$ is the mean reversion speed and $k/l_{0}$ the mean reversion level.
\end{enumerate}
We observe that~\eqref{con:v.1}-\eqref{con:v.7} hold. In fact, because the function $\R\rightarrow\R_{+}$, $v\mapsto (v^{+})^{\gamma}$ is increasing, non-positive on $]-\infty,0[$ and non-negative on $]0,\infty[$ for any $\gamma > 0$, the one-sided conditions~\eqref{con:v.3} and~\eqref{con:v.5} follow. More precisely,
\begin{align*}
\mathrm{sgn}(v-\tilde{v})(\zeta(\cdot,v) - \zeta(\cdot,\tilde{v})) \leq  -l_{0}|v-\tilde{v}|\quad\text{and}\quad\mathrm{sgn}(v)\zeta(\cdot,v) \leq k - l_{0}v
\end{align*}
for every $v,\tilde{v} > 0$ a.s. From Example~\ref{ex:processes with positive paths} we infer the validity of~\eqref{con:v.7} and the other conditions are readily checked. Consequently, Proposition~\ref{pr:transformed SDE} applies to the SDE~\eqref{eq:specific model}.
\end{Example}

\subsection{The logarithmised pre-default valuation PDE}\label{se:4.2}

Now we combine the stochastic volatility model of the previous section with the market model from Section~\ref{se:3} to \emph{characterise pre-default value processes via mild solutions} to the associated parabolic PDE.

Let $(\Omega,\mathscr{F},(\mathscr{F}_{t})_{t\in [0,T]},\tilde{P})$ serve as underlying probability space such that the usual conditions hold. We assume that both $S$ and $V$ take positive values only and consider the following specifications for our market model:
\begin{enumerate}[label=(P.\arabic*), ref=P.\arabic*, leftmargin=\widthof{(P.4)} + \labelsep]
\item\label{con:p.1} $G(\tau_{\mathcal{I}})$ and $G(\tau_{\mathcal{C}})$ are deterministic, absolutely continuous, positive and satisfy~\eqref{eq:condition on the survival process}, $\tau_{\mathcal{I}}$ and $\tau_{\mathcal{C}}$ are $(\mathscr{F}_{t})_{t\in [0,T]}$-conditionally independent under $\tilde{P}$ and $G(\tau) = G(\tau_{\mathcal{I}})G(\tau_{\mathcal{C}})$.
\item\label{con:p.2} The risk-free rate $r$, the remuneration rates $\null_{+}c$, $\null_{-}c$, the funding rates $\null_{+}\tilde{f}$, $\null_{-}\tilde{f}$ and the hedging rates $\null_{+}\tilde{h}$, $\null_{-}\tilde{h}$, which have integrable paths, are deterministic. That is,
\begin{equation*}
r_{t} = \hat{r}(t),\quad \null_{i}c_{t} = \hat{c}_{i}(t),\quad \null_{i}\tilde{f}_{t} = \hat{f}_{i}(t),\quad \null_{i}\tilde{h}_{t} = \hat{h}_{i}(t)
\end{equation*}
for every $t\in [0,T]$, both $i\in\{+,-\}$ and some $\hat{r},\hat{c}_{i},\hat{f}_{i}$, $\hat{h}_{i}\in\mathscr{L}^{1}(\R)$.
\item\label{con:p.3} There is a measurable function $\phi:]0,\infty[^{2}\rightarrow\R_{+}$ and a real-valued continuous function $\hat{\pi}$ on $[0,T]\times]0,\infty[^{2}$ such that $\Phi(s,v) = \phi(s(T),v(T))$ for any $s,v\in C([0,T])$ that are positive and $\pi = \hat{\pi}(\cdot,S,V)$.
\item\label{con:p.4} There are $\alpha,\beta\in C([0,T])$ with $0\leq\alpha\leq \beta\leq 1$ and some continuous function $\hat{H}:[0,T]\times]0,\infty[^{2}\times\R\rightarrow\R$ satisfying $C(Y) = \alpha Y$, $\varepsilon(Y)=\beta Y$,
\begin{equation*}
\tilde{H}(Y) = \hat{H}(\cdot,S,V,Y)\quad\text{and}\quad \tilde{F}(Y) = (1-\alpha)Y\quad\text{for all $Y\in\mathscr{S}$}.
\end{equation*}
\end{enumerate}

\begin{Remark}\label{re:integrability conditions}
Under~\eqref{con:p.1} and~\eqref{con:p.2}, Remark~\ref{re:integrability of pre-default components} entails that any $\mathscr{F}_{T}$-measurable random variable lies in $\tilde{\mathscr{L}}(r,\tau)$ if and only if it is $\tilde{P}$-integrable. Further, if $Y$ is a process that is $(\mathscr{F}_{t})_{t\in [0,T]}$-progressively measurable, then $Y\in\tilde{\mathscr{S}}(r,\tau)$ $\Leftrightarrow$ $\int_{0}^{T}\tilde{E}\big[|Y_{t}|\big]\,dt < \infty$.
\end{Remark}

It is readily seen that~\eqref{con:p.1} holds in Example~\ref{ex:default time specification} as soon as the two processes $\lambda^{(\mathcal{I})}$ and $\lambda^{(\mathcal{C})}$ there are deterministic. Put differently, for both $i\in\{\mathcal{I},\mathcal{C}\}$ there is $\hat{\lambda}_{i}\in\mathscr{L}^{1}(\R_{+})$ such that $\lambda^{(i)} = \hat{\lambda}_{i}$ and $\int_{0}^{t}\hat{\lambda}_{i}(s)\,ds > 0$ for some $t\in ]0,T]$.

Further,~\eqref{con:p.1} and~\eqref{con:p.2} imply~\eqref{con:m.2} and~\eqref{con:m.4}, respectively. If, conversely,~\eqref{con:m.4} is satisfied and the pre-default rates $\null_{+}f$, $\null_{-}f$, $\null_{+}h$, $\null_{-}h$ are deterministic, then, as required in~\eqref{con:p.2}, so are the rates $\null_{+}\tilde{f}$, $\null_{-}\tilde{f}$, $\null_{+}\tilde{h}$, $\null_{-}\tilde{h}$, according to~\eqref{eq:pre-default version}.

In~\eqref{con:p.3} the payoff functional $\Phi$ is reduced to a function of the terminal state and the dividend rate $\pi$ to a function of the current state of the process $[0,T]\times\Omega\rightarrow [0,T]\times ]0,\infty[^{2}$, $(t,\omega)\mapsto (t,S_{t},V_{t})(\omega)$. So, if~\eqref{con:p.1} and~\eqref{con:p.2} are valid, then~\eqref{con:m.1} holds if and only if
\begin{equation}\label{eq:three terms}
\tilde{E}[\phi(S_{T},V_{T})]\quad\text{and}\quad \int_{0}^{T}\tilde{E}\big[|\hat{\pi}(t,S_{t},V_{t})|\big]\,dt\quad\text{are finite},
\end{equation}
as Remark~\ref{re:integrability conditions} shows. Now let for the moment~\eqref{con:p.4} be valid. Then we can use $\hat{H}(\cdot,S,V,Y)$ as pre-default version of $\hat{H}(\cdot,S,V,\tilde{Y})$ for any $\tilde{Y}\in\tilde{\mathscr{S}}$ that is integrable up to time $\tau$ with pre-default version $Y$ if
\begin{equation}\label{eq:4.13}
\tilde{E}\big[|\hat{H}(t,S_{t},V_{t},Y_{t})|\big] < \infty\quad\text{for all $t\in [0,T]$}.
\end{equation}
Under this condition,~\eqref{con:m.3} follows from~\eqref{con:p.4}. Hence, let us now assume that~\eqref{con:p.1}-\eqref{con:p.4}, \eqref{eq:three terms} and~\eqref{eq:4.13} hold. Then~\eqref{con:m.1}-\eqref{con:m.4} are satisfied and we may turn to the pre-default valuation in~\eqref{eq:pre-default valuation}. 

Moreover,~\eqref{con:p.4} now ensures that the collateral process $C(\mathscr{V})$ and the close-out value $\varepsilon(\mathscr{V})$ are deterministic ordered fractions of any given pre-default value process $\mathscr{V}$ and the financing hypothesis~\eqref{eq:financing hypothesis} holds.

By recalling the time-dependent random functionals $\null_{0}\B$, $\null_{\mathcal{I}}\B$ and $\null_{\mathcal{C}}\B$ in~\eqref{eq:inhomogeneity process components}, we see that the random functional $\B$ given by~\eqref{eq:inhomogeneity process} is of the form
\begin{align*}
\B_{t}(Y) &= \pi_{t} - \big(c_{t}(Y)-r_{t}\big)C_{t}(Y) - \big(f_{t}(Y) - r_{t}\big)F_{t}(Y) - \big(r_{t} - h_{t}(Y)\big)H_{t}(Y) - r_{t}Y_{t}\\
&\quad - \frac{\dot{G}_{t}(\tau_{\mathcal{I}})}{G_{t}(\tau_{\mathcal{I}})}\big(\varepsilon_{t}(Y) + \mathrm{LGD}_{\mathcal{I}}\big((\varepsilon_{t} - C_{t})^{-} + F_{t}^{+}\big)(Y)\mathbbm{1}_{\{\mathcal{I}=\mathcal{B}\}} -Y_{t}\big)\\
&\quad - \frac{\dot{G}_{t}(\tau_{\mathcal{C}})}{G_{t}(\tau_{\mathcal{C}})}\big(\varepsilon_{t}(Y) -\mathrm{LGD}_{\mathcal{C}}(\varepsilon_{t} - C_{t})^{+}(Y) - Y_{t}\big) = \hat{B}(t,S_{t},V_{t},Y_{t})
\end{align*}
for all $t\in [0,T]$ and each continuous $Y\in\mathscr{S}$ with the real-valued measurable function $\hat{B}$ defined on $[0,T]\times]0,\infty[^{2}\times\R$ via
\begin{equation}\label{eq:inhomogeneity function}
\begin{split}
\hat{B}(t,s,&v,y) := \hat{\pi}(t,s,v) - \big(\hat{c}_{+}\alpha + \hat{f}_{+}(1-\alpha)\big)(t)y^{+} + \big(\hat{c}_{-}\alpha + \hat{f}_{-}(1-\alpha)\big)(t)y^{-}\\
&\quad - (\hat{r} - \hat{h}_{+})(t)\hat{H}^{+}(t,s,v,y) + (\hat{r} - \hat{h}_{-})(t)\hat{H}^{-}(t,s,v,y)\\
&\quad + g_{\mathcal{I}}(t)\bigg((1-\beta)(t)y - \mathrm{LGD}_{\mathcal{I}}\big((\beta-\alpha)(t)y^{-} + (1-\alpha)(t)y^{+}\big)\mathbbm{1}_{\{\mathcal{B}\}}(\mathcal{I})\bigg)\\
&\quad + g_{\mathcal{C}}(t)\bigg((1-\beta)(t)y + \mathrm{LGD}_{\mathcal{C}}(\beta - \alpha)(t)y^{+}\bigg),
\end{split}
\end{equation}
where $g_{i}:[0,T]\rightarrow\R$ is a measurable integrable function satisfying $g_{i} = \dot{G}(\tau_{i})/G(\tau_{i})$ for both $i\in\{\mathcal{I},\mathcal{C}\}$. In this deterministic setting, all the cash flows and costs and benefits
\begin{equation*}
\null_{\mathrm{con}}\mathrm{CF}(\mathscr{V}),\quad\null_{\mathrm{col}}\mathrm{C}(\mathscr{V}),\quad \null_{\mathrm{fun}}\mathrm{C}(\mathscr{V}),\quad \null_{\mathrm{hed}}\mathrm{C}(\mathscr{V})\quad\text{and}\quad\null_{\mathrm{def}}\mathrm{CF}(\mathscr{V},\tilde{\mathscr{V}})
\end{equation*}
are measurable functionals of $(\tau_{\mathcal{I}},\tau_{\mathcal{C}}, S , V, \mathscr{V}, \tilde{\mathscr{V}})$ for any $\tilde{\mathscr{V}}\in\tilde{\mathscr{S}}$ that is integrable up to time $\tau$ with pre-default version $\mathscr{V}$. Further, the continuous process $A(\mathscr{V})$, given by~\eqref{eq:finite variation process} and appearing in~\eqref{eq:pre-default valuation}, is adapted to the natural filtration of $(S,V,\mathscr{V})$ and we may now deal with the assumed existence of $\tilde{P}$.

So, suppose that~\eqref{eq:stochastic volatility model} is solved strongly by $S$ and $V$ on $[t_{0},T]$ under $P$ for some $t_{0}\in [0,T]$ and~\eqref{con:v.1} holds. Let $\lambda$, $\tilde{\lambda}$ be two $(\mathscr{F}_{t})_{t\in [t_{0},T]}$-progressively measurable processes with square-integrable paths and define a continuous local martingale $Z$ via
\begin{equation}\label{eq:density process}
Z^{t_{0}} = 1\quad\text{and}\quad Z_{t} = \exp\bigg(-\int_{t_{0}}^{t}\lambda_{s}\,dW_{s} - \int_{t_{0}}^{t}\tilde{\lambda}_{s}\,d\tilde{W}_{s} - \frac{1}{2}\int_{t_{0}}^{t}\lambda_{s}^{2} + \tilde{\lambda}_{s}^{2}\,ds\bigg)
\end{equation}
for any $t\in [t_{0},T]$ a.s. Under the condition that $E[Z_{T}] = 1$, Girsanov's theorem states that $W^{(\lambda)}:= W + \int_{t_{0}}^{\cdot\vee t_{0}}\lambda_{s}\,ds$ and $W^{(\tilde{\lambda})}:=\tilde{W} + \int_{t_{0}}^{\cdot\vee t_{0}}\tilde{\lambda}_{s}\,ds$ are independent $(\mathscr{F}_{t})_{t\in [0,T]}$-Brownian motions under the measure $\hat{P}_{t_{0},\lambda,\tilde{\lambda}}$ on $(\Omega,\mathscr{F})$ given by $\hat{P}_{t_{0},\lambda,\tilde{\lambda}}(A):=E[Z_{T}\mathbbm{1}_{A}]$. 

Assuming that $S$ is the only price process to consider, Remark~\ref{re:martingale property} and~\eqref{con:p.1} entail that $\hat{P}_{t_{0},\lambda,\tilde{\lambda}}$ is an equivalent local martingale measure after time $t_{0}$ if and only if
\begin{equation*}
(b - \hat{r})(t) = \theta(t,V_{t})\big(\lambda_{t}\sqrt{1-\rho(t)^{2}} + \tilde{\lambda}_{t}\rho(t)\big)\quad\text{for a.e.~$t\in [t_{0},T]$ a.s.}
\end{equation*}
Indeed, this follows from the representation~\eqref{eq:stochastic exponential} in Lemma~\ref{le:price process representation} and the relation between $\hat{W}$ and $(W,\tilde{W})$, stated directly before~\eqref{eq:transformed stochastic volatility model}. Thus, if $\theta(\cdot,v) > 0$ for all $v > 0$, then we propose to take the \emph{market prices of risk}
\begin{equation}\label{eq:market prices}
\tilde{\lambda}_{t} = \gamma \theta(t,V_{t})\quad\text{and}\quad \lambda_{t} = \bigg( \frac{(b-\hat{r})(t)}{\theta(t,V_{t})} - \gamma\theta(t,V_{t})\rho(t)\bigg)\frac{1}{\sqrt{1-\rho(t)^{2}}}
\end{equation}
for all $t\in [t_{0},T]$ and some $\gamma\in\R_{+}$. As $V$ is a strong solution, $\lambda$ is independent of $W$. Hence, if $b$, $\hat{r}$, $\rho$ and $\theta$ are continuous, then it follows from Theorem~21.4 in~\cite{Bau00} and Lemma~35 in~\cite{ConKal20} that
\begin{equation*}
E[Z_{T}] = E\bigg[\exp\bigg(-\gamma\int_{t_{0}}^{T}\theta(t,V_{t})\,d\tilde{W}_{t} - \frac{1}{2}\gamma^{2}\int_{t_{0}}^{T}\theta(t,V_{t})^{2}\,dt\bigg)\bigg],
\end{equation*}
and for $E[Z_{T}] =1$ to hold, it suffices that $\exp(\frac{\gamma^{2}}{2}\int_{t_{0}}^{T}\theta(t,V_{t})^{2}\,dt)$ is $P$-integrable, by Novikov's condition. In this case, we set $\tilde{P}_{t_{0},V,\gamma}:= \hat{P}_{t_{0},\lambda,\tilde{\lambda}}$ and for the log-price process $X=\log(S)$ we see that $(X,V)$ solves the SDE
\begin{equation}\label{eq:transformed SDE 2}
d\begin{pmatrix}
X_{t} \\
V_{t}
\end{pmatrix}
 = \begin{pmatrix}
\hat{r}(t) - \frac{1}{2}\theta(t,V_{t})^{2} \\
\big(\zeta - \gamma\hat{\eta}\hat{\theta}\big)(t,V_{t})
\end{pmatrix}\,dt + \begin{pmatrix}
\theta(t,V_{t})\sqrt{1-\rho(t)^{2}} & \theta(t,V_{t})\rho(t) \\
0 & \eta(t,V_{t})
\end{pmatrix}\,d
\begin{pmatrix}
W_{t}^{(\lambda)} \\
W_{t}^{(\tilde{\lambda})}
\end{pmatrix}
\end{equation} 
for $t\in [t_{0},T]$ under $\tilde{P}_{t_{0},V,\gamma}$, where $\hat{\eta},\hat{\theta}:[0,T]\times\R\rightarrow\R$ are given by $\hat{\eta}(t,v):=\eta(t,v^{+})$ and $\hat{\theta}(t,v):=\theta(t,v^{+})$, since the values of $\eta$ and $\theta$ on $[0,T]\times ]-\infty,0[$ are irrelevant. In particular, as $\gamma=0$ feasible, there exists an equivalent local martingale measure, and we refer to Section~3 in~\cite{WonHey06} for a more specific analysis in the Heston model.

Under certain conditions, we now show that if $\mathscr{V}\in\mathscr{S}$ is of the form $\mathscr{V}_{t}$ $= u(t,X_{t},V_{t})$ for all $t\in [0,T]$ and some continuous function $u:[0,T]\times\R\times ]0,\infty[\rightarrow\R$, then it solves~\eqref{eq:pre-default valuation} as soon as $u$ is a mild solution to the parabolic semilinear PDE
\begin{equation}\label{eq:parabolic equation}
\begin{split}
\frac{\partial u}{\partial t}(t,x,v) + \mathscr{L}_{\hat{r},\zeta - \gamma\hat{\eta}\hat{\theta}}(u)(t,x,v) = - \hat{B}(t,e^{x},v,u(t,x,v))
\end{split}
\end{equation}
for all $(t,x,v)\in [0,T[\times\R\times ]0,\infty[$ with terminal value condition $u(T,\cdot,\cdot) = \phi(\exp(\cdot),\cdot)$. Thereby, $ \mathscr{L}_{\hat{r},\zeta - \gamma\hat{\eta}\hat{\theta}}$ stands for the differential operator~\eqref{eq:differential operator} when $(b,\zeta)$ is replaced by $(\hat{r},\zeta - \gamma\hat{\eta}\hat{\theta})$ with the arbitrarily chosen $\gamma\in\R_{+}$.

For introduction of the mild solution concept and the resulting characterisation of value processes, we first ensure that Proposition~\ref{pr:transformed SDE} applies to any two-dimensional SDE whose coefficients agree with those of~\eqref{eq:transformed SDE 2}, by introducing the following condition:
\begin{enumerate}[label=(V.\arabic*), ref=V.\arabic*, leftmargin=\widthof{(V.8)} + \labelsep]
\setcounter{enumi}{7}
\item\label{con:v.8} If $\gamma > 0$, then $\theta$ is continuous, $\theta(\cdot,0) = 0$ and there are $k_{\theta,\eta},l_{\theta,\eta},\lambda_{\theta,\eta}\in\mathscr{L}^{1}(\R)$, $\varepsilon_{0} > 0$ and $\lambda_{\eta,0}\in\mathscr{L}^{2}(\R_{+})$ such that
\begin{equation}\label{eq:additional conditions}
(\eta\theta)(\cdot,v) \geq k_{\theta,\eta} + l_{\theta,\eta}v\quad\text{and}\quad \mathrm{sgn}(v-\tilde{v})\big((\eta\theta)(\cdot,v) - (\eta\theta)(\cdot,\tilde{v})\big)\geq \lambda_{\theta,\eta}|v-\tilde{v}|
\end{equation}
for any $v,\tilde{v}\in\R_{+}$ and $|\eta(\cdot,v)| \leq \lambda_{\eta,0}v^{1/2}$ for all $v\in ]0,\varepsilon_{0}[$ a.e.
\end{enumerate}

\begin{Remark}
Let $\theta(t,\cdot)$ and $\eta(t,\cdot)$ be non-negative and increasing on $\R_{+}$ for a.e.~$t\in [0,T]$, $\eta(\cdot,0) = 0$ and~\eqref{con:v.4} be valid for $\rho_{1}(v)= v^{1/2}$ for all $v\in\R_{+}$. Then~\eqref{eq:additional conditions} and the succeeding bound on $\eta$ hold for $k_{\theta,\eta} = l_{\theta,\eta}=\lambda_{\theta,\eta} = 0$, $\varepsilon_{0} = 1$ and $\lambda_{\eta,0}=\lambda_{\eta,1}$.
\end{Remark}

Now we assume that~\eqref{con:v.1}-\eqref{con:v.8} hold, in which case the same conditions~\eqref{con:v.1}-\eqref{con:v.7} follow for $\hat{r}$ and $\zeta - \gamma\hat{\theta}\hat{\eta}$ instead of $b$ and $\zeta$, respectively, which shows that Proposition~\ref{pr:transformed SDE} also covers the type of two-dimensional SDE that we just derived.

Thus, for any $(s,x,v)\in [0,T]\times\R\times ]0,\infty[$ let $(\tilde{X}^{s,x,v},\tilde{V}^{s,v})$ be a strong solution to~\eqref{eq:transformed stochastic volatility model} on $[s,T]$ when $(b,\zeta)$ is replaced by $(\hat{r},\zeta-\gamma\hat{\eta}\hat{\theta})$ such that $(\tilde{X}_{s}^{s,x,v},\tilde{V}_{s}^{s,v}) = (x,v)$ a.s.~and $\tilde{V}^{s,v} > 0$. Further, denote the law of the continuous process
\begin{equation*}
[0,T]\times\Omega\rightarrow\R\times ]0,\infty[,\quad (t,\omega)\mapsto (\tilde{X}_{t\vee s}^{s,x,v},\tilde{V}_{t\vee s}^{s,v})(\omega)
\end{equation*}
by $\tilde{P}_{s,x,v}$ and recall the canonical process $(\hat{X},\hat{V})$ from Section~\ref{se:4.1}. Then, for a measurable function $\tilde{B}:[0,T]\times ]0,\infty[^{2}\times\R\rightarrow\R$ a \emph{mild solution} to the terminal value problem~\eqref{eq:parabolic equation} with $\tilde{B}$ instead of $\hat{B}$ is a measurable function $u:[0,T]\times\R\times ]0,\infty[\rightarrow\R$ for which
\begin{equation*}
\int_{s}^{T}\big|\tilde{B}\big(t,\exp(\hat{X}_{t}),\hat{V}_{t},u(t,\hat{X}_{t},\hat{V}_{t})\big)\big|\,dt
\end{equation*}
is finite and $\tilde{P}_{s,x,v}$-integrable such that the implicit integral equation
\begin{equation}\label{eq:mild solution}
\begin{split}
\tilde{E}_{s,x,v}\big[\phi(\exp(\hat{X}_{T}),\hat{V}_{T})\big] &= u(s,x,v) \\
&\quad- \tilde{E}_{s,x,v}\bigg[\int_{s}^{T}\tilde{B}\big(t,\exp(\hat{X}_{t}),\hat{V}_{t},u(t,\hat{X}_{t},\hat{V}_{t})\big)\,dt\bigg]
\end{split}
\end{equation}
holds for any $(s,x,v)\in [0,T]\times\R\times ]0,\infty[$. Provided that $\phi$ is continuous, we also recall that a \emph{classical solution} to this terminal value problem is a real-valued continuous function $u$ on $[0,T]\times\mathbb{R}\times ]0,\infty[$ that lies in $C^{1,2,2}([0,T[\times\mathbb{R}\times ]0,\infty[)$ and satisfies~\eqref{eq:parabolic equation} and $u(T,\cdot,\cdot)$ $= \phi(\exp(\cdot),\cdot)$. 

Thereby, we stress the fact that if $\tilde{B}$ satisfies the two conditions~\eqref{eq:affine growth condition} and~\eqref{eq:Lipschitz condition on compact sets} considered below and $\phi$ is bounded and continuous, then every bounded classical solution is also a mild solution. For instance, see Section 2.4 in~\cite{KalSch16} for a concise justification.

Next, as we will use any of the derived local martingale measures, let $(X^{s,x,v},V^{s,v})$ be a strong solution to~\eqref{eq:transformed stochastic volatility model} on $[s,T]$ satisfying $(X_{s}^{s,x,v},V_{s}^{s,v}) = (x,v)$ a.s.~and $V^{s,v} > 0$, and write $P_{s,x,v}$ for the law of the process $[0,T]\times\Omega\rightarrow\R\times ]0,\infty[$, $(t,\omega)\mapsto (X_{t\vee s}^{s,x,v},V_{t\vee s}^{s,v})(\omega)$ for each $(s,x,v)\in [0,T]\times\R\times ]0,\infty[$, just as in Proposition~\ref{pr:transformed SDE}. Let us also recall the regularity conditions that we used for this derivation:
\begin{enumerate}[label=(V.\arabic*), ref=V.\arabic*, leftmargin=\widthof{(V.9)} + \labelsep]
\setcounter{enumi}{8}
\item\label{con:v.9} $b$, $\hat{r}$, $\rho$ and $\theta$ are continuous and $\theta(\cdot,v) > 0$ for all $v > 0$.
\end{enumerate}

Now we come to one of the main results of our work, \emph{a characterisation of value processes in terms of mild solutions}, which motivates to call any mild solution to the terminal value problem~\eqref{eq:parabolic equation} \emph{a pre-default valuation function}. 

\begin{Theorem}\label{th:pre-default value process}
Let~\eqref{con:v.1}-\eqref{con:v.9} and~\eqref{con:p.1}-\eqref{con:p.4} hold, $\phi$ be bounded and
\begin{equation}\label{eq:pre-default condition 1}
\int_{s}^{T}\tilde{E}_{s,x,v}\big[|\hat{\pi}(t,\exp(\hat{X}_{t}),\hat{V}_{t})|\big]\,dt\quad\text{and}\quad E_{s,x,v}\bigg[\exp\bigg(\frac{\gamma^{2}}{2}\int_{s}^{T}\theta(t,\hat{V}_{t})^{2}\,dt\bigg)\bigg]
\end{equation}
be finite for all $(s,x,v)\in [0,T]\times\R\times ]0,\infty[$. Further, let $u\in C_{b}([0,T]\times\R\times ]0,\infty[)$ and define $\mathscr{V}\in\mathscr{S}$ by $\mathscr{V}_{t}:=u(t,X_{t},V_{t})$. Then the following two assertions hold:
\begin{enumerate}[(i)]
\item Suppose for each $(s,x,v)\in [0,T]\times\R\times ]0,\infty[$ that
\begin{equation}\label{eq:pre-default condition 2}
\sup_{t\in [s,T]}\tilde{E}_{s,x,v}\big[\big|\hat{H}\big(t,\exp(\hat{X}_{t}),\hat{V}_{t},u(t,\hat{X}_{t},\hat{V}_{t})\big)\big|\big] < \infty,
\end{equation}
and $\mathscr{V}$ solves~\eqref{eq:pre-default valuation} on $[s,T]$ whenever $(X,V)$ is a solution to~\eqref{eq:transformed stochastic volatility model} on $[s,T]$ with $(X^{s},V^{s}) = (x,v)$ a.s.~and $\tilde{P} = \tilde{P}_{s,V,\gamma}$. Then $u$ is a mild solution to~\eqref{eq:parabolic equation}.
\item Conversely, let $u$ be a mild solution to~\eqref{eq:parabolic equation} and $(s,x,v)\in [0,T]\times\R\times ]0,\infty[$ be such that~\eqref{eq:pre-default condition 2} holds. If $(X,V)$ solves~\eqref{eq:transformed stochastic volatility model} on $[s,T]$,
\begin{equation*}
(X^{s},V^{s}) = (x,v)\quad\text{a.s.},\quad\tilde{P} = \tilde{P}_{s,V,\gamma}
\end{equation*}
and $(\mathscr{F}_{t})_{t\in [0,T]}$ is the right-continuous filtration of the augmented natural filtration of $(X,V)$, then $\mathscr{V}$ is a solution to~\eqref{eq:pre-default valuation} on $[s,T]$.
\end{enumerate}
\end{Theorem}

\begin{Remark}
The integrability conditions in~\eqref{eq:pre-default condition 1} and~\eqref{eq:pre-default condition 2} on $\hat{\pi}$ and $\hat{H}$ are satisfied if there are $c_{\hat{\pi}}\in\mathscr{L}^{1}(\R_{+})$ and $c_{\hat{H}}\in\R_{+}$ such that
\begin{equation*}
|\hat{\pi}(\cdot,e^{x},v)| \leq c_{\hat{\pi}}(1 + x + v)\quad\text{for any $(x,v)\in\R\times ]0,\infty[$ a.e.}
\end{equation*}
and $|\hat{H}(\cdot,e^{x},v,y)| \leq c_{\hat{H}}(1 + x + v + |y|)$ for every $(x,v,y)\in\R\times ]0,\infty[\times\R$, as the two growth estimates~\eqref{eq:log-price growth estimate} and~\eqref{eq:variance growth estimate} show.
\end{Remark}

For an analysis of mild solutions to semilinear PDEs such as~\eqref{eq:parabolic equation} we will now apply results from~\cite{Kal20}. Let $\tilde{B}$ be a real-valued measurable function on $[0,T]\times ]0,\infty[^{2}\times\mathbb{R}$ for which the following \emph{affine growth condition} and \emph{Lipschitz condition on compact sets} hold: There is $c_{\tilde{B}}\in\mathscr{L}^{1}(\mathbb{R}_{+})$ such that
\begin{equation}\label{eq:affine growth condition}
|\tilde{B}(\cdot,e^{x},v,y)| \leq c_{\tilde{B}}(1 + |y|)
\end{equation}
for all $v > 0$ and any $x,y\in\mathbb{R}$ a.e. Moreover, for every $n\in\mathbb{N}$ there is $\lambda_{\tilde{B},n}\in\mathscr{L}^{1}(\mathbb{R}_{+})$ satisfying
\begin{equation}\label{eq:Lipschitz condition on compact sets}
|\tilde{B}(\cdot,e^{x},v,y) - \tilde{B}(\cdot,e^{x},v,\tilde{y})| \leq \lambda_{\tilde{B},n}|y-\tilde{y}|
\end{equation}
for all $(x,v)\in\mathbb{R}\times ]0,\infty[$ and any $y,\tilde{y}\in [-n,n]$ a.e. If $\phi$ is bounded, then Theorem~2.15 in~\cite{Kal20} yields a unique bounded mild solution $u_{\tilde{B},\phi}$ to the terminal value problem~\eqref{eq:parabolic equation} when $\hat{B}$ is replaced by $\tilde{B}$.

To ensure that~\eqref{eq:affine growth condition} and~\eqref{eq:Lipschitz condition on compact sets} hold for $\hat{B}$, we require that the dividend function $\hat{\pi}$ obeys a suitable bound and the pre-default hedging function $\hat{H}$ is both of affine growth and locally Lipschitz continuous in $y\in\R$, uniformly in $(t,s,v)\in [0,T]\times ]0,\infty[^{2}$:
\begin{enumerate}[label=(P.\arabic*), ref=P.\arabic*, leftmargin=\widthof{(P.5)} + \labelsep]
\setcounter{enumi}{4}
\item\label{con:p.5} There exist some $c_{\hat{\pi}}\in\mathscr{L}^{1}(\mathbb{R}_{+})$ and $c_{\hat{H}}\geq 0$ satisfying $|\hat{\pi}(\cdot,\exp(x),v)|\leq c_{\hat{\pi}}$ for every $(x,v)\in\mathbb{R}\times ]0,\infty[$ a.e.~and
\begin{equation*}
|\hat{H}(t,e^{x},v,y)| \leq c_{\hat{H}}(1 + |y|)
\end{equation*}
for all $(t,x,v,y)\in [0,T]\times\R\times]0,\infty[\times\R$. Further, for each $n\in\mathbb{N}$ there is $\lambda_{\hat{H},n}\geq 0$ satisfying
\begin{equation*}
|\hat{H}(t,e^{x},v,y) - \hat{H}(t,e^{x},v,\tilde{y})| \leq \lambda_{\hat{H},n}|y-\tilde{y}|
\end{equation*}
for each $(t,x,v)\in [0,T]\times\mathbb{R}\times ]0,\infty[$ and every $y,\tilde{y}\in [-n,n]$.
\end{enumerate}

Then an application of Theorem 2.15 in~\cite{Kal20} yields \emph{the existence and uniqueness of a mild solution, including a right-continuity and value analysis}. Thereby, $J$ denotes an interval in $\R$ with $\underline{d} := \inf J$ and $\overline{d}:=\sup J$.

\begin{Proposition}\label{pr:mild solution}
Let~\eqref{con:v.1}-\eqref{con:v.8} and~\eqref{con:p.1}-\eqref{con:p.5} be valid and $\phi$ be bounded. Then there is a unique bounded mild solution $u_{\phi}$ to~\eqref{eq:parabolic equation} that is right-continuous if $\phi$ is continuous. Moreover, if
\begin{equation}\label{eq:boundary condition}
\text{$\underline{d} > - \infty$ (resp.~$\overline{d} < \infty$) implies $\hat{B}(\cdot,e^{x},v,\underline{d}) \geq 0$ (resp.~$\hat{B}(\cdot,e^{x},v,\overline{d}) \leq 0$)}
\end{equation}
for all $(x,v)\in\mathbb{R}\times ]0,\infty[$ a.e., then $u_{\phi}$ takes all its values in $J$ as soon as $\phi$ does.
\end{Proposition}

\begin{Remark}
As payoff function, $\phi$ is modelled to be $\mathbb{R}_{+}$-valued. Thus, the pre-default valuation function $u_{\phi}$ is non-negative if
\begin{equation}\label{eq:non-negativity condition}
\hat{\pi}(\cdot,e^{x},v) \geq (\hat{r}-\hat{h}_{+})\hat{H}^{+}(\cdot,e^{x},v,0) - (\hat{r} - \hat{h}_{-})\hat{H}^{-}(\cdot,e^{x},v,0)
\end{equation}
for any $(x,v)\in\mathbb{R}\times ]0,\infty[$ a.e. In this case, $u_{\phi} > 0$ follows from $\phi > 0$.
\end{Remark}

Next, from Lemma 4.2 in~\cite{Kal20} we obtain \emph{a sensitivity analysis}. Let $\tilde{\phi}:]0,\infty[^{2}\rightarrow\R$ be measurable and bounded and suppose in the setting of Proposition~\ref{pr:mild solution} that $u_{\phi}$ and $u_{\tilde{B},\tilde{\phi}}$ are $J$-valued. This is the case if $\phi,\tilde{\phi}\in J$ and condition~\eqref{eq:boundary condition} holds not only for $\hat{B}$ but also for $\tilde{B}$. Then
\begin{equation}\label{eq:comparison}
\tilde{B}\geq\hat{B}\quad\text{on $[0,T]\times]0,\infty[^{2}\times J$}\quad\text{and}\quad \tilde{\phi}\geq\phi\quad\text{entails}\quad u_{\tilde{B},\tilde{\phi}} \geq u_{\phi}.
\end{equation}
So, $u_{\phi}$ depends on an increasing way on the dividend function $\hat{\pi}$, the remuneration rate $\hat{c}_{-}$, the pre-default funding rate $\hat{f}_{-}$ and the pre-default hedging rate $\hat{h}_{+}$. At the same time, \emph{the comparison}~\eqref{eq:comparison} shows that $u_{\phi}$ depends decreasingly on $\hat{c}_{+}$, $\hat{f}_{+}$ and $\hat{h}_{-}$.

If~\eqref{eq:non-negativity condition} holds, then the dependence is monotonically increasing on the term $g_{\mathcal{C}}$. To draw the same conclusion for $g_{\mathcal{I}}$, the estimate $1-\beta \geq\mathrm{LGD}_{\mathcal{I}}(1-\alpha)$ has to be valid when the investor $\mathcal{I}$ is a bank. This, for instance, is the case whenever the fractions $\alpha$ and $\beta$ coincide.

Finally, to give conditions under which the mild solution $u_{\tilde{B},\tilde{\phi}}$ turns into a classical one, we use Theorem 2.4 in~\cite{BecSch05}. This leads to a local Lipschitz condition on the coefficients in~\eqref{eq:transformed SDE 2} and the determinant of the diffusion coefficient should not vanish:
\begin{enumerate}[label=(V.\arabic*), ref=V.\arabic*, leftmargin=\widthof{(V.10)} + \labelsep]
\setcounter{enumi}{9}
\item\label{con:v.10} $\hat{r}$ and $\rho$ are Lipschitz continuous, $\zeta$, $\eta$ and $\theta$ are locally Lipschitz continuous on $[0,T]\times ]0,\infty[$ and $|\eta|$ and $|\theta|$ are positive on $[0,T]\times ]0,\infty[$.
\end{enumerate}

Moreover, $\tilde{B}$ is required to be of affine growth and locally Lipschitz continuous in $y\in\mathbb{R}$, uniformly in $(t,s,v)\in [0,T]\times ]0,\infty[^{2}$. So, the functions $c_{\tilde{B}}$ and $\lambda_{\tilde{B},n}$ in~\eqref{eq:affine growth condition} and~\eqref{eq:Lipschitz condition on compact sets}, respectively, should be bounded for all $n\in\N$. According to the assumptions in~\cite{BecSch05}, the function $\tilde{B}$ has to be continuously differentiable as well. However, a short investigation of the proof therein shows that local H\oe lder continuity suffices.

In summary, if~\eqref{con:v.1}-\eqref{con:v.8} and~\eqref{con:v.10} hold, $\tilde{B}$ satisfies these conditions and $\tilde{\phi}$ is continuous, then~\cite{BecSch05} asserts that the terminal value problem~\eqref{eq:parabolic equation} with $(\tilde{B},\tilde{\phi})$ instead of $(\hat{B},\phi)$ admits a unique bounded classical solution. This function must be $u_{\tilde{B},\tilde{\phi}}$, as Theorem 2.15 in~\cite{Kal20} yields uniqueness among bounded mild solutions. To ensure that $\hat{B}$ meets the same requirements as $\tilde{B}$, we require the following condition:
\begin{enumerate}[label=(P.\arabic*), ref=P.\arabic*, leftmargin=\widthof{(P.6)} + \labelsep]
\setcounter{enumi}{5}
\item\label{con:p.6} $\hat{\pi}$ is bounded. Further, $\hat{\pi}$, the rates $\hat{r}$, $\hat{c}_{+},\hat{c}_{-}$, $\hat{f}_{+},\hat{f}_{-}$, $\hat{h}_{+}$, $\hat{h}_{-}$, the functions $g_{\mathcal{I}}$, $g_{\mathcal{C}}$, the fractions $\alpha$, $\beta$ and the hedging function $\hat{H}$ are locally H\oe lder continuous.
\end{enumerate}

Then Theorem~2.4 in~\cite{BecSch05} combined with Theorem~2.15 in~\cite{Kal20} yield \emph{sufficient conditions for the unique mild solution} of Proposition~\ref{pr:mild solution} \emph{to be classical}.

\begin{Corollary}\label{co:classical solution}
Let~\eqref{con:v.1}-\eqref{con:v.8},~\eqref{con:v.10} and~\eqref{con:p.1}-\eqref{con:p.6} hold and $\phi$ be bounded and continuous. Then $u_{\phi}$ is the unique bounded classical solution to~\eqref{eq:parabolic equation}.
\end{Corollary}

\section{Proofs for the preliminary results and the market model}\label{se:5}

\subsection{Proofs for the representations of conditional expectations}

\begin{proof}[Proof of Lemma~\ref{le:essential identity}]
For $\tilde{A}\in\tilde{\mathscr{F}}_{s}$ there is $A\in\mathscr{F}_{s}$ such that $\{\rho > s\}\cap A = \{\rho > s\}\cap \tilde{A}$. Thus, the properties of conditional expectation yield that
\begin{align*}
E\big[X\mathbbm{1}_{\{\rho > t\}}P(\rho > s|\mathscr{F}_{s})\mathbbm{1}_{\tilde{A}}\big] &= E\big[E[X\mathbbm{1}_{\{\rho > t\}}|\mathscr{F}_{s}]P(\rho > s|\mathscr{F}_{s})\mathbbm{1}_{A}\big] \\
&= E\big[E[X\mathbbm{1}_{\{\rho > t\}}|\mathscr{F}_{s}]\mathbbm{1}_{\{\rho > s\}\cap\tilde{A}}\big].
\end{align*}
This implies the assertion.
\end{proof}

\begin{proof}[Proof of Corollary~\ref{co:identification}]
If $X = \tilde{X}$ a.s.~on $\{\rho > t\}$, then the $\mathscr{F}_{t}$-measurability of $X$ yields that $X G_{t}(\rho) = E[X\mathbbm{1}_{\{\rho > t\}}|\mathscr{F}_{t}] = E[\tilde{X}\mathbbm{1}_{\{\rho > t\}}|\mathscr{F}_{t}]$ a.s. 

Conversely, suppose that~\eqref{eq:conditional relation} holds. Then $\tilde{X}P(\rho > t|\tilde{\mathscr{F}}_{t})= XP(\rho > t|\tilde{\mathscr{F}_{t}})$ a.s.~on $\{G_{t}(\rho) > 0\}$, by Lemma~\ref{le:essential identity}. Thus,
\begin{equation*}
E[\tilde{X}\mathbbm{1}_{\{\rho > t\}\cap\{G_{t}(\rho) > 0\}\cap\tilde{A}}] = E[\tilde{X}P(\rho > t|\tilde{\mathscr{F}}_{t})\mathbbm{1}_{\{G_{t}(\rho) > 0\}\cap\tilde{A}}] = E[X\mathbbm{1}_{\{\rho > t\}\cap \{G_{t}(\rho) > 0\}\cap\tilde{A}}]
\end{equation*}
for any $\tilde{A}\in\tilde{\mathscr{F}}_{t}$. We first choose $\tilde{A}=\{n\geq \tilde{X} > X\}$ and then $\tilde{A}=\{\tilde{X} \leq X \leq n\}$ in this identity for each $n\in\N$ to infer that $X = \tilde{X}$ a.s.~on $\{\rho > t\}$, since
\begin{equation*}
P(\{\rho > t\}\cap \{G_{t}(\rho) = 0\}) = E[G_{t}(\rho)\mathbbm{1}_{\{G_{t}(\rho) = 0\}}] = 0.
\end{equation*}
\end{proof}

\begin{proof}[Proof of Lemma~\ref{le:stopped integral}]
Since $E[\tilde{X}_{t}\mathbbm{1}_{\{\rho > t\}}|\mathscr{F}_{t}] = X_{t}G_{t}(\rho)$ a.s.~for every $t\in [s,T]$, Fubini's theorem directly yields that
\begin{equation*}
E\bigg[\int_{s}^{T\wedge\rho} \tilde{X}_{t}\,dt\,\mathbbm{1}_{A}\bigg] = \int_{s}^{T}E\big[\tilde{X}_{t}\mathbbm{1}_{\{\rho > t\}}\mathbbm{1}_{A}\big]\,dt  = E\bigg[\int_{s}^{T}X_{t}G_{t}(\rho)\,dt\,\mathbbm{1}_{A}\bigg]
\end{equation*}
for each $A\in\mathscr{F}_{s}$. Thus, the claim holds.
\end{proof}

\begin{proof}[Proof of Lemma~\ref{le:conditional distribution}]
The system $\mathscr{E}$ of all sets $]s,\tilde{s}]\times A$, where $s,\tilde{s}\in [0,t]$ satisfy $s\leq \tilde{s}$ and $A\in\{\emptyset,\{\infty\}\}$, is an $\cap$-stable generator of $\mathscr{B}([0,t]\cup\{\infty\})$ and we readily check that
\begin{align*}
P(s_{1} < \sigma\leq t_{1},s_{2} < \tau \leq t_{2}|\mathscr{F}_{t}) &= P(s_{1} < \sigma\leq t_{1}|\mathscr{F}_{t}) P(s_{2} < \tau\leq t_{2}|\mathscr{F}_{t})\quad\text{a.s.}
\end{align*}
for any $s_{1},t_{1},s_{2},t_{2}\in [0,t]\cup\{\infty\}$ with $s_{1}\leq t_{1}$ and $s_{2}\leq t_{2}$. In particular, the $d$-system of all $C\in\mathscr{B}(([0,t]\cup\{\infty\})^{2})$ for which~\eqref{eq:conditional distribution} holds includes $\mathscr{E}\times\mathscr{E}$. Hence, the claim follows from the monotone class theorem.
\end{proof}

\begin{proof}[Proof of Proposition~\ref{pr:conditional integral representation}]
For fixed $\tilde{s}\in ]s,T[$ and every $n\in\N$ let $\mathbb{T}_{n}$ be a partition of $[\tilde{s},T]$ of the form $\mathbb{T}_{n}=\{t_{0,n},\dots,t_{k_{n},n}\}$ for some $k_{n}\in\N$ and $t_{0,n},\dots,t_{k_{n},n}\in [\tilde{s},T]$ with $\tilde{s} = t_{0,n} < \cdots < t_{k_{n},n} = T$. We denote its mesh by
\begin{equation*}
|\mathbb{T}_{n}| = \max_{i\in\{0,\dots,k_{n}-1\}} (t_{i+1,n} - t_{i,n})
\end{equation*}
 and assume that the resulting sequence $(\mathbb{T}_{n})_{n\in\N}$ is refining, which means that $\mathbb{T}_{n}\subset\mathbb{T}_{n+1}$ for all $n\in\N$, and satisfies $
\lim_{n\uparrow\infty} |\mathbb{T}_{n}| = 0$. Then the sequences $(\null_{n}X)_{n\in\mathbb{N}}$ and $(\null_{n}G(\tau))_{n\in\mathbb{N}}$ of left-continuous $(\mathscr{F}_{t})_{t\in [0,T]}$-adapted processes defined via
\begin{equation*}
\null_{n}X_{t} := \sum_{i=0}^{k_{n}-1} X_{t_{i,n}}\mathbbm{1}_{]t_{i,n},t_{i+1,n}]}(t)
\quad\text{and}\quad \null_{n}G_{t}(\tau):= \sum_{i=0}^{k_{n}-1} G_{t_{i,n}}(\tau)\mathbbm{1}_{]t_{i,n},t_{i+1,n}]}(t)
\end{equation*}
satisfy $\lim_{n\uparrow\infty} \null_{n}X_{t}(\omega) = X_{t}(\omega)$ and $\lim_{n\uparrow\infty} \null_{n}G_{t}(\tau)(\omega) = G_{t}(\tau)(\omega)$ for all $(t,\omega)\in ]\tilde{s},T]\times\Omega$ for which $X(\omega)$ and $G(\tau)(\omega)$ are left-continuous at $t$. For the decreasing sequence $(\tau_{n})_{n\in\mathbb{N}}$ of $[0,T]\cup\{\infty\}$-valued random variables given by
\begin{equation*}
\tau_{n}(\omega) := \sum_{i=0}^{k_{n}-1}t_{i+1,n}\mathbbm{1}_{\{t_{i,n} < \tau \leq t_{i+1,n}\}}(\omega),\quad\text{if $\tau(\omega) < \infty$,}
\end{equation*}
and $\tau_{n}(\omega):= \infty$, if $\tau(\omega) = \infty$, we have $\inf_{n\in\mathbb{N}}\tau_{n} = \tau$ on $\{\tilde{s} < \tau\}$. For given $n\in\mathbb{N}$ we define a left-continuous $(\tilde{\mathscr{F}}_{t})_{t\in [0,T]}$-adapted process $\null_{n}\tilde{X}$ by using the definition of $\null_{n}X$ when $X$ is replaced by $\tilde{X}$ and compute that
\begin{align*}
E[\null_{n}\tilde{X}_{T}^{\sigma}\mathbbm{1}_{\{\tilde{s} < \sigma\leq T\wedge\tau_{n}\}}|\mathscr{F}_{s}] &= \sum_{i=0}^{k_{n}-1}E[X_{t_{i,n}}P(t_{i,n} < \sigma\leq t_{i+1,n},\sigma \leq \tau_{n}|\mathscr{F}_{t_{i,n}})|\mathscr{F}_{s}]\\
&= - E\bigg[\int_{]\tilde{s},T]}\null_{n}X_{t}\,\null_{n}G_{t}(\tau)\,dG_{t}(\sigma)\,\bigg|\,\mathscr{F}_{s}\bigg]\quad\text{a.s.}
\end{align*}
Indeed, the $(\mathscr{F}_{t})_{t\in [0,T]}$-conditional independence of $\sigma$ and $\tau$ gives
\begin{align*}
P(t_{i,n} < \sigma\leq t_{i+1,n}, \sigma \leq \tau_{n}|\mathscr{F}_{t_{i,n}}) &= P(t_{i,n} < \sigma \leq t_{i+1,n}, \tau = \infty|\mathscr{F}_{t_{i,n}})\\
&\quad + \sum_{j=i}^{k_{n}-1} P(t_{i,n} < \sigma\leq t_{i+1,n}, t_{j,n} < \tau\leq t_{j+1,n}|\mathscr{F}_{t_{i,n}})\\
&= - E[G_{t_{i,n}}(\tau)(G_{t_{i+1,n}}(\sigma) - G_{t_{i,n}}(\sigma))|\mathscr{F}_{t_{i,n}}]\quad\text{a.s.}
\end{align*}
for each $i\in\{0,\dots,k_{n}-1\}$. By construction, $|\null_{n}X_{t}|\leq \sup_{\tilde{t}\in ]s,T]}|X_{\tilde{t}}|$ for every $n\in\mathbb{N}$ and all $t\in ]\tilde{s},T]$. Therefore, dominated convergence yields that
\begin{equation*}
\lim_{n\uparrow\infty} \int_{]\tilde{s},T]}\null_{n}X_{t}\,\null_{n}G_{t}(\tau)\,dG_{t}(\sigma) = \int_{]\tilde{s},T]} X_{t}G_{t}(\tau)\,dG_{t}(\sigma).
\end{equation*}
Since $|\int_{]\tilde{s},T]}\null_{n}X_{t}\,\null_{n}G_{t}(\tau)\,dG_{t}(\sigma)|$ does not exceed $\sup_{t\in ]s,T]}|X_{t}|G_{t}(\tau)(V_{T}(\sigma)-V_{s}(\sigma))$ for each $n\in\mathbb{N}$, dominated convergence also implies that
\begin{equation*}
E[\tilde{X}_{T}^{\sigma}\mathbbm{1}_{\{\tilde{s} < \sigma\leq T\wedge\tau\}}|\mathscr{F}_{s}] = - E\bigg[\int_{]\tilde{s},T]} X_{t}G_{t}(\tau)\,dG_{t}(\sigma)\,\bigg|\,\mathscr{F}_{s}\bigg]\quad\text{a.s.}
\end{equation*}
Finally, for any sequence $(s_{n})_{n\in\mathbb{N}}$ in $]s,T[$ that converges to $s$, we have $\lim_{n\uparrow\infty} \mathbbm{1}_{]s_{n},T]}(\sigma)$ $= \mathbbm{1}_{]s,T]}(\sigma)$ and $\lim_{n\uparrow\infty} \int_{]s,s_{n}]}X_{t}G_{t}(\tau)\,dG_{t}(\sigma) = 0$. Hence, the claim follows from a final application of the Dominated Convergence Theorem.
\end{proof}

\subsection{Proofs for the conditionally independent hitting times}

\begin{proof}[Proof of Lemma~\ref{le:hitting time}]
For any $\omega\in\Omega$ we have $\tau_{j}(\omega)\leq t$ if and only if $X_{t}^{(j)}(\omega)\geq \xi_{j}(\omega)$, since the increasing function $X^{(j)}(\omega)$ is right-continuous. Hence, $\{\tau_{j} \leq t\}$ coincides with $\{\xi_{j} \leq X_{t}^{(j)}\}$ and lies in $\tilde{\mathscr{F}}_{t}$.

Further, this entails that $\{\tau_{1} > s_{1},\dots,\tau_{j} > s_{j}\} = \{\xi_{1} > X_{s_{1}}^{(1)},\dots,\xi_{j} > X_{s_{j}}^{(j)}\}$. For this reason, the independence of $\xi$ and $\mathscr{F}_{T}$ and the independence of $\xi_{1},\dots,\xi_{n}$ yield that
\begin{align*}
P(\tau_{1} > s_{1},\dots \tau_{j} > s_{j}|\mathscr{F}_{t}) &= P(\xi_{1} > x_{1},\dots,\xi_{j} > x_{j})_{|(x_{1},\dots,x_{n}) = (X_{s_{1}}^{(1)},\dots,X_{s_{j}}^{(j)})}\\
&= G_{1}(X_{s_{1}}^{(1)})\cdots G_{n}(X_{s_{j}}^{(j)})
\quad\text{a.s.}
\end{align*}
These considerations imply all the assertions.
\end{proof}

\begin{proof}[Proof of Lemma~\ref{le:minimal hitting time}]
(i) By~\eqref{eq:hitting time conditional distribution}, we have $G_{s}(\rho) > 0$ a.s.~if and only if $G_{i}(X_{s}^{(i)}) > 0$ a.s.~for any $i\in\{1,\dots,m\}$. Therefore, the definition of the essential supremum implies the claim. 

(ii) From the representation $P(\rho > s) = E[ \prod_{i=1}^{m}G_{i}(X_{s}^{(i)})]$ we infer that $P(\rho > s) = 1$ if and only if $G_{i}(X_{s}^{(i)}) = 1$ a.s.~for all $i\in\{1,\dots,m\}$. In addition, $P(\rho > s) = 0$ if only if $G_{i}(X_{s}^{(i)}) = 0$ for some $i\in\{1,\dots,m\}$ a.s., which yields the assertions.

(iii) Let $(s_{n})_{n\in\N}$ be an increasing sequence in $[0,s[$ that converges to $s$. Then the $\sigma$-continuity of probability measures and dominated convergence lead to
\begin{equation*}
P(\rho \geq s) = \lim_{n\uparrow\infty} P(\rho > s_{n}) = \lim_{n\uparrow\infty} E\bigg[\prod_{i=1}^{m}G_{i}(X_{s_{n}}^{(i)})\bigg] = E\bigg[\prod_{i=1}^{m}G_{i}(X_{s}^{(i)})\bigg] = P(\rho > s),
\end{equation*}
since $\lim_{n\uparrow\infty} G_{i}(X_{s_{n}}^{(i)}) = G_{i}(X_{s}^{(i)})$ a.s.~for each $i\in\{1,\dots,m\}$, due to the a.s.~left-continuity of $G_{i}(X^{(i)})$. Hence, $P(\rho = s) = P(\rho\geq s) - P(\rho > s) = 0$.
\end{proof}

\begin{proof}[Proof of Proposition~\ref{pr:density formula}]
Since $P(\rho > 0) = E[G_{0}(\rho)] = G_{1}(\hat{x}_{1})\cdots G_{m}(\hat{x}_{m})$ and $\Lambda_{0} = \Omega$, we merely need to check the asserted formula for $t > 0$.

For this purpose, we define an $]0,1[$-valued continuously differentiable function $\varphi$ on $]a_{1},b_{1}[\times\cdots\times]a_{m},b_{m}[$ by $\varphi(x) := \prod_{i=1}^{m} G_{i}(x_{i})$ and observe that the path
\begin{equation*}
]0,t]\rightarrow [0,\infty[^{m},\quad s\mapsto X_{s}(\omega)
\end{equation*}
is absolutely continuous and takes all its values in $]a_{1},b_{1}[\times\cdots\times ]a_{m},b_{m}[$ for each $\omega\in\Lambda_{t}$, as $a_{i}\leq \hat{x}_{i} < X_{s}^{(i)}(\omega)\leq X_{t}^{(i)}(\omega) < b_{i}$ for every $s\in ]0,t]$ and all $i\in\{1,\dots,m\}$. Thus,
\begin{align*}
\varphi(X_{s}) - \varphi(X_{t}) = - \sum_{i=1}^{m}\int_{s}^{t}\frac{\partial\varphi}{\partial x_{i}}(X_{\tilde{s}})\,dX_{\tilde{s}}^{(i)} = -\int_{s}^{t}\varphi(X_{\tilde{s}})\sum_{i=1}^{m}\lambda_{\tilde{s}}^{(i)}\bigg(\frac{G_{i}'}{G_{i}}\bigg)(X_{\tilde{s}}^{(i)})\,d\tilde{s}
\end{align*}
on $\Lambda_{t}$, by the Fundamental Theorem of Calculus for Lebesgue-Stieltjes integrals. Now we take expectations and apply Fubini's theorem to the effect that
\begin{equation}\label{eq:auxiliary density formula}
E[\varphi(X_{s});\Lambda_{t}] = P(\rho > t) - \int_{s}^{t}E\bigg[\varphi(X_{\tilde{s}})\sum_{i=1}^{m}\lambda_{\tilde{s}}^{(i)}\bigg(\frac{G_{i}'}{G_{i}}\bigg)(X_{\tilde{s}}^{(i)});\Lambda_{t}\bigg]\,d\tilde{s}.
\end{equation}
Thereby, we used that $E[\varphi(X_{t});\Lambda_{t}] = E[G_{t}(\rho)\mathbbm{1}_{\Lambda_{t}}] = P(\rho > t)$. From the right-continuity of $G_{1},\dots,G_{m}$ and monotone convergence we infer that
\begin{equation*}
G_{1}(\hat{x}_{1})\cdots G_{m}(\hat{x}_{m})P(\Lambda_{t}) = \lim_{n\uparrow\infty}E[\varphi(X_{s_{n}});\Lambda_{t}]
\end{equation*}
for any decreasing zero sequence $(s_{n})_{n\in\N}$ in $]0,t]$. Finally, we replace $s$ by $s_{n}$ in~\eqref{eq:auxiliary density formula} for each $n\in\N$ to deduce the claim from monotone convergence.
\end{proof}

\subsection{Proofs for the market model with default}\label{se:5.3}

\begin{proof}[Proof of Proposition~\ref{pr:pre-default valuation}]
Because~\eqref{eq:con.2} is equivalent to~\eqref{con:c.2}, Lemma~\ref{le:stopped integral} entails for the contractual cash flows and the collateral, funding and hedging costs and benefits that
\begin{align*}
\tilde{E}\big[\null_{\mathrm{con}}\mathrm{CF}_{s} &- \null_{\mathrm{col}}\mathrm{C}_{s}(\mathscr{V}) - \null_{\mathrm{fun}}\mathrm{C}_{s}(\tilde{\mathscr{V}}) - \null_{\mathrm{hed}}\mathrm{C}_{s}(\tilde{\mathscr{V}})\big|\mathscr{F}_{s}\big]\\
&= \tilde{E}\bigg[D_{s,T}(r)\Phi(S,V)G_{T}(\tau) + \int_{s}^{T}D_{s,t}(r)\null_{0}\B_{t}(\mathscr{V})G_{t}(\tau)\,dt\,\bigg|\,\mathscr{F}_{s}\bigg]\quad\text{a.s.}
\end{align*}
Now we consider the cash flows $\null_{\mathrm{def}}\mathrm{CF}(\mathscr{V},\tilde{\mathscr{V}})$ on default, given by~\eqref{eq:cash flows and costs 5}, and note that $\{\tau_{i} < \tau_{j}\}\cap\{s < \tau < T\}$ $= \{s < \tau_{i} < \tau_{j}\wedge T\}$ for both $i,j\in\{\mathcal{I},\mathcal{C}\}$ with $i\neq j$. By~\eqref{eq:con.3} and~\eqref{con:c.3}, we obtain
\begin{align*}
\tilde{E}\big[\null_{\mathrm{def}}\mathrm{CF}_{s}(\mathscr{V},\tilde{\mathscr{V}})\big|\mathscr{F}_{s}\big] &= - \tilde{E}\bigg[\int_{s}^{T}D_{s,t}(r)\null_{\mathcal{I}}\B_{t}(\mathscr{V})G_{t}(\tau_{\mathcal{C}})\,dG_{t}(\tau_{\mathcal{I}})\,\bigg|\,\mathscr{F}_{s}\bigg]\\
&\quad  - \tilde{E}\bigg[\int_{s}^{T}D_{s,t}(r)\null_{\mathcal{C}}\B_{t}(\mathscr{V})G_{t}(\tau_{\mathcal{I}})\,dG_{t}(\tau_{\mathcal{C}})\,\bigg|\,\mathscr{F}_{s}\bigg]\quad\text{a.s.}
\end{align*}
from two applications of Proposition~\ref{pr:conditional integral representation}, since~\eqref{eq:market condition 2} ensures $\tau_{\mathcal{I}} \neq \tau_{\mathcal{C}}$ a.s.~on $\{\tau < \infty\}$. This shows the assertion.
\end{proof}

\begin{proof}[Proof of Proposition~\ref{pr:martingale characterisation}]
According to our considerations preceding Proposition~\ref{pr:pre-default valuation}, it follows immediately from~\eqref{con:c.2} and~\eqref{con:c.3} that the integral
\begin{equation*}
\int_{0}^{T}D_{0,t}(r)\big(|\null_{0}\B_{t}(\mathscr{V})|G_{t}(\tau)\,dt + |\null_{\mathcal{I}}\B_{t}(\mathscr{V})|G_{t}(\tau_{\mathcal{C}})\,dV_{t}(\tau_{\mathcal{I}}) + |\null_{\mathcal{C}}\B_{t}(\mathscr{V})|G_{t}(\tau_{\mathcal{I}})\,dV_{t}(\tau_{\mathcal{C}})\big),
\end{equation*}
which bounds $\int_{0}^{t}D_{0,s}(r)\,dA_{s}(\mathscr{V})$ for any $t\in [0,T]$, is $\tilde{P}$-integrable. This justifies that the process in~\eqref{eq:finite variation process 2} is indeed integrable and we may turn to the second assertion. 

For only if we observe that $D_{0,s}(\tau)\mathscr{V}_{s}G_{s}(\tau)$ and hence, $\null_{\mathscr{V}}M_{s}$ is integrable for fixed $s\in [t_{0},T]$. Indeed, we get
\begin{equation*}
\tilde{E}\big[D_{0,s}(r)|\mathscr{V}_{s}|G_{s}(\tau)\big] \leq \tilde{E}\bigg[\bigg|D_{0,T}(r)\Phi(S,V)G_{T}(\tau) + \int_{s}^{T}D_{0,t}(r)\,dA_{t}(\mathscr{V})\bigg|\bigg] < \infty,
\end{equation*}
by using that $D_{0,s}(r)D_{s,t}(r) = D_{0,t}(r)$ for any $t\in [s,T]$. In addition, the $\mathscr{F}_{s}$-measurability of $\int_{0}^{s}D_{0,\tilde{s}}(r)\,dA_{\tilde{s}}(\mathscr{V})$ yields that
\begin{align*}
\null_{\mathscr{V}}M_{s} &= D_{0,s}(r)\mathscr{V}_{s}G_{s}(\tau) + \int_{0}^{s}D_{0,\tilde{s}}(r)\,dA_{\tilde{s}}(\mathscr{V})\\
&= \tilde{E}\bigg[D_{0,T}(r)\Phi(S,V)G_{T}(\tau) + \int_{0}^{T}D_{0,t}(r)\,dA_{t}(\mathscr{V})\,\bigg|\,\mathscr{F}_{s}\bigg] = \tilde{E}\big[\null_{\mathscr{V}}M_{T}|\mathscr{F}_{s}\big]\quad\text{a.s.},
\end{align*}
which implies the martingale property of $\null_{\mathscr{V}}M$ relative to $(\mathscr{F}_{t})_{t\in [t_{0},T]}$. For if the integrability of $\null_{\mathscr{V}}M_{s}$ entails that of $D_{0,s}(r)\mathscr{V}_{s}G_{s}(\tau)$ and we have
\begin{equation}\label{eq:auxiliary conditional representation}
D_{0,s}(r)\mathscr{V}_{s}G_{s}(\tau) = \tilde{E}\bigg[D_{0,T}(r)\mathscr{V}_{T}G_{T}(\tau) + \int_{s}^{T}D_{0,t}(r)\, dA_{t}(\mathscr{V})\bigg|\mathscr{F}_{s}\bigg]\quad\text{a.s.},
\end{equation}
because $\null_{\mathscr{V}}M_{s} = \tilde{E}[\null_{\mathscr{V}}M_{T}|\mathscr{F}_{s}]$ a.s. Consequently, it holds that 
\begin{equation*}
\tilde{E}\big[|\mathscr{V}_{s}|G_{s}(\tau)\big] \leq \tilde{E}\bigg[\bigg|D_{s,T}(r)\Phi(S,V)G_{T}(\tau) + \int_{s}^{T}D_{s,t}(r)\,dA_{t}(\mathscr{V})\bigg|\bigg] < \infty.
\end{equation*}
This allows us to multiplicate both sides in~\eqref{eq:auxiliary conditional representation} with $D_{0,s}(-r)$ and, since $\mathscr{V}_{T}G_{T}(\tau)$ may be replaced by $\Phi(S,V)G_{T}(\tau)$, we see that $\mathscr{V}$ solves~\eqref{eq:pre-default valuation}.
\end{proof}

\begin{proof}[Proof of Proposition~\ref{pr:backward stochastic representation}]
We recall from~\eqref{eq:finite variation process} that the continuous process~\eqref{eq:finite variation process 2} is of finite variation, just as the process $[0,T]\times\Omega\rightarrow ]0,\infty[$, $(t,\omega)\mapsto D_{0,t}(r)(\omega)$. Hence, the first claim follows directly from It{\^o}'s formula.

Let us now assume that $\null_{\mathscr{V}}M$ is a continuous $(\mathscr{F}_{t})_{t\in [t_{0},T]}$-semimartingale and choose $t\in [t_{0},T]$. Then from It{\^o}'s product rule we infer that
\begin{align*}
\mathscr{V}_{t}G_{t}(\tau) - \mathscr{V}_{s}G_{s}(\tau) = - \int_{s}^{t}\,\big(dA_{\tilde{s}}(\mathscr{V}) - r_{\tilde{s}}\mathscr{V}_{\tilde{s}}G_{\tilde{s}}(\tau)\,d\tilde{s}\big) + \int_{s}^{t}D_{0,\tilde{s}}(-r)\,d\null_{\mathscr{V}}M_{\tilde{s}}
\end{align*}
for all $s\in [t_{0},t]$ a.s., which gives the first identity. In the case that $G(\tau) > 0$ It{\^o}'s product rule also shows that $\mathscr{V}$ is a continuous $(\mathscr{F}_{t})_{t\in [t_{0},T]}$-semimartingale and
\begin{equation*}
\mathscr{V}_{t} - \mathscr{V}_{s} = \int_{s}^{t}\frac{1}{G_{\tilde{s}}(\tau)}\,d\mathscr{V}_{\tilde{s}}G_{\tilde{s}}(\tau) - \int_{s}^{t}\frac{\mathscr{V}_{\tilde{s}}}{G_{\tilde{s}}(\tau)}\,dG_{\tilde{s}}(\tau)
\end{equation*}
for every $s\in [t_{0},t]$ a.s. Thus, the second identity~\eqref{eq:backward stochastic representation} follows from the first~\eqref{eq:auxiliary backward integral representation} and the definition of $A$ in~\eqref{eq:finite variation process}.

Finally, suppose there is a continuous $(\mathscr{F}_{t})_{t\in [t_{0},T]}$-semimartingale $M$ such that~\eqref{eq:backward stochastic representation} holds when $\null_{\mathscr{V}}M$ is replaced by $M$. Then from It{\^o}'s formula we obtain that
\begin{align*}
\null_{\mathscr{V}}M_{t} - \null_{\mathscr{V}}M_{s} &=  - \int_{s}^{t}D_{0,\tilde{s}}(r)G_{\tilde{s}}(\tau)\mathscr{V}_{\tilde{s}}\bigg(r_{\tilde{s}}\,d\tilde{s} - \frac{1}{G_{\tilde{s}}(\tau_{\mathcal{I}})}\,dG_{\tilde{s}}(\tau_{\mathcal{I}}) - \frac{1}{G_{\tilde{s}}(\tau_{\mathcal{I}})}\,dG_{\tilde{s}}(\tau_{\mathcal{C}})\bigg)\\
&\quad + \int_{s}^{t}D_{0,\tilde{s}}(r)G_{\tilde{s}}(\tau)\bigg(d\mathscr{V}_{\tilde{s}} + \frac{1}{G_{\tilde{s}}(\tau)}dA_{\tilde{s}}(\mathscr{V})\bigg) = M_{t} - M_{s}
\end{align*}
for each $s\in [t_{0},t]$ a.s. This shows that if $M_{t_{1}} = \null_{\mathscr{V}}M_{t_{1}}$ a.s., then $M$ and $\null_{\mathscr{V}}M$ must be indistinguishable, as claimed.
\end{proof}

\begin{proof}[Proof of Corollary~\ref{co:semimartingale characterisation}]
For only if $D_{0,t_{0}}(r)\mathscr{V}_{t_{0}}G_{t_{0}}(\tau)$ is integrable, since the process~\eqref{eq:finite variation process 2} satisfies this property and $\null_{\mathscr{V}}M$ is a continuous $(\mathscr{F}_{t})_{t\in [t_{0},T]}$-martingale, by Proposition~\ref{pr:martingale characterisation}.
Moreover, Proposition~\ref{pr:backward stochastic representation} states that~\eqref{eq:backward stochastic representation} is valid for all $s\in [t_{0},T]$ a.s.

For if $\mathscr{V}$ is an $(\mathscr{F}_{t})_{t\in [t_{0},T]}$-semimartingale, since~\eqref{eq:backward stochastic representation} holds for any $s\in [t_{0},T]$ a.s.~when $\null_{\mathscr{V}}M$ is replaced by the $(\mathscr{F}_{t})_{t\in [t_{0},T]}$-martingale $M - M_{t_{0}} + \null_{\mathscr{V}}M_{t_{0}}$. Hence, the uniqueness assertion of Proposition~\ref{pr:backward stochastic representation} yields $M_{t} - M_{t_{0}} = \null_{\mathscr{V}}M_{t} - \null_{\mathscr{V}}M_{t_{0}}$ for each $t\in [t_{0},T]$ a.s.~and Proposition~\ref{pr:martingale characterisation} shows that $\mathscr{V}$ solves~\eqref{eq:pre-default valuation}.
\end{proof}

\section{Proofs for the volatility model and the valuation PDE}\label{se:6}

\subsection{Proofs for the price process and its quasi variance}

\begin{proof}[Proof of Lemma \ref{le:price process representation}]
Regarding uniqueness, suppose that $S$ and $\tilde{S}$ are two solutions to the first SDE in~\eqref{eq:stochastic volatility model} with $S_{t_{0}} = \tilde{S}_{t_{0}}$ a.s. For fixed $n\in\mathbb{N}$ there is $k_{\theta}\in\mathscr{L}^{2}(\R_{+})$ such that $|\theta(\cdot,v)|\leq k_{\theta}$ for any $v\in [1/n,n]$ a.e. Thus, It{\^o}'s formula yields that
\begin{equation*}
E\big[|S_{t}^{\tau_{n}} - \tilde{S}_{t}^{\tau_{n}}|^{2}\big] \leq \int_{t_{0}}^{t}(2b^{+} + k_{\theta}^{2})(s) E\big[|S_{s}^{\tau_{n}} - \tilde{S}_{s}^{\tau_{n}}|^{2}\big]\,ds\quad\text{for all $t\in [t_{0},T]$}
\end{equation*}
and the stopping time $\tau_{n}:=\inf\{t\in [t_{0},T]\,|\, V_{t}\notin ]1/n,n[\text{ or } |S_{t}|\vee |\tilde{S}_{t}|\geq n\}$. By Gronwall's inequality and $\sup_{n\in\mathbb{N}} \tau_{n} = \infty$, the continuous processes $S$ and $\tilde{S}$ are indistinguishable.

Regarding existence and the claimed representation, note that $\int_{t_{0}}^{T}\theta(s,V_{s})^{2}\,ds < \infty$, as $V([t_{0},T]\times\{\omega\})$ is compact in $]0,\infty[$ for any $\omega\in\Omega$. For this reason, the stochastic and Lebesgue integrals in~\eqref{eq:stochastic exponential} are well-defined. Hence, if $S$ is an adapted continuous process for which~\eqref{eq:stochastic exponential} holds, then It{\^o}'s formula shows that it solves the first SDE in~\eqref{eq:stochastic volatility model}.

The second claim is a direct consequence of Novikov's condition, which entails that the continuous local martingale $\exp(\int_{t_{0}}^{\cdot}\theta(s,V_{s})\,d\hat{W}_{s} - \frac{1}{2}\int_{t_{0}}^{\cdot}\theta(s,V_{s})^{2}\,ds)$ is a martingale and hence,
\begin{equation*}
E[S_{t}]e^{-\int_{t_{0}}^{t}b(s)\,ds} = E\bigg[\chi E\bigg[\exp\bigg(\int_{t_{0}}^{t}\theta(s,V_{s})\,d\hat{W}_{s} - \frac{1}{2}\int_{t_{0}}^{t}\theta(s,V_{s})^{2}\,ds\bigg)\bigg|\mathscr{F}_{t_{0}}\bigg]\bigg] = E[\chi]
\end{equation*}
for all $t\in [t_{0},T]$. Thereby, we used the fact that $S_{t}$ is integrable, which follows from the same reasoning if  $\chi$ is split into its positive and negative part.
\end{proof}

\begin{proof}[Proof of Lemma~\ref{le:log-price comparison estimate}]
By using the sublinear growth and H\oe lder condition for $\theta$, we notice that $|\theta(\cdot,v)^{2} - \theta(\cdot,\tilde{v})^{2}| \leq \big(2 k_{\theta} + \lambda_{\theta}(v^{1/2}+ \tilde{v}^{1/2})\big)\lambda_{\theta}|v-\tilde{v}|^{1/2}$ for any $v,\tilde{v} > 0$ a.e. Hence, the Cauchy-Schwarz inequality implies that
\begin{equation*}
\frac{1}{2}E\bigg[\int_{t_{0}}^{t\wedge\sigma}|\theta(s,V_{s})^{2} - \theta(s,\tilde{V}_{s})^{2}|\,ds\bigg]\leq c_{2,1}(t)\sup_{s\in [t_{0},t]}\big(1 + E\big[V_{s}^{\sigma}\big] + E\big[\tilde{V}_{s}^{\sigma}\big]\big)^{\frac{1}{2}} E\big[|V_{s}^{\sigma} - \tilde{V}_{s}^{\sigma}|\big]^{\frac{1}{2}}
\end{equation*}
with $c_{2,1}:[t_{0},T]\rightarrow\R_{+}$ given by $c_{2,1}(t) := \int_{t_{0}}^{t}(k_{\theta}(s) + \lambda_{\theta}(s))\lambda_{\theta}(s)\,ds$. Moreover, from the Burkholder-Davis-Gundy inequality we infer that
\begin{equation*}
E\bigg[\sup_{\tilde{s}\in [t_{0},t]}\bigg|\int_{t_{0}}^{\tilde{s}\wedge\sigma}\theta(s,V_{s}) - \theta(s,\tilde{V}_{s})\,d\hat{W}_{s}\bigg|\bigg] \leq c_{2,2}(t)\sup_{s\in [t_{0},t]} E\big[|V_{s}^{\sigma} - \tilde{V}_{s}^{\sigma}|\big]^{\frac{1}{2}},
\end{equation*}
where $c_{2,2}:[t_{0},T]\rightarrow\R_{+}$ is defined via $c_{2,2}(t):= 2(\int_{t_{0}}^{t}\lambda_{\theta}(s)^{2}\,ds)^{1/2}$. Because we have $c_{2}(t_{0},\cdot)= c_{2,1} + c_{2,2}$, the desired estimate follows.
\end{proof}

The proof of Proposition~\ref{pr:positivity of the variance process}, which extends several ideas from the proof of Theorem~2.2 in~\cite{MisPos08}, relies on the construction of the following function.

\begin{Lemma}\label{le:positivity of the variance process}
Under~\eqref{con:v.7}, there is $\varphi\in C^{1}(]0,\infty[)$ satisfying the following two conditions:
\begin{enumerate}[(i)]
\item $\lim_{x\downarrow 0} \varphi(x) = -\infty$ and $\varphi' > 0$. Further, $\varphi$ is twice continuously differentiable on $]0,\varepsilon]$ and $[\varepsilon,\infty[$ such that $\varphi''(x) > 0$ for any $x > \varepsilon$.
\item $((\eta^{2}/2)\varphi'' + \zeta\varphi')(\cdot,x)$ is bounded from below by $- c_{0}\varphi_{0}(x)$, whenever $x < \varepsilon$, and by $-c_{\zeta}\varphi_{\zeta}(x)\varphi'(x)$, if $x > \varepsilon$, for all $x > 0$ with $x\neq\varepsilon$ a.e.
\end{enumerate}
\end{Lemma}

\begin{proof}
We define $\hat{\varphi}\in C^{1}([\varepsilon,\infty[)$ by $\hat{\varphi}(x):= (1/\varepsilon)\exp(\int_{\varepsilon}^{x}\varphi_{\zeta}(y)^{-1}\,dy)$. Then it follows readily that the function $\varphi:]0,\infty[\rightarrow\mathbb{R}$ given by 
$\varphi(x):=\log(x)$ for $x < \varepsilon$ and
\begin{equation*}
\varphi(x):=\log(\varepsilon) + \int_{\varepsilon}^{x}\hat{\varphi}(y)\,dy\quad\text{for $x\geq\varepsilon$}
\end{equation*}
is a feasible choice, which, however, is not of class $C^{2}$, since $\varphi_{+}''(\varepsilon) = 1/(\varphi_{\zeta}(\varepsilon)\varepsilon)$.
\end{proof}

\begin{proof}[Proof of Proposition~\ref{pr:positivity of the variance process}]
The assertion follows if we can verify that the stopping time $\sigma:=\inf\{t\in [t_{0},T]\,|\, V_{t}\leq 0\}$ satisfies $\sigma = \infty$ a.s. To this end, we choose $n_{0}\in\mathbb{N}$ with $v_{0}/n_{0} < \varepsilon < n_{0}v_{0}$ and introduce two stopping times by
\begin{equation*}
\sigma_{m}:=\inf\{t\in [t_{0},T]\,|\, V_{t}\leq v_{0}/m\} \quad\text{and}\quad \overline{\sigma}_{n}:=\inf\{t\in [t_{0},T]\,|\, V_{t} \geq nv_{0}\}
\end{equation*}
for $m,n\in\mathbb{N}$ with $m\wedge n\geq n_{0}$. Further, let $\varphi\in C^{1}(]0,\infty[)$ satisfy the conditions of Lemma~\ref{le:positivity of the variance process} and set $\varphi(x):=0$ for any $x\in ]-\infty,0]$ and $B_{\varepsilon}(\omega):=\{t\in [t_{0},T]\,|\, V_{t}(\omega)\neq\varepsilon\}$ for all $\omega\in\Omega$. Then the generalised It{\^o} formula in~\cite{Pes05} gives
\begin{equation*}
\begin{split}
&\varphi(V_{T}^{\sigma_{m,n}}) - \varphi(v_{0}) = \int_{t_{0}}^{T\wedge\sigma_{m,n}}\varphi'(V_{s})\,dV_{s} + \frac{1}{2}\int_{B_{\varepsilon}}\varphi''(V_{s})\mathbbm{1}_{[t_{0},\sigma_{m,n}]}(s)\,d\langle V\rangle_{s}\\
&\geq - k_{\varepsilon} - \varphi_{\zeta}(nv_{0})\varphi'(n v_{0})\int_{t_{0}}^{T}c_{\zeta}(s)\,ds + \int_{t_{0}}^{T\wedge\sigma_{m,n}}\varphi'(V_{s})\eta(s,V_{s})\,d\tilde{W}_{s}\quad\text{a.s.},
\end{split}
\end{equation*}
where $\sigma_{m,n}:=\sigma_{m}\wedge\overline{\sigma}_{n}$ and $k_{\varepsilon}:= \varphi_{0}(\varepsilon)\int_{t_{0}}^{T}c_{0}(s)\,ds + \varphi_{\zeta}(\varepsilon)\varphi'(\varepsilon)\int_{t_{0}}^{T}c_{\zeta}(s)\,ds
$. For any $l\in\N$ we define a stopping time by $\hat{\sigma}_{l}:=\inf\{t\in [t_{0},T]\,|\,\int_{t_{0}}^{t}\varphi'(V_{s})^{2}\eta(s,V_{s})^{2}\,ds\geq l\}$. Then
\begin{equation*}
E\big[\varphi(V_{T}^{\sigma_{m,n}})\big] = \lim_{l\uparrow\infty} E\big[\varphi(V_{T}^{\hat{\sigma}_{l}\wedge\sigma_{m,n}})\big] \geq \varphi(v_{0}) -k_{\varepsilon} - \varphi_{\zeta}(nv_{0})\varphi'(n v_{0})\int_{t_{0}}^{T}c_{\zeta}(s)\,ds
\end{equation*}
follows from dominated convergence. At the same time, as $V_{\sigma_{m}} = v_{0}/m$ a.s.~on $\{\sigma_{m} < \infty\}$ and $\lim_{x\downarrow 0} \varphi(x) = -\infty$,  we obtain that
\begin{equation*}
E\big[\varphi(V_{T}^{\sigma_{m,n}})\big] \leq \varphi(v_{0}/m)P(\sigma_{m}\leq\overline{\sigma}_{n}\wedge T) + \max_{x\in ]0,nv_{0}]}\varphi^{+}(x).
\end{equation*}
These two estimates imply that $\lim_{m\uparrow\infty} P(\sigma_{m}\leq\overline{\sigma}_{n}\wedge T) = P(\sigma\leq\overline{\sigma}_{n}\wedge T) = 0$. Hence, $\sigma >\overline{\sigma}_{n}\wedge T$ a.s.~and the desired result follows from the fact that $\sup_{n\in\N}\overline{\sigma}_{n} = \infty$.
\end{proof}

\begin{proof}[Proof of Proposition~\ref{pr:transformed SDE}]
(i) The drift and diffusion coefficients of the SDE~\eqref{eq:transformed stochastic volatility model} are independent of the first spatial coordinate. For this reason, the claim follows directly from Corollary 3.9, Remark 3.10 and Proposition 3.13 in~\cite{KalMeyPro21}, which yield pathwise uniqueness for the second SDE in~\eqref{eq:stochastic volatility model}.

(ii) From Theorem 3.27 in~\cite{KalMeyPro21} we know that there is a unique strong solution $V^{t_{0},v_{0}}$ to the second SDE in~\eqref{eq:stochastic volatility model} such that $V_{t_{0}}^{t_{0},v_{0}} = v_{0}$ a.s. According to Proposition~\ref{pr:positivity of the variance process}, the positivity of $V^{t_{0},v_{0}}(\omega)$ for $P$-a.e.~$\omega\in\Omega$ holds and, by taking a continuous modification if necessary, we may assume that $V^{t_{0},v_{0}}(\omega) > 0$ for all $\omega\in\Omega$.

Furthermore, we may choose a continuous process $X^{t_{0},x_{0},v_{0}}$ that is adapted to the augmented natural filtration of the two-dimensional $(\mathscr{F}_{t})_{t\in [0,T]}$-Brownian motion $(W,\tilde{W})$ and satisfies~\eqref{eq:log-price dynamics} for $V = V^{t_{0},v_{0}}$. This argumentation justifies the assertion.

(iii) For any sequence $(s_{n},x_{n},v_{n})_{n\in\N}$ and each point $(s,x,v)$ in $[0,t]\times\R\times ]0,\infty[$ such that $s\leq s_{n}$ for all $n\in\N$ and $\lim_{n\uparrow\infty} (s_{n},x_{n},v_{n}) = (s,x,v)$, Theorem 21.4 and Lemma~21.9 in~\cite{Bau00} yield
\begin{equation*}
\lim_{n\uparrow\infty} E\big[\varphi(s_{n},X_{t}^{s_{n},x_{n},v_{n}},V_{t}^{s_{n},v_{n}})\big] = E\big[\varphi(s,X_{t}^{s,x,v},V_{t}^{s,v})\big]
\end{equation*}
if $(X_{t}^{s_{n},x_{n},v_{n}},V_{t}^{s_{n},v_{n}})_{n\in\N}$ converges in probability to $(X_{t}^{s,x,v},V_{t}^{s,v})$. By using Lemma~\ref{le:log-price comparison estimate}, this can be inferred from Theorem 4.6 in~\cite{KalMeyPro21}, which gives a first moment estimate for random It{\^o} processes and implies the two estimates~\eqref{eq:variance comparison estimate} and~\eqref{eq:variance growth estimate}. 

(iv) According to Lemma 3.5 in~\cite{Kal17}, for instance, the Markov property of the triple $((\hat{X},\hat{V}),(\hat{\mathscr{F}}_{t})_{t\in [0,T]},\mathbb{P})$ holds if we can show that
\begin{equation*}
E_{s,x,v}\big[\varphi(\hat{X}_{t},\hat{V}_{t})\big|\hat{\mathscr{F}}_{\tilde{s}}\big] = E_{\tilde{s},\hat{X}_{\tilde{s}},\hat{V}_{\tilde{s}}}\big[\varphi(\hat{X}_{t},\hat{V}_{t})\big]\quad\text{$P_{s,x,v}$-a.s.}
\end{equation*}
for any $s,\tilde{s},t\in [0,T]$ with $s\leq\tilde{s}\leq t$, each $(x,v)\in\R\times ]0,\infty[$ and every Lipschitz continuous function $\varphi:\R\times ]0,\infty[\rightarrow [0,1]$. This follows from the same reasoning as in the proof of Theorem 5.1.5 in~\cite{StrVar72}.

Finally, as Lemma 3.14 in~\cite{Kal17} shows that any Markov process with right-continuous paths  that is right-hand Feller is also strongly Markov, the proof is complete.
\end{proof}

\subsection{Proofs for the valuation function as mild and classical solution}

\begin{proof}[Proof of Theorem~\ref{th:pre-default value process}]
(i) Let $(s,x,v)\in [0,T]\times\R\times ]0,\infty[$ and $(X,V)$ be a solution to~\eqref{eq:transformed stochastic volatility model} on $[s,T]$ with $(X^{s},V^{s}) = (x,v)$ a.s., in which case it also solves~\eqref{eq:transformed SDE 2} under $\tilde{P}$. In particular, $\tilde{P}_{s,x,v}$ must be its law under $\tilde{P}$. Hence,~\eqref{eq:three terms} and~\eqref{eq:4.13} hold and our discussion preceding~\eqref{eq:inhomogeneity function} shows that~\eqref{con:m.1}-\eqref{con:m.4} are valid. For this reason, we may consider solutions to~\eqref{eq:pre-default valuation}.

Since $\alpha$, $\beta$, $\hat{H}$ and $u$ are continuous, so are the collateral process $\alpha\mathscr{V}$, the pre-default funding process $(1-\alpha)\mathscr{V}$, the pre-default hedging process $\hat{H}(\cdot,\exp(X),V,\mathscr{V})$ and the close-out value $\beta\mathscr{V}$, which justifies that~\eqref{con:c.1} holds.

By Remark~\ref{re:integrability conditions}, the conditions~\eqref{con:c.2} and~\eqref{con:c.3} follow from~\eqref{eq:pre-default condition 2} and the boundedness of $u$. For this reason, Propositions~\ref{pr:martingale characterisation} and~\ref{pr:backward stochastic representation} entail that $\mathscr{V}$ is an $(\mathscr{F}_{t})_{t\in [s,T]}$-semimartingale and the process $M:[s,T]\times\Omega\rightarrow\R$ given by
\begin{align*}
M_{t} &:= D_{0,t}(r)\mathscr{V}_{t}G_{t}(\tau) + \int_{0}^{t}D_{0,\hat{s}}(r)G_{\hat{s}}(\tau)\big(\hat{B}(\hat{s}, \exp(X_{\hat{s}}),V_{\hat{s}},\mathscr{V}_{\hat{s}}) + (\hat{r} - g_{\mathcal{I}} - g_{\mathcal{C}})(\hat{s})\mathscr{V}_{\hat{s}}\big)\,d\hat{s}
\end{align*}
is a continuous $(\mathscr{F}_{t})_{t\in [s,T]}$-martingale under $\tilde{P}$ such that
\begin{equation}\label{eq:specific backward integral representation}
\mathscr{V}_{t} = \phi(\exp(X_{T}),V_{T}) + \int_{t}^{T}\hat{B}(\hat{t},\exp(X_{\hat{t}}),V_{\hat{t}},\mathscr{V}_{\hat{t}})\,d\hat{t} - \int_{t}^{T}\frac{D_{0,\hat{t}}(-r)}{G_{\hat{t}}(\tau)}\,dM_{\hat{t}}\quad\text{a.s.}
\end{equation}
for all $t\in [s,T]$ a.s. Thereby we used the definitions of the two random functionals $A$, $\B$ and the function $\hat{B}$ in~\eqref{eq:finite variation process},~\eqref{eq:inhomogeneity process} and~\eqref{eq:inhomogeneity function} to find the relation
\begin{equation}\label{eq:representation of A}
\dot{A}_{t}(Y) = G_{t}(\tau)\big(\hat{B}(t,\exp(X_{t}),V_{t},Y) + (\hat{r} - g_{\mathcal{I}} - g_{\mathcal{C}})(t)Y\big)
\end{equation}
for a.e.~$t\in [0,T]$ a.s.~for any continuous $Y\in\mathscr{S}$. Because $\phi(\exp(X_{T}),V_{T})$ is bounded and $\int_{s}^{T}|\hat{B}(t,\exp(X_{t}),V_{t},\mathscr{V}_{t})|\,dt$ is integrable, we obtain that
\begin{equation*}
\tilde{E}\bigg[\sup_{t\in [s,T]} \bigg|\int_{s}^{t}\frac{D_{0,\hat{s}}(-r)}{G_{\hat{s}}(\tau)}\,dM_{\hat{s}}\bigg|\bigg] < \infty.
\end{equation*}
In particular, the $(\mathscr{F}_{t})_{t\in [s,T]}$-local martingale $\int_{s}^{\cdot}D_{0,\hat{s}}(-r)G_{\hat{s}}(\tau)^{-1}\,dM_{\hat{s}}$ is of class (DL) and therefore, it must be a martingale. Now we may take expectations in~\eqref{eq:specific backward integral representation} to see that $u$ satisfies~\eqref{eq:mild solution}. 

(ii) We already know that $(X,V)$ solves~\eqref{eq:transformed SDE 2} on $[s,T]$ under $\tilde{P}$. Thus, the condition on the filtration entails that the process $N:[s,T]\times\Omega\rightarrow\mathbb{R}$ defined via
\begin{equation*}
N_{t} := \mathscr{V}_{t} + \int_{s}^{t}\hat{B}(\hat{s},\exp(X_{\hat{s}}),V_{\hat{s}},\mathscr{V}_{\hat{s}})\,d\hat{s}
\end{equation*}
is a continuous $(\mathscr{F}_{t})_{t\in [s,T]}$-martingale under $\tilde{P}$. Indeed, its integrability follows directly from the definition of a mild solution. The martingale property $\tilde{E}[N_{t}|\mathscr{F}_{\tilde{s}}] = N_{\tilde{s}}$ a.s.~for any $\tilde{s},t\in [s,T]$ with $\tilde{s}\leq t$ is implied by
\begin{equation*}
\tilde{E}\bigg[\mathscr{V}_{t} + \int_{\tilde{s}}^{t}\hat{B}(\hat{s},\exp(X_{\hat{s}}),V_{\hat{s}},\mathscr{V}_{\hat{s}})\,d\hat{s}\,\bigg|\,\mathscr{F}_{\tilde{s}}\bigg] = \mathscr{V}_{\tilde{s}}\quad\text{a.s.}
\end{equation*}
This identity is a consequence of Proposition~\ref{pr:transformed SDE}, by extending the Markov property of the diffusion with measure theoretical methods. See Proposition 3.7 in~\cite{Kal17}[Chapter 3], for example. In particular, $\mathscr{V}$ is an $(\mathscr{F}_{t})_{t\in [s,T]}$-semimartingale.

Hence, for any continuous $(\mathscr{F}_{t})_{t\in [s,T]}$-local martingale $M$ that is indistinguishable from the stochastic integral $\int_{s}^{\cdot}D_{0,\hat{s}}(r)G_{\hat{s}}(\tau)\,dN_{\hat{s}}$ the representation~\eqref{eq:specific backward integral representation} holds. Now we obtain that
\begin{equation*}
M_{t} + D_{0,s}(r)u(s,x,v)G_{s}(\tau) =  D_{0,t}(r)\mathscr{V}_{t}G_{t}(\tau) + \int_{s}^{t}D_{0,\hat{s}}(r)\,dA_{\hat{s}}(\mathscr{V})
\end{equation*}
for each $t\in [s,T]$ a.s. by the uniqueness assertion of Proposition~\ref{pr:backward stochastic representation}. The boundedness of $u$ and the integrability of the random variable $\int_{s}^{T}D_{0,\hat{s}}(r)|\dot{A}_{\hat{s}}(\mathscr{V})|\,d\hat{s}$ imply that
\begin{equation*}
\tilde{E}\bigg[\sup_{t\in [s,T]}\bigg|D_{0,t}(r)\mathscr{V}_{t}G(\tau) + \int_{0}^{t}D_{0,\hat{s}}(r)\,dA_{\hat{s}}(\mathscr{V})\bigg|\bigg] < \infty.
\end{equation*}
Thus, $D_{0,\cdot}(r)\mathscr{V}G(\tau) + \int_{0}^{\cdot}D_{0,\hat{s}}(r)\,dA_{\hat{s}}(\mathscr{V})$ is a continuous $(\mathscr{F}_{t})_{t\in [s,T]}$-local martingale that is  of class (DL). Therefore, it is a martingale and we may invoke Proposition~\ref{pr:martingale characterisation} to complete the proof.
\end{proof}

\begin{proof}[Proof of Proposition~\ref{pr:mild solution}]
By~\eqref{con:p.1},~\eqref{con:p.2} and~\eqref{con:p.4}, all the rates $\hat{r}$, $\hat{c}_{+},\hat{c}_{-}$, $\hat{f}_{+},\hat{f}_{-}$, $\hat{h}_{+}$, $\hat{h}_{-}$ and the functions $g_{\mathcal{I}}$, $g_{\mathcal{C}}$, $\alpha$, $\beta$ are integrable. Thus,~\eqref{con:p.5} ensures that the affine growth condition~\eqref{eq:affine growth condition} and the Lipschitz condition~\eqref{eq:Lipschitz condition on compact sets} on compacts sets holds for $\hat{B}$.

Moreover, we note that $\hat{B}(t,\cdot,\cdot,\cdot)$ is continuous for a.e.~$t\in [0,T]$, since $\hat{\pi}$ and $\hat{H}$ are continuous, according to~\eqref{con:p.3} and~\eqref{con:p.4}. Hence, as we know that Proposition~\ref{pr:transformed SDE} is applicable, all requirements of Theorem~2.15 in~\cite{Kal20} are met and the assertions follow.
\end{proof}

\begin{proof}[Proof of Corollary~\ref{co:classical solution}]
By Proposition~\ref{pr:transformed SDE} and our discussion preceding the corollary, the assumptions~(2.2),~(2.3) and (2.9)-(2.13) in~\cite{BecSch05} hold. Further, from~\eqref{con:p.5} and~\eqref{con:p.6} we infer that $\hat{B}$ is of affine growth and locally Lipschitz continuous in $y\in\mathbb{R}$, uniformly in $(t,s,v)\in [0,T]\times ]0,\infty[^{2}$, and it is also locally H\oe lder continuous. This shows that the hypotheses~(2.19),~(2.20) and (2.21) in~\cite{BecSch05} are satisfied and the assertion follows from Theorem 2.4 in this reference.
\end{proof}

\let\OLDthebibliography\thebibliography
\renewcommand\thebibliography[1]{
\OLDthebibliography{#1}
\setlength{\parskip}{0pt}
\setlength{\itemsep}{2pt}}

\end{document}